\documentclass[11pt, a4paper, twoside, reqno]{amsart}
\usepackage[utf8]{inputenc}
\usepackage[T1]{fontenc}
\usepackage{lmodern}
\usepackage{microtype}
\usepackage[frak=boondox]{mathalpha}
\usepackage{dsfont}
\usepackage{bbold}
\usepackage{upgreek}
\usepackage{mathrsfs}
\usepackage{euscript}
\usepackage{stmaryrd}
\usepackage{amssymb}
\usepackage{mathabx}
\usepackage[top=1in, bottom=1in, left=1in, right=1in]{geometry}
\usepackage{enumitem}
\usepackage{parskip}
\usepackage{float}
\usepackage{caption}
\usepackage{tabto}
\usepackage{color}
\usepackage{mdframed}
\usepackage{hyphenat}
\usepackage{comment}
\usepackage{setspace}
\usepackage{etoolbox}
\usepackage{amsmath, amssymb, amsthm, mathtools}
\usepackage{extarrows}
\usepackage{tikz-cd}
\usepackage[all,cmtip]{xy} 
\usepackage{graphicx}

\usepackage[numbers]{natbib}
\setcitestyle{open={},close={}}
\makeatletter
\renewcommand{\@biblabel}[1]{[#1]\hfill}
\makeatother

\usepackage[hidelinks]{hyperref}
\usepackage[nameinlink]{cleveref}
\newtheoremstyle{mystyle}
  {\topsep}    
  {\topsep}   
  {\itshape}   
  {}           
  {\bfseries}   
  {.}          
  {.5em}        
  {}            

\newtheoremstyle{spacedremark} 
  {\topsep}    
  {\topsep}     
  {\normalfont} 
  {}           
  {\bfseries}   
  {.}          
  {.5em}        
  {}               

\newtheorem{thm}{Theorem}[subsection]
\crefname{thm}{Theorem}{Theorems}

\newtheorem{lem}[thm]{Lemma}
\crefname{lem}{Lemma}{Lemmas}

\crefname{Q}{Question}{Questions}

\newtheorem{prop}[thm]{Proposition}
\crefname{prop}{Proposition}{Propositions}

\newtheorem{cor}[thm]{Corollary}
\crefname{cor}{Corollary}{Corollaries}

\crefname{conj}{Conjecture}{Conjectures}

\newtheorem{defn}[thm]{Definition}
\crefname{defn}{Definition}{Definitions}

\newtheorem{inner-recall-star}{Theorem}
\crefname{inner-recall-star}{Theorem}{Theorems}
\newenvironment{recall*}[1]
  {\begin{inner-recall-star}[see \Cref{#1}]}
  {\end{inner-recall-star}}

\theoremstyle{spacedremark}

\newtheorem{rem}[thm]{Remark}
\crefname{rem}{Remark}{Remarks}
\AtEndEnvironment{rem}{\hfill\qedsymbol}

\newtheorem{con}[thm]{Construction}
\crefname{con}{Construction}{Constructions}

\newtheorem{coex}[thm]{Counterexample}
\crefname{coex}{Counterexample}{Counterexamples}
\AtEndEnvironment{coex}{\hfill\qedsymbol}
\newtheorem{coex*}[thm]{Counterexamples}
\crefname{coex*}{Counterexample}{Counterexamples}

\crefname{obs}{Observation}{Observations}
\newtheorem{con-rem}[thm]{Concluding Remark}
\crefname{con-rem}{Concluding Remark}{Concluding Remarks}

\newtheorem{ex}[thm]{Example}
\crefname{ex}{Example}{Examples}

\newtheorem*{thm*}{Theorem}

\DeclareMathOperator*{\colim}{colim}

\newcommand{\plim}{\varprojlim}

\DeclareMathAlphabet{\duc}{U}{dutchcal}{m}{n}
\SetMathAlphabet{\duc}{bold}{U}{dutchcal}{b}{n}
\DeclareFontFamily{U}{BOONDOX-calo}{\skewchar\font=45 }
\DeclareFontShape{U}{BOONDOX-calo}{m}{n}{<-> s*[1.0] BOONDOX-r-calo}{}
\DeclareFontShape{U}{BOONDOX-calo}{b}{n}{<-> s*[1.0] BOONDOX-b-calo}{}
\DeclareMathAlphabet{\mal}{U}{BOONDOX-calo}{m}{n}
\SetMathAlphabet{\mal}{bold}{U}{BOONDOX-calo}{b}{n}
\newcommand*{\mcap}{\mathbin{\scalebox{1.1}{\ensuremath{\bigcap}}}}

\newcommand{\shrp}{{\mathord{\text{\tiny $\sharp$}}}}
\newcommand{\esc}{\EuScript}

\setlist[enumerate]{leftmargin=*, nosep}
\begin{document}

\title[On the internal homotopy theory of the motivic homotopy category]{On the internal homotopy theory of\\ the motivic homotopy category}

\author[Dipankar Maity]{Dipankar Maity}
\address{Department of Mathematical Sciences \\
Indian Institute of Science Education and Research (IISER) Mohali \\ Knowledge city, Sector 81, SAS Nagar, Manauli PO 140306 \\
India}
\email{homotopytypes@gmail.com}

\begin{abstract}
We develop an $\infty$-categorical framework for internal homotopy theory (truncations, homotopy groups, and Eilenberg-MacLane objects) of an accessible localization of an $\infty$-topos. For the motivic localization over a perfect field $k$, we show that a motivic space has trivial motivic $n$-truncation iff it has trivial Nisnevich $n$-truncation ($n\geq-2$), identifying internal $n$-Eilenberg-MacLane objects with strongly $\mathbb{A}^1$-invariant sheaves of groups when $n=1$ and strictly $\mathbb{A}^1$-invariant sheaves of abelian groups when $n\geq2$. We further prove that a generically free Nisnevich group sheaf is strongly $\mathbb{A}^1$-invariant iff it is birational (which for the free group sheaf $L_{nis}F_{grp}A$ holds iff $A$ is birational), whereas for an \'etale-weakly unramified sheaf of sets $A$, $L_{nis}F_{grp}A$ is $\mathbb{A}^1$-invariant iff $A$ is itself $\mathbb{A}^1$-invariant. Non-birational, \'etale-weakly unramified motivic sheaves of sets (such as $\mathbb{G}_m$) thus generate (non-abelian) $\mathbb{A}^1$-invariant group sheaves that are not strongly $\mathbb{A}^1$-invariant. Combining this obstruction with essentially smooth base change, we prove that for a qcqs scheme $S$, $\mathcal{H}^{mot}(S)$ is an $\infty$-topos iff $S=\emptyset$.

We also study relative truncation and connectivity under a localization. Over a qcqs scheme $S$, this identifies external $\mathbb{A}^1$-$n$-coverings with internal motivic $n$-coverings via pullback along the localization unit, yielding a motivic $n$-Galois correspondence over $k$. For connective morphisms over $k$, we establish a higher $\mathbb{A}^1$-Van Kampen theorem and a Nisnevich-local-to-global $\mathbb{A}^1$-$n$-connectivity criterion; in particular, a $k$-scheme is $\mathbb{A}^1$-connected iff it admits a Nisnevich cover by $\mathbb{A}^1$-connected schemes whose pairwise fiber products have $k$-points.
\end{abstract}
\maketitle

\tableofcontents

\section{Introduction}

\subsection{The philosophy and context of the article}

Two of the central geometric pillars of modern mathematics, differential geometry and algebraic geometry, are traditionally presented from opposite foundational directions. Due to the foundational work of Grothendieck, modern algebraic geometry enjoys a transparent, \textit{synthetic} (or constructive) presentation as the geometry associated to the algebraic theory of commutative $k$-algebras, generated by the affine spaces $\{\mathbb{A}^n_k\}_{n \geq 0}$. By contrast, differential geometry is classically introduced \textit{deductively}: one first posits the rigid $1$-category of topological spaces, and subsequently carves out smooth manifolds as topological spaces equipped with extra structure.

From a modern perspective, however, differential geometry admits an entirely parallel constructive formulation via the theory of $C^\infty$-rings and $C^\infty$-schemes (see, e.g., Moerdijk-Reyes [\cite{MR1083355}] or Joyce [\cite{MR3987298}]). Just as commutative $k$-algebras are models for the Lawvere theory of polynomial maps between affine spaces $\{\mathbb{A}^n_k\}_{n \geq 0}$, $C^\infty$-rings are models for the Lawvere theory of smooth maps between Euclidean spaces $\{\mathbb{R}^n\}_{n \geq 0}$. Passing to sheaves on the category of affine $C^\infty$-schemes with respect to the open cover topology, one recovers the category $\mathrm{Man}$ of smooth manifolds inside separated $C^\infty$-schemes as those finitely presented objects satisfying formal smoothness (i.e., lifting along infinitesimal deformations). In this way, smooth manifolds are built directly from smooth analysis on $\mathbb{R}^n$, entirely mirroring how smooth $k$-schemes are built from polynomial algebra on $\mathbb{A}^n_k$.

From this constructive vantage point, a natural inverse question arises: how can one recover the world of topological spaces from smooth manifolds alone? Recovering the rigid $1$-category $\mathrm{Top}$ is hopeless. Since $\mathrm{Top}$ is cocomplete, one would first have to cocomplete $\mathrm{Man}$ while preserving the topological data of each manifold, e.g.\ via $\mathrm{Shv}_{\mathrm{open}}(\mathrm{Man})$; however, because smooth manifolds are triangulable [\cite{MR2545}], gluing them along colimits produces only triangulable spaces, whereas general spaces such as the Cantor set are not even homotopy equivalent to such spaces. This obstruction disappears in weak homotopy theory, since every topological space is weakly equivalent to a CW complex [\cite{MR1867354}, Proposition 4.13]. This insight also repairs another serious problem: namely, the faithful functor $\mathrm{Man}\to \mathrm{Top}$ becomes full up to homotopy, because every continuous map between smooth manifolds is homotopic to a smooth map [\cite{MR1930091}, Theorem 6.26]. However, preserving faithfulness then requires imposing $\mathbb{R}^1$-invariance, which collapses the $1$-category $\mathrm{Shv}_{\mathrm{open}}(\mathrm{Man})$ to $\mathrm{Set}$. To ensure essential surjectivity, one must thus add homotopy colimits. A combination of homotopy colimits, open cover gluing, and $\mathbb{R}^1$-invariant does solve the problem.  Indeed, the $\infty$-category of spaces $\mathrm{Top}[W^{-1}]$, with $W$ the weak homotopy equivalences, is the universal presentable $\infty$-category generated by $\mathrm{Man}$ subject to open-cover descent and contracting the Lawvere generator $\mathbb{R}^1$  [\cite{dugger1999sheaves}, Remark 3.4.10].

Transporting this exact blueprint back to algebraic geometry yields a canonical candidate for a theory of ``homotopy types of schemes,'' namely the motivic homotopy category of Morel and Voevodsky [\cite{morel19991}]. From the perspective of universal homotopy theory, the motivic homotopy category $\mathcal{H}^{mot}(k)$ is constructed from the category $Sm_k$ of smooth $k$-schemes as the universal presentable $\infty$-category satisfying Nisnevich descent and homotopy invariance with respect to the Lawvere generator $\mathbb{A}^1_k$ [\cite{robalo2012noncommutative}, Theorem 5.2].

Here, however, the analogy encounters a profound structural divergence. Whereas the universal homotopy theory of smooth manifolds yields an $\infty$-topos (namely, $\mathrm{Top}[W^{-1}]\simeq \mathrm{Spc}$), which greatly simplifies homotopical calculations, the motivic homotopy category $\mathcal{H}^{mot}(k)$ is far from being an $\infty$-topos. Geometrically, this divergence stems from the absence of ``good open covers'' (as used in the proof at [\cite{dugger1999sheaves}]) in algebraic geometry: every smooth manifold admits an open cover whose entire \v{C}ech nerve consists of (disjoint unions of) contractible Euclidean spaces $\mathbb{R}^n$, forcing the $\mathbb{R}^1$-localization of the sheaf $\infty$-topos to collapse to $\mathrm{Spc}$. By contrast, while a smooth $k$-scheme admits covers by schemes \'etale over the affine spaces $\mathbb{A}^n_k$, one cannot ensure the (\'etale-local) $\mathbb{A}^1$-contractibility of the higher intersections in the \v{C}ech nerve, so that even $\mathcal{H}^{\mathbb{A}^1}_{\acute{e}t}(k)$ fails to be an $\infty$-topos, let alone $\mathcal{H}^{\mathbb{A}^1}_{nis}(k)=:\mathcal{H}^{mot}(k)$.

On a technical footing, while the $\infty$-category of Nisnevich sheaves of spaces $\mathcal{P}_{nis}(k) \subset \mathcal{P}(Sm_k)$ is a left-exact localization and hence an $\infty$-topos [\cite{hoyois2017six}, Corollary 3.9], the subsequent $\mathbb{A}^1$-localization $L_{mot} \colon \mathcal{P}_{nis}(k) \to \mathcal{H}^{mot}(k)$ is not left-exact. Closer inspection against Lurie's $\infty$-categorical Giraud axioms [\cite{lurie2009higher}, Proposition 6.1.5.3] reveals the exact point of failure: over a field, $\mathcal{H}^{mot}(k)$ is presentable, has universal colimits, and possesses disjoint coproducts (\Cref{motivic giraud's theorem}), but its groupoid objects fail to be effective. 

However, demonstrating this failure via an explicit counterexample is nontrivial and was a primary motivation for this paper (see \S3.2). Over a perfect field $k$, we show that a Nisnevich-local free group sheaf generated by a motivic sheaf of sets is strongly $\mathbb{A}^1$-invariant (i.e., admits an effective delooping) if and only if it (or equivalently, its generating sheaf) is birational, whereas it remains $\mathbb{A}^1$-invariant whenever the generating sheaf is `etale-weakly unramified' (\Cref{tau weakly unramified}, \Cref{intro:free-sheaf-bir}). Consequently, in stark contrast to topology, where the free group on any discrete space (i.e., a set) is the fundamental group of a wedge of circles, free group sheaves on discrete motivic spaces (i.e., those in the image of the left adjoint to the forgetful functor $L_{\mathbb{A}^1}\mathrm{Grp}_k \to Sh_k^{\mathbb{A}^1}$) need not arise as fundamental groups of motivic spaces. This obstruction formally proves that $\mathcal{H}^{mot}(k)$ is not an $\infty$-topos (\Cref{mot not topos for perf}). Moreover, by repeatedly applying essentially smooth base change arguments, we strengthen this result to arbitrary non-empty qcqs schemes (\Cref{intro:not-topos}).

On a side note, as the preceding discussion indicates, ensuring the existence of effective deloopings (and thus extracting an $\infty$-topos from the motivic homotopy category) requires passing to birational localization. Heuristically, identifying every scheme with its generic points repairs the lack of `very' local homotopical contractibility caused by the absence of good covers. The resulting birational motivic homotopy category $\mathcal{H}^{b\mathbb{A}^1}(S):=L_{dense}\mathcal{H}^{mot}(S)$ was studied in detail in a separate work of the author [\cite{0bat}], where we proved that $\mathcal{H}^{b\mathbb{A}^1}(S)$ is indeed an $\infty$-topos ([\cite{0bat}, Theorem 3.1.13]).

Although $\mathcal{H}^{b\mathbb{A}^1}(k)$ is an $\infty$-topos, it retains very little cohomological information (specifically, only the weight $0$ part of motivic homotopy theory). Capturing meaningful but computable motivic invariants therefore requires identifying which essential topos-theoretic structures survive within $\mathcal{H}^{mot}(k)$ itself. Because $\mathcal{H}^{mot}(k)$ embeds into the ambient $\infty$-topos $\mathcal{P}_{nis}(k)$, its homotopy theory has historically been studied \textit{extrinsically}: following Morel [\cite{morel2012a1}], one typically defines the $\mathbb{A}^1$-homotopy sheaves $\pi_i^{\mathbb{A}^1}(\esc{X}) := \pi_i^{nis}(L_{mot}\esc{X})$, Postnikov truncations $\tau_{\leq n}^{nis}(L_{mot}\esc{X})$, and $\mathbb{A}^1$-coverings by borrowing the topos-theoretic machinery of $\mathcal{P}_{nis}(k)$. Yet as a presentable $\infty$-category in its own right, $\mathcal{H}^{mot}(k)$ also possesses an \textit{internal} homotopy theory governed by its intrinsic truncation functors $\tau_{\leq n}^{mot} \colon \mathcal{H}^{mot}(k) \to \mathcal{H}^{mot}_{\leq n}(k)$, internal loop spaces $\Omega_{\mathcal{H}^{mot}}$, and internal homotopy groups $\pi_i^{mot}(\esc{X}) := \tau_{\leq 0}^{mot}\Omega^i \esc{X}$.

For $i = 0$, the internal connected components functor $\pi_0^{mot}$ is formally the left adjoint to the inclusion of $\mathbb{A}^1$-invariant Nisnevich sheaves $Sh_{nis}^{\mathbb{A}^1}(k) \hookrightarrow \mathcal{H}^{mot}(k)$. For a motivic space $\esc{X} \in \mathcal{H}^{mot}(k)$, the canonical comparison map
$$\pi_0^{\mathbb{A}^1}(\esc{X}) = \pi_0^{nis}(\esc{X}) \longrightarrow \pi_0^{mot}(\esc{X})$$
was famously conjectured by Morel [\cite{morel2012a1}, Conjecture 1.12] to be an isomorphism, a statement recently disproven by Ayoub [\cite{ayoubcounterexamples}]. Consequently, the ambient functor $\pi_0^{\mathbb{A}^1}$ fails to preserve motivic colimits, making the intrinsic functor $\pi_0^{mot}$ indispensable for internal categorical arguments (for example, local-to-global connectivity).

Nevertheless, the situation remains remarkably well-behaved. Despite the failure of Morel's conjecture and the non-topos nature of $\mathcal{H}^{mot}(k)$, the internal and external homotopy theories of $\mathcal{H}^{mot}$ remain tightly locked together. Specifically, we show that:
\begin{enumerate}
    \item over a perfect field, a motivic space $\esc{X}$ satisfies $\tau_{\leq n}^{nis}\esc{X}\simeq *$ if and only if $\tau_{\leq n}^{mot}\esc{X}\simeq *$ for all $n\geq -2$ (\Cref{intro:path-inj}), yielding an identification of internal $n$-Eilenberg-MacLane objects with $n$-strictly $\mathbb{A}^1$-invariant $n$-group sheaves (\Cref{motivic em recall}) and an internal looping-delooping equivalence (\Cref{intro:mot-bar-omega});
    \item over an arbitrary qcqs base scheme, internal $n$-truncated morphisms (motivic $n$-coverings) over $L_{mot}\esc{X}$ are equivalent via pullback to external $\mathbb{A}^1$-$n$-coverings (which is Morel's $\mathbb{A}^1$-coverings when $n=0$) over any Nisnevich-local space $\esc{X}$ (\Cref{intro:cov-equiv});
    \item over an arbitrary qcqs base scheme, motivic $0$-connective morphisms coincide with motivic effective epimorphisms, and $L_{mot}$ preserves effective epimorphisms and detects Nisnevich effective epimorphisms over motivic spaces (\Cref{Lmot preserves and detects eff epi}).
\end{enumerate}
Moreover, these bridges allow us to systematically use the internal framework of $\mathcal{H}^{mot}$ to streamline several external results, for example:
\begin{enumerate}[label=(\alph*)]
    \item over a perfect field, combining (1) and (2) yields universal $n$-coverings and a full-fledged $n$-th motivic Galois correspondence (\Cref{intro:galois-corr});
    \item combining (1) and (3) yields a higher $\mathbb{A}^1$-Van Kampen theorem (\Cref{intro:van-kampen}) alongside a Nisnevich-local-to-global $\mathbb{A}^1$-$n$-connectivity criterion for schemes (\Cref{intro:local-global-conn}).
    \item combining (2) and (3) gives an $\mathbb{A}^1$-image factorization, whereby every morphism of Nisnevich-local spaces factors as a motivic effective epimorphism followed by an $\mathbb{A}^1$-monomorphism (\Cref{motivic n connective factorization});
\end{enumerate}
\subsection{A quick summary of the article}

We develop our results in two stages: first, we establish a general $\infty$-categorical framework for internal homotopy theory under an accessible localization $L : \mathcal{C} \to L\mathcal{C}$ of an $\infty$-topos (\S2), and subsequently we apply this formalism to $L_{mot} : \mathcal{P}_{nis}(S) \to \mathcal{H}^{mot}(S)$ (\S3). While the contents of \S2 are categorically significant and are directly used in [\cite{0bat}] to deduce the birational analog, in this summary we present them from the motivic perspective to provide motivation from the motivic perspective.

\subsubsection*{1. Motivic truncations and Motivic Eilenberg-MacLane objects}
Over an arbitrary qcqs base scheme $S$, internal motivic $n$-truncation is computed by the sequential colimit $L_{\leq n}^{mot} \simeq \colim_m (\tau_{\leq n}^{nis} L_{mot})^{\circ m}$ (\Cref{formula for motivic truncation}, \Cref{remark on truncation of localization}). For $n=0$, this can explicitly be realized as the iterated $\mathbb{A}^1$-chain connected components construction $\mathcal{L} = \colim_m \mathcal{S}^{\circ m}$ (where $\mathcal{S} = \pi_0^{ch,\mathbb{A}^1}$) of Balwe-Hogadi-Sawant [\cite{amit}], and one has $\pi_0^{mot}\simeq \mathcal{L}\pi_0^{nis}$ (\Cref{piimot is Lpiinis}). Over a field $k$, an object $\esc{X} \in \mathcal{H}^{mot}(k)$ is called \emph{motivically connected} if $\pi_0^{mot}\esc{X} \simeq *$, whereas it is \emph{$\mathbb{A}^1$-connected} in the classical sense if $\pi_0^{\mathbb{A}^1}\esc{X} \simeq *$. Building on the ghost homotopy techniques of Balwe-Sawant [\cite{MR4404880}] on $\pi_0^{mot}\simeq \mathcal{L}\pi_0^{nis}$ and using Choudhury-Roy [\cite{choudhury2022characterisation}] and Bachmann [\cite{bachmann2024strongly}], our first key structural result shows that these notions coincide over a perfect field:

\begin{inner-recall-star}[see \Cref{motivic 0 truncation is injective}, \Cref{mot conn if and only if trunc trivial}]\label{intro:path-inj}
Let $k$ be a perfect field, $\esc{X} \in \mathcal{H}^{mot}(k)$, and $n \geq -2$. Then $\tau_{\leq n}^{mot}\esc{X} \simeq *$ if and only if $\tau_{\leq n}^{nis}\esc{X} \simeq *$. If $n \geq 0$ and $\esc{X}$ is pointed at $x : * \to \esc{X}$, these conditions are further equivalent to the vanishing of $\pi_i^{mot}(\esc{X},x) := \tau_{\leq 0}^{mot}\Omega^i(\esc{X},x)$, or equivalently the vanishing of $\pi_i^{\mathbb{A}^1}(\esc{X},x)$, for all $0 \leq i \leq n$.
\end{inner-recall-star}

For $n\geq 0$, we denote the pointed subcategory appearing in the theorem above by $\mathcal{H}^{mot}(k)_{\geq n}$. It follows that the full subcategory of internal $n$-Eilenberg-MacLane objects, $$\mathcal{EM}_n(\mathcal{H}^{mot}(k)) := \mathcal{H}^{mot}_{\leq n}(k)_* \mcap \mathcal{H}^{mot}(k)_{\geq n},$$ is equivalent to $\mathcal{H}^{mot}(k) \mcap \mathcal{EM}_n(\mathcal{P}_{nis}(k))$(\Cref{motivic em vs nis em}). Combined with Morel's theorem that strongly $\mathbb{A}^1$-invariant abelian sheaves are strictly $\mathbb{A}^1$-invariant over perfect fields [\cite{morel2012a1}, Theorem 4.46], our categorical machinery yields a complete internal characterization:

\begin{inner-recall-star}[\Cref{pi_i mot=nis}, \Cref{motivic em}]\label{motivic em recall}
Let $k$ be a perfect field. For all $i \geq 1$ and every pointed motivic space $\esc{X} \in \mathcal{H}^{mot}(k)_*$, there are equivalences $\pi_i^{mot}\esc{X} \simeq \pi_i^{nis}\esc{X} \simeq \pi_i^{\mathbb{A}^1}\esc{X}$. Moreover, the internal bar-homotopy-group adjunctions (\Cref{internal bar pi adjunct}) restrict to equivalences of categories
\begin{flalign*}
    \mathrm{B}_{mot}^n \simeq \mathrm{B}^n_{nis} : n\mbox{-}\mathrm{Grp}_k^{L_{mot}} &\simeq \mathcal{EM}_n(\mathcal{H}^{mot}(k)) : \pi_n^{mot}\simeq \pi_n^{\mathbb{A}^1}\simeq \pi_n^{nis}
\end{flalign*}
where $n\mbox{-}\mathrm{Grp}^{L_{mot}}_k$ is the category of $n$-strongly $\mathbb{A}^1$-invariant $n$-group sheaves, i.e., $Sh_{k,*}^{\mathbb{A}^1}$ when $n=0$, $\mathrm{Grp}^{\mathbb{A}^1}_k$ (the category of strongly $\mathbb{A}^1$-invariant sheaves of groups), and $\mathrm{Ab}^{\mathbb{A}^1}_k$ (the category of strictly $\mathbb{A}^1$-invariant sheaves of abelian groups) when $n \geq 2$.
\end{inner-recall-star}

More generally, Nisnevich delooping induces an equivalence between pointed motivically $1$-connective spaces and grouplike motivic monoids whose $\pi_0^{mot}$ is strongly $\mathbb{A}^1$-invariant:

\begin{recall*}{mot bar omega}\label{intro:mot-bar-omega}
Over a perfect field $k$, there is an equivalence of $\infty$-categories
\[
\mathrm{B}_{mot}\simeq\mathrm{B}_{nis} : \mathcal{G}rp_{mot}(\mathcal{H}^{mot}(k)) \simeq \mathcal{H}_*^{mot}(k)_{\geq 1} : \Omega,
\]
where $\mathcal{G}rp_{mot}(\mathcal{H}^{mot}(k))$ denotes the full subcategory of $\mathrm{Grp}(\mathcal{H}^{mot}(k))$ consisting of grouplike motivic monoids with strongly $\mathbb{A}^1$-invariant $\pi_0^{mot}$.
\end{recall*}

\subsubsection*{2. Free group sheaves, birationality, and the non-topos theorem}
In an $\infty$-topos $\mathcal{E}$, the fundamental group functor $\pi_1^{\mathcal{E}}$ induces an equivalence between $\mathcal{EM}_1(\mathcal{E})$ and the category of all discrete group objects $\mathrm{Grp}(\mathrm{Disc}(\mathcal{E}))$ [\cite{lurie2009higher}, Proposition 7.2.2.12]. In the motivic setting, however, \Cref{motivic em recall} identifies the essential image of $\pi_1^{mot}$ not with all $\mathbb{A}^1$-invariant group sheaves $L_{\mathbb{A}^1}\mathrm{Grp}_k$, but only with the strongly $\mathbb{A}^1$-invariant ones, $\mathrm{Grp}_k^{\mathbb{A}^1}$. The question of whether $\mathcal{H}^{mot}(k)$ is an $\infty$-topos therefore hinges entirely on whether the inclusion $\mathrm{Grp}_k^{\mathbb{A}^1} \subset L_{\mathbb{A}^1}\mathrm{Grp}_k$ is strict.

While an abelian counterexample (i.e., one belonging to $L_{\mathbb{A}^1}\mathrm{Ab}_k\setminus \mathrm{Ab}^{\mathbb{A}^1}_k$) was known due to Choudhury [\cite{choudhury2014connectivity}, Lemma 5.6] (namely $\tilde{\mathbb{Z}}_{nis}(\mathbb{G}_m)$), we develop a systematic geometric machine that produces broad classes of non-abelian group, monoid, and abelian group counterexamples. In stark contrast to classical topology, where every free group is the fundamental group of a wedge of circles, we discover that free group sheaves in motivic homotopy theory that arise as motivic fundamental groups are necessarily birational.

\begin{recall*}{free groups are strong if and only if they are birational}\label{intro:free-grp-bir}
Let $k$ be a perfect field and let $\mathrm{F}$ be a Nisnevich sheaf of groups on $Sm_k$ which is generically free (i.e., its values on $\mathcal{F}_k$ are free groups). Then $\mathrm{F}$ is strongly $\mathbb{A}^1$-invariant if and only if it is birational local.
\end{recall*}

Moreover, we also show that Nisnevich-local free group sheaves generated by Nisnevich sheaves of sets are birational precisely when the generating sheaf of sets is itself birational. On the other hand, for a large class of sheaves of sets, which we call \emph{\'etale-weakly unramified} (\Cref{tau weakly unramified}), the generated free sheaves are motivic whenever the generating sheaves are:

\begin{recall*}{free grp sheaf strong vs birational}\label{intro:free-sheaf-bir}
Let $k$ be a perfect field and let $A$ be a Nisnevich sheaf of sets (resp.\ pointed sets) on $Sm_k$. Then the following are equivalent:
\begin{enumerate}
    \item $L_{nis}F_{grp}A$ (resp.\ $L_{nis}\tilde{F}_{grp}A$) is strongly $\mathbb{A}^1$-invariant.
    \item $L_{nis}F_{grp}A$ (resp.\ $L_{nis}\tilde{F}_{grp}A$) is birational.
    \item $A$ is birational.
\end{enumerate}
Moreover, if $A$ is \'etale-weakly unramified, then $L_{nis}F_{grp}A$ (resp.\ $L_{nis}\tilde{F}_{grp}A$) is $\mathbb{A}^1$-invariant if and only if $A$ is $\mathbb{A}^1$-invariant.
\end{recall*}

Consequently, any \'etale-weakly unramified $\mathbb{A}^1$-invariant sheaf $A$ that is not birational generates a free group sheaf $L_{nis}F_{grp}(A)$ that lies in $L_{\mathbb{A}^1}\mathrm{Grp}_k \setminus \mathrm{Grp}_k^{\mathbb{A}^1}$ (along with analogous free monoid and free abelian sheaves; see \Cref{free mon gm not strong}, \Cref{free ab strong vs bir}, and \Cref{main coex abelian}). Sheaves satisfying these conditions are abundant: examples include $\mathbb{G}_m$ (\Cref{main coex 2}), codimension $\geq 2$ complements of $\mathbb{A}^1$-rigid varieties (\Cref{main coex 1}), and rationalized Milnor $K$-theory $\underline{\mathrm{K}}_n^M \otimes \mathbb{Q}$ for $n \geq 1$ (\Cref{milnor is fully weakly un}, \Cref{milnor coex}).

Armed with both the identification of motivic Eilenberg-MacLane objects (\Cref{motivic em recall}) and this explicit reservoir of non-strong motivic group sheaves, we establish a rigorous, purely $\infty$-categorical proof of the observation of Spitzweck and \O{}stv\ae{}r [\cite{spitzweck2012motivic}, Remark 3.5] (see also [\cite{MR3890956}, Proposition 4.12]) that the motivic homotopy category of a perfect field is not an $\infty$-topos (\Cref{mot not topos for perf}). In fact, using essentially smooth base change arguments, we obtain a much stronger version:

\begin{recall*}{mot not topos}\label{intro:not-topos}
Let $S$ be a qcqs scheme. Then the motivic homotopy category $\mathcal{H}^{mot}(S)$ is an $\infty$-topos if and only if $S = \emptyset$.
\end{recall*}

\subsubsection*{3. $\mathbb{A}^1$-$n$-truncated and $\mathbb{A}^1$-$n$-connective morphisms: the motivic Galois correspondence, and the Van Kampen theorem}
In \S2.4 and \S3.4, we investigate the internal and external analogs of $n$-truncated morphisms and $n$-connective morphisms; in particular, for $n=0$, $\mathbb{A}^1$-coverings of Morel and $\mathbb{A}^1$-effective epimorphisms.

For truncated morphisms (\S2.4.1, \S3.4.1), we compare the $\infty$-category $Cov^n_{\mathcal{H}^{mot}}(L_{mot}\esc{X})$ of internal $n$-truncated morphisms over $L_{mot}\esc{X}$ with the $\infty$-category $Cov^n_{\mathbb{A}^1}(\esc{X})$ of $\mathbb{A}^1$-$n$-coverings of $\esc{X}$ (which for $n=0$ is the $\infty$-categorical analogue of Morel's category of $\mathbb{A}^1$-coverings [\cite{morel2012a1}, \S7], and for $n=-1$ consists of $\mathbb{A}^1$-monomorphisms). Because $L_{mot}$ is locally cartesian and effective (i.e., the localization unit is an effective epimorphism on each object; see \Cref{Effective localization remark}) over an arbitrary base scheme, pullback along the localization unit bridges the ambient and internal worlds:

\begin{recall*}{Lmot covering is covering of Lmot}\label{intro:cov-equiv}
Suppose $\esc{X}$ is a Nisnevich sheaf of spaces over an arbitrary base scheme $S$ and $n \geq -2$. We then have an equivalence of $\infty$-categories
\[
\eta^*_{mot}:Cov^n_{{\mathbb{A}^1}}({L_{mot}}\esc{X})=Cov^n_{\mathcal{H}^{mot}}({L_{mot}}\esc{X})\leftrightarrows Cov^n_{{\mathbb{A}^1}}(\esc{X}): L_{mot},
\]
where $\eta_{mot}$ is the localization unit $\mathrm{id}\to L_{mot}$ (which, when $\esc{X}$ is pointed, restricts to an equivalence $\eta^*_{mot}:Cov^n_{\mathcal{H}^{mot}}(L_{mot}\esc{X})_* \leftrightarrows Cov^n_{\mathbb{A}^1}(\esc{X})_* : L_{mot}$).
\end{recall*}

Over a perfect field $k$, combining this equivalence with the higher fundamentality of $L_{mot}$ on connected spaces (\Cref{layers of connected motivic space}) yields universal motivic and $\mathbb{A}^1$-$n$-covering spaces for all $n \geq 0$ (\Cref{motivically simply connected covering space}, \Cref{a1 simply connected covering space}; recovering [\cite{morel2012a1}, Theorem 7.8] when $n=0$), identifies $$Cov^n_{\mathcal{H}^{mot}(k)}(L_{mot}\esc{X}) \simeq Cov^n_{\mathbb{A}^1}(\esc{X}) \simeq Act_{\tau_{\leq n}^{mot}\Omega(L_{mot}\esc{X},x)}\bigl(\mathcal{H}^{mot}_{\leq n}(k)\bigr)$$ for pointed $\mathbb{A}^1$-connected spaces (giving $Cov_{\mathbb{A}^1}(\esc{X}) \simeq Act_{\pi_1^{\mathbb{A}^1}(\esc{X},x)}\bigl(Sh_{nis}^{\mathbb{A}^1}(k)\bigr)$ when $n=0$), and culminates in a unified $n$-th motivic Galois correspondence:

\begin{recall*}{motivic galois corr}\label{intro:galois-corr}
Over a perfect field $k$, the motivic homotopy category satisfies the $n$-th Galois correspondence for every $n \geq 0$: given a pointed $\mathbb{A}^1$-$n$-connected space $\esc{X}$, there is a commutative square of equivalences of posets (with notations to be explained)
\[
\xymatrix{
Cov^n_{\mathbb{A}^1}(\esc{X})_*^{>_{\mathbb{A}^1}n}
\ar[r]^-{\simeq} \ar[d]_{\simeq}
&
Cov^n_{\mathcal{H}^{mot}(k)}(L_{mot}\esc{X})_*^{>_{mot}n}
\ar[d]^-{\simeq}
\\
Sub_{(n+1)\mbox{-}\mathrm{Grp}^{L_{mot}}}(\pi_{n+1}^{\mathbb{A}^1}\esc{X})
\ar[r]^-{\simeq}
&
Sub_{(n+1)\mbox{-}\mathrm{Grp}^{L_{mot}}}(\pi_{n+1}^{mot}L_{mot}\esc{X}).
}
\]
\end{recall*}

Finally, for $\mathbb{A}^1$-connective morphisms (\S2.4.2, \S3.4.2), we first identify motivic $0$-connective morphisms with motivic effective epimorphisms, which leads us to a local-to-global $\mathbb{A}^1$-connectivity criterion for Nisnevich coverings (\Cref{motivic local to global conn}). In fact, we apply the internal truncation $\tau_{\leq n}^{mot}$ to the pointed \v{C}ech nerve of a covering and use the equivalence $\mathcal{EM}_n(\mathcal{H}^{mot}(k)) \simeq n\mbox{-}\mathrm{Grp}^{L_{mot}}$ to obtain a higher $\mathbb{A}^1$-Van Kampen theorem. For $n=1$, this extends Morel's Zariski Van Kampen theorem [\cite{morel2012a1}, Theorem 7.12] to arbitrary spaces while incorporating a connectivity hypothesis on triple intersections to resolve a subtle $2$-cocycle coherence gap in \textit{loc.\ cit.}:

\begin{recall*}{gen Van Kampen}\label{intro:van-kampen}
Let $k$ be a perfect field, let $n \geq 0$, and let $\{\esc{U}_i \to \esc{X}\}_{i \in I}$ be a pointed external geometric $\mathbb{A}^1$-covering of a $k$-space $\esc{X}$ (in particular, a Zariski covering $\{U_i \to X\}_{i \in I}$ of a smooth $k$-scheme $X$) such that each $m$-fold intersection $\esc{U}_{i_1\dots i_m} := \esc{U}_{i_1} \times_{\esc{X}} \cdots \times_{\esc{X}} \esc{U}_{i_m}$ for $1 \leq m \leq n+2$ is pointed (by the base points of $\esc{U}_i$'s) and $\mathbb{A}^1$-$(n-1)$-connected for every $i_1,\dots,i_m \in I$. Then:
\begin{enumerate}
    \item $\esc{X}$ is $\mathbb{A}^1$-$(n-1)$-connected.
    \item The $n$-th $\mathbb{A}^1$-homotopy group sheaf of $\esc{X}$ is given by the following coequalizer in $n\mbox{-}\mathrm{Grp}^{L_{mot}}_k$:
\[
    \coprod_{i,j} \pi_n^{\mathbb{A}^1}(\esc{U}_i \times_{\esc{X}} \esc{U}_j) \rightrightarrows \coprod_{i} \pi_n^{\mathbb{A}^1}(\esc{U}_i) \to \pi_n^{\mathbb{A}^1}(\esc{X}) \to *
\]
(with $\pi_0^{\mathbb{A}^1}$ replaced by $\pi_0^{mot}L_{mot} \simeq L_{\leq 0}^{mot}$ when $n=0$), where $\coprod$ denotes the coproduct in $n\mbox{-}\mathrm{Grp}^{L_{mot}}_k$.
\end{enumerate}
\end{recall*}

Specializing local-to-global connectivity (\Cref{motivic local to global conn}) and the higher Van Kampen coequalizer yields a Nisnevich-local-to-global $\mathbb{A}^1$-$n$-connectivity criterion for schemes, mirroring the classical topological fact that spaces admitting path-connected covers with non-empty pairwise intersections are path-connected:

\begin{recall*}{a1 geom conn}\label{intro:local-global-conn}
Over a perfect field $k$, a scheme $X/k$ is $\mathbb{A}^1$-connected if (and only if) there is a Nisnevich covering $\{U_i \to X\}_{i \in I}$ such that each $U_i$ is $\mathbb{A}^1$-connected and for every pair $i,j \in I$ the fiber product $U_i \times_X U_j$ has a $k$-point. More generally, for any $n \geq 1$, a scheme $X/k$ is $\mathbb{A}^1$-$n$-connected if (and only if) there is a pointed Nisnevich covering $\{U_i \to X\}_{i \in I}$ (necessarily a Zariski open covering when $X \in Sm_k$) such that each $U_i$ is $\mathbb{A}^1$-$n$-connected and for every $2 \leq m \leq n+2$ and $i_1,\dots,i_m \in I$, the fiber product $U_{i_1\dots i_m}$ is $\mathbb{A}^1$-$(n-1)$-connected.
\end{recall*}
\subsection{Organization of the paper}

This paper is structured to separate abstract $\infty$-categorical localization phenomena from their geometric applications in motivic homotopy theory.

In \textbf{Section 2}, we study the internal homotopy theory of a presentable $\infty$-category, focusing on accessible localizations $L \colon \mathcal{C} \to L\mathcal{C}$ of an $\infty$-topos $\mathcal{C}$. Rather than assuming $L$ is left-exact, we isolate minimal axiomatic conditions on $L$, specifically \emph{path-cartesianity} (\Cref{path cart}), \emph{path-injectivity} (\Cref{path injectivity}), \emph{effectivity} (\Cref{Effective localization remark}), and \emph{$m$-fundamentality}, under which internal truncations and homotopy groups (\S2.1), connective objects and deloopings (\S2.2), Eilenberg-MacLane objects (\S2.3), and both spatial and geometric $n$-covering theories (\S2.4) behave as they do in an $\infty$-topos. This general treatment applies simultaneously to the $\mathbb{A}^1$-localization (carried out in \S3) and to the birational localization studied in [\cite{0bat}].

In \textbf{Section 3}, we specialize this framework to the motivic localization $L_{mot} \colon \mathcal{P}_{nis}(S) \to \mathcal{H}^{mot}(S)$. Subsections \S3.1-\S3.4 parallel \S2.1-\S2.4 (with \S3.2 and \S3.3 treating Eilenberg-MacLane spaces and deloopings in the reverse order of \S2.2 and \S2.3, and \S3.2 additionally establishing the birationality and \'etale-weak unramifiedness criteria for free group sheaves that obstruct $\mathcal{H}^{mot}(k)$ from being an $\infty$-topos). In each subsection, we verify the required geometric hypotheses for $L_{mot}$ (summarized at the end of \S3.4) and deduce the corresponding motivic theorems. 

Readers primarily interested in motivic applications may proceed directly to Section 3, treating the axiomatic framework of Section 2 as a black box and referring back to it for categorical definitions and proofs as needed. For convenience, the verification of the categorical hypotheses satisfied by $L_{mot}$ is summarized at the conclusion of \S3.4 in \Cref{properties of Lmot}.
\subsection{Notation and terminology}
\textit{Categorical.}
We freely use the language of $\infty$-categories without fixing a particular model; a standard reference is [\cite{lurie2009higher}], where $\infty$-categories are modeled by quasi-categories and $\infty$-groupoids by Kan complexes.

We write $\mathrm{Sets}$, $\mathrm{Grp}$, and $\mathrm{Ab}$ for the $1$-categories of sets, groups, and abelian groups (implicitly viewed as $\infty$-categories via the nerve functor $\mathrm{N}$), $\mathrm{Spc}$ for the $\infty$-category of spaces ($\infty$-groupoids), and $\widehat{Cat}_\infty$ (resp.\ $Pr^L$) for the very large $\infty$-category of large (resp.\ large presentable) $\infty$-categories.

For an $\infty$-category $\mathcal{C}$, $\mathrm{Map}_\mathcal{C}(-,-)$ denotes the $\mathrm{Spc}$-valued mapping space in $\mathcal{C}$, $*$ denotes a terminal object (when present), and $\mathcal{C}_* := \mathcal{C}_{*/}$ denotes the $\infty$-category of pointed objects in $\mathcal{C}$. Given $X \in \mathcal{C}$ (resp.\ $X_* = (x \colon * \to X) \in \mathcal{C}_*$), we write $\mathcal{C}_{/X}$ (resp.\ $\mathcal{C}_{*/X} := (\mathcal{C}_*)_{/X_*}$) for the corresponding slice (resp.\ pointed slice) $\infty$-category, and $|M_\bullet|$ for the geometric realization (colimit) in $\mathcal{C}$ of a simplicial object $M_\bullet \in \mathrm{Fun}(\Delta^{op}, \mathcal{C})$. By $\mathcal{P}(\mathcal{C}) := \mathrm{Fun}(\mathcal{C}^{op}, \mathrm{Spc})$ we mean the $\infty$-category of presheaves of spaces on a small $\infty$-category $\mathcal{C}$. When $\mathcal{C}$ admits finite coproducts, $\mathcal{P}_\Sigma(\mathcal{C}) \subset \mathcal{P}(\mathcal{C})$ denotes the full subcategory of presheaves that send finite coproducts to finite products, with localization functor $L_\Sigma \colon \mathcal{P}(\mathcal{C}) \to \mathcal{P}_\Sigma(\mathcal{C})$.

\textit{Algebraic geometry.}
Throughout this paper, $S$ denotes a quasi-compact and quasi-separated (qcqs) scheme, and $k$ denotes a field (not assumed perfect unless explicitly stated). Morphisms of schemes are always assumed to be quasi-compact and separated; in particular, smooth morphisms are of finite presentation.

For a qcqs scheme $S$, $Sm_S$ denotes the category of smooth $S$-schemes, and for a field $k$, $\mathcal{F}_k$ denotes the category of finitely generated field extensions of $k$. We write $\mathcal{P}(S) := \mathcal{P}(Sm_S)$ for the $\infty$-category of presheaves of spaces on $Sm_S$ (whose objects we call \emph{$S$-spaces}), and set $\mathcal{P}_\Sigma(S) := \mathcal{P}_\Sigma(Sm_S)$. If $\sigma$ is a Grothendieck topology on $Sm_S$, $\mathcal{P}_\sigma(S) \subset \mathcal{P}(S)$ stands for the $\infty$-topos of $S$-spaces satisfying descent with respect to $\sigma$-sieves, with localization ($\sigma$-sheafification) functor $L_\sigma \colon \mathcal{P}(S) \to \mathcal{P}_\sigma(S)$, and $Sh_\sigma(S) := \mathrm{Disc}(\mathcal{P}_\sigma(S))$ denotes the $1$-category of $\sigma$-sheaves of sets on $Sm_S$. When $S = \mathrm{Spec}(k)$, we replace $\mathrm{Spec}(k)$ by $k$ in subscripts and arguments, writing $Sm_k$, $\mathcal{P}(k)$, $\mathcal{P}_{nis}(k)$, etc. For $\esc{X} \in \mathcal{P}(k)$ and $K \in \mathcal{F}_k$, $\esc{X}(K)$ (and similarly $\esc{X}(\mathbb{A}^1_K)$) denotes the filtered colimit of $\esc{X}(U)$ (resp.\ $\esc{X}(\mathbb{A}^1_U)$) over smooth models $U \in Sm_k$ of $K/k$.

\textit{Motivic homotopy.}
We denote by $\mathcal{H}^{mot}(S) := \mathcal{P}_{nis}(S) \mcap L_{\mathbb{A}^1}\mathcal{P}(S)$ the full subcategory of $\mathcal{P}(S)$ consisting of $\mathbb{A}^1$-local, Nisnevich-local objects, with motivic localization functor $L_{mot} \colon \mathcal{P}(S) \to \mathcal{H}^{mot}(S)$ (using the same symbol for its restriction $\mathcal{P}_{nis}(S) \to \mathcal{H}^{mot}(S)$), and write $Sh_S^{\mathbb{A}^1} = Sh_{nis}^{\mathbb{A}^1}(S) := \mathrm{Disc}(\mathcal{H}^{mot}(S))$ for the $1$-category of $\mathbb{A}^1$-invariant Nisnevich sheaves of sets. Because our main connectivity and fundamentality results are sensitive to the dimension of the base scheme, we work primarily with $\mathcal{H}^{mot}(k)$ over a field $k$. When a qcqs base $S$ is fixed and clear from context (as in the first part of \S3.1), we occasionally suppress $S$ from the notation and write $\mathcal{P}$, $\mathcal{P}_{nis}$, $\mathcal{H}^{mot}$, and $Sh_{nis}^{\mathbb{A}^1}$.

On $\mathcal{P}(S)$, $\tau_{\leq n}$ denotes sectionwise $n$-truncation, $\pi_0 = \tau_{\leq 0}$ denotes the presheaf of path components, and $\pi_n$ ($n \geq 1$) denotes the presheaf of $n$-th homotopy groups of a pointed $S$-space. The symbols $\tau_{\leq n}^{nis}$ and $\pi_n^{nis} := L_{nis}\pi_n$ denote the truncation and homotopy group functors internal to the Nisnevich $\infty$-topos $\mathcal{P}_{nis}(S)$, while $\pi_i^{\mathbb{A}^1}$ stands, as usual, for $\pi_i^{nis}L_{mot}$; their intrinsic counterparts in $\mathcal{H}^{mot}(S)$ are denoted $\tau_{\leq n}^{mot}$ and $\pi_i^{mot}$ (\S3.1). Finally, calligraphic letters $\mathcal{C}, \mathcal{D}$ denote $\infty$-categories; in \S2, italic capitals $X, Y, Z, E, F$ denote objects of an abstract $\infty$-category $\mathcal{C}$, whereas in \S3, italic capitals $S, U, V, W, X, Y, Z$ denote schemes, script letters $\esc{X}, \esc{Y}, \esc{Z}, \esc{U}, \esc{F}$ denote $S$-spaces, and capitals $A, F, G, M$ denote discrete presheaves or sheaves of sets, monoids, and groups.

\section{Internal homotopy theory in a presentable \texorpdfstring{$\infty$}{}-category}

In an $\infty$-topos $\mathcal{C}$, internal homotopy theory is governed by left-exactness [\cite{lurie2009higher}, \S6.5, \S7.2]: truncations $\tau_{\leq n}^{\mathcal{C}}$ preserve finite products, vanishing truncations coincide with vanishing homotopy sheaves ($\tau_{\geq n}^{\mathcal{C}}\mathcal{C}_* = \mathcal{C}_{\geq n}$), the bar-loop adjunction restricts to an equivalence $\mathrm{Grp}(\mathcal{C}) \simeq \mathcal{C}_{*,\geq 1}$, and $\mathcal{EM}_n(\mathcal{C}) \simeq n\mbox{-}\mathrm{Grp}(\mathrm{Disc}(\mathcal{C}))$. Under an accessible localization $L \colon \mathcal{C} \to L\mathcal{C}$ that is not left-exact, these identities split into two competing frameworks: an \emph{extrinsic} ($L$-local) theory borrowed from $\mathcal{C}$, governed by $\tau_{\leq n}^{\mathcal{C}}L$ and $\pi_i^L := \pi_i^{\mathcal{C}}L$ (\Cref{pinL}), and an \emph{intrinsic} ($L\mathcal{C}$-internal) theory native to $L\mathcal{C}$, governed by $\tau_{\leq n}^{L\mathcal{C}} $ and $\pi_i^{L\mathcal{C}} := \tau_{\leq 0}^{L\mathcal{C}}\Omega_{L\mathcal{C}}^i $.

Rather than assuming left-exactness, this section isolates the minimal axiomatic conditions on $L$ needed to lock the intrinsic and extrinsic theories together. While truncated objects always satisfy $(L\mathcal{C})_{\leq n} = \mathcal{C}_{\leq n} \mcap L\mathcal{C}$ (\Cref{truncation of localization}), ensuring that $\pi_n^{L\mathcal{C}}$ lands in $n$-group objects requires \emph{path-cartesianity} (\Cref{path cart,path cart for groups}). Our main absolute results (\S2.2-\S2.3) are driven by \emph{path-injectivity} (\Cref{path injectivity}), which requires $L_{\leq 0} \colon \tau_{\leq 0}^{\mathcal{C}}L\mathcal{C} \to \mathrm{Disc}(L\mathcal{C})$ to detect $*$ and yields $(L\mathcal{C})_{\geq n} = L\mathcal{C} \mcap \mathcal{C}_{\geq n}$ (\Cref{conn of localization}), the delooping equivalence $\mathrm{Grp}^L(L\mathcal{C}) \simeq (L\mathcal{C})_{*,\geq 1}$ (\Cref{loop bar for localization}), and the Eilenberg-MacLane classification $\mathcal{EM}_n(L\mathcal{C}) \simeq n\mbox{-}\mathrm{Grp}^L$ (\Cref{grp em}); adding \emph{$j$-fundamentality} (\Cref{defn m fundamental}) (for all $j<n$) further identifies $(L\mathcal{C})_{\geq n} = \tau_{\geq n}^{L\mathcal{C}}L\mathcal{C}_*$  \Cref{tau preserves connectivity}). Finally, in \S2.4, we relativize the study to $n$-truncated and $n$-connective morphisms.

\subsection{Truncations}
Let $n \ge -1$ be an integer. A space $X \in \mathrm{Spc}$ is called $n$-truncated if, for all choices of basepoints $x \in X$ and for all $i > n$, the homotopy groups $\pi_i(X,x)$ are trivial. For the absolute base case of $n = -2$, a space $X$ is called $(-2)$-truncated if it is contractible (which algebraically translates to requiring that $X$ is $(-1)$-connected and $\pi_i(X,x)$ is trivial for all $i \ge 0$). The following notion from [\cite{lurie2009higher}, Definition 5.5.6.1] offers an immediate $\infty$-categorical generalization of this concept:
\begin{defn}\label[defn]{defn n truncated}
Let $\mathcal{C}$ be an $\infty$-category. An object $F$ of $\mathcal{C}$ is said to be $n$-truncated if, for all $E\in \mathcal{C}$, the $\infty$-groupoid $\mathrm{Map}_\mathcal{C}(E,F)$ is $n$-truncated. In other words, the presheaf $\mathrm{Map}_{\mathcal{C}}(-,F):\mathcal{C}^{op}\to\mathrm{Spc}$ factors through the $\infty$-category $\mathrm{Spc}_{\leq n}$ of $n$-truncated spaces.
\end{defn}

We denote the category of $n$-truncated objects in $\mathcal{C}$ by $\mathcal{C}_{\leq n}$. In particular, $\mathcal{C}_{\leq 0}$ is a discrete category and we denote $\mathrm{Disc}(\mathcal{C}):=\mathrm{N}(\mathcal{C}_{\leq 0})$.

If $\mathcal{C}$ is a presentable $\infty$-category, the full subcategory of $n$-truncated objects is closed under limits and is therefore given by a localization $\tau_{\leq n}^\mathcal{C}$ [\cite{lurie2009higher}, Proposition 5.5.6.5].

We now analyze the concept of $n$-truncated objects under a localization.
\begin{lem}\label[lem]{truncation of localization}
Let $L\mathcal{C}\subset \mathcal{C}$ be a localizing subcategory. Then the inclusion $\mathcal{C}_{\leq n}\mcap L\mathcal{C} \subset (L\mathcal{C})_{\leq n}$ is an equivalence, i.e., we have an equality $(L\mathcal{C})_{\leq n}=\mathcal{C}_{\leq n}\mcap L\mathcal{C}$.
\end{lem}
\begin{proof}
The forward containment is obvious. Indeed, if $A$ belongs to $L\mathcal{C}$ and $\mathcal{C}_{\leq n}$, then for any $B\in L\mathcal{C}$, one has $\mathrm{Map}_{L\mathcal{C}}(B,A)= \mathrm{Map}_{\mathcal{C}}(B,A)$. But this last term is $n$-truncated since $A\in \mathcal{C}_{\leq n}$.

To see the reverse inclusion, note that since $(L\mathcal{C})_{\leq n}\subset L\mathcal{C}$, it suffices to show that $(L\mathcal{C})_{\leq n}\subset \mathcal{C}_{\leq n}$. So let $X\in (L\mathcal{C})_{\leq n}$, and let $Y\in \mathcal{C}$ be an arbitrary object. Then we have $\mathrm{Map}_{\mathcal{C}}(Y,X)\simeq \mathrm{Map}_{L\mathcal{C}}(LY,X)$. But this last term is an $n$-truncated $\infty$-groupoid because $X\in (L\mathcal{C})_{\leq n}$. This shows that $X\in \mathcal{C}_{\leq n}$ as well. Thus, $(L\mathcal{C})_{\leq n}\subset \mathcal{C}_{\leq n}$.
\end{proof}
In other words, the following diagram is a pullback square in $Cat_\infty$:
\[
\xymatrix{
(L\mathcal{C})_{\leq n}\ar@{^(->}[r]\ar@{^(->}[d]&
L\mathcal{C}\ar@{^(->}[d]\\
\mathcal{C}_{\leq n}\ar@{^(->}[r]&\mathcal{C}
}
\]

\begin{prop}\label[prop]{L trunc is a localization}
    Let $L\mathcal{C}\subset \mathcal{C}$ be an accessible localization of a presentable $\infty$-category $\mathcal{C}$. Then $(L\mathcal{C})_{\leq n}$ is a localization of $\mathcal{C}_{\leq n}$ and of $\mathcal{C}$.
\end{prop}
\begin{proof}
    Since both $\mathcal{C}_{\leq n}$ and $L\mathcal{C}$ are closed under limits in $\mathcal{C}$, we conclude immediately from the lemma above that $(L\mathcal{C})_{\leq n}\subset \mathcal{C}_{\leq n}$ is as well.
\end{proof}
Let us denote the corresponding localization functor by $L_{\leq n}:\mathcal{C}\to (L\mathcal{C})_{\leq n}$ (and by the same symbol its restriction $L_{\leq n}:\mathcal{C}_{\leq n}\to (L\mathcal{C})_{\leq n}$).

On the other hand, when $L$ is an accessible localization of a presentable $\infty$-category $\mathcal{C}$, so that $L\mathcal{C}$ is also presentable, we know by [\cite{lurie2009higher}, Proposition 5.5.6.5] that $(L\mathcal{C})_{\leq n}\subset L\mathcal{C}$ is also a localization. We denote the internal truncation functor for the $\infty$-category $L\mathcal{C}$ by $\tau^{L\mathcal{C}}_{\leq n}:L\mathcal{C}\to (L\mathcal{C})_{\leq n}$. It follows that:
\begin{lem}\label[lem]{formulae for tau under localization}
    $\tau^{L\mathcal{C}}_{\leq n} \simeq {L_{\leq n}\tau_{\leq n}^\mathcal{C}}_{|L\mathcal{C}}$.
\end{lem}

\begin{proof}
Suppose $X\in L\mathcal{C}$ and $Y\in (L\mathcal{C})_{\leq n}=L\mathcal{C}\mcap \mathcal{C}_{\leq n}$. We get
\begin{flalign*}
    \mathrm{Map}_{(L\mathcal{C})_{\leq n}}(\tau^{L\mathcal{C}}_{\leq n}X, Y)&\simeq \mathrm{Map}_{L\mathcal{C}}(X, Y)\text{ [by definition of } \tau^{L\mathcal{C}}_{\leq n}\text{]}\\&=\mathrm{Map}_{\mathcal{C}}(X, Y)\\&\simeq \mathrm{Map}_{\mathcal{C}_{\leq n}}(\tau_{\leq n}^{\mathcal{C}}X, Y)\text{ [since } Y\in \mathcal{C}_{\leq n}\text{]}
    \\&
    \simeq \mathrm{Map}_{(L\mathcal{C})_{\leq n}}(L_{\leq n}\tau_{\leq n}^{\mathcal{C}}X, Y).
\end{flalign*}
The result now follows from the Yoneda lemma.
\end{proof}
However, ${\tau_{\leq n}^\mathcal{C}}$ may not preserve $L$-local objects and hence the candidate $L_{\leq n}$ is not redundant in the formula; in other words, $\tau^{L\mathcal{C}}_{\leq n}$ may not simply be the restriction ${\tau_{\leq n}^\mathcal{C}}_{|L\mathcal{C}}$ in general.
\begin{defn}\label[defn]{defn m fundamental}
   We say that a localization $L$ is \emph{$m$-fundamental} on an object $Y \in \mathcal{C}$ if $\tau_{\leq m}^{\mathcal{C}} LY \simeq \tau_{\leq m}^{L\mathcal{C}} LY$ (equivalently, $\tau_{\leq m}^{\mathcal{C}} LY \in L\mathcal{C}$).
\end{defn}

\textit{Internal homotopy groups.} Let $\mathcal{C}$ be an $\infty$-category with finite limits. Given a pointing $x:*\to X$, we define the loop space functor $\Omega_{\mathcal{C}}:\mathcal{C}_*\to \mathcal{C}_*$ (and by iteration, the $n$-fold loop space functor $\Omega^n_\mathcal{C}: \mathcal{C}_*\to \mathcal{C}_*$) by either of the following two equivalent pullback diagrams:
\begin{figure}[ht]
\[
\xymatrix{
    \Omega_{\mathcal{C}}X \ar[r]^-{p_1} \ar[d]_-{p_2} & {*}\ar[d]^{x} & & \Omega_{\mathcal{C}}X \ar[r] \ar[d] & X \ar[d]^{\Delta }\\{*}
\ar[r]^{x} & X & & {*} \ar[r]^{(x,x)} & X \times X
}\]
\caption{Loop object}\label{loop object}
\end{figure}
\begin{defn}\label[defn]{LC homotopy group}
The $0$-th internal homotopy path components functor of a presentable $\infty$-category $\mathcal{C}$ is defined as $\pi_0^\mathcal{C}:=\tau_{\leq 0}^{\mathcal{C}}$. For $n>0$, the $n$-th internal homotopy group functor $\pi_n^\mathcal{C}: \mathcal{C}_*\to \mathrm{Disc}(\mathcal{C})_*$ is defined as $\tau_{\leq 0}^{\mathcal{C}}\Omega ^n_{\mathcal{C}}=\pi_0^\mathcal{C}\Omega^n_\mathcal{C}$.
\end{defn}

\begin{lem}\label[lem]{pi0 to Lpi0}
    Let $L\mathcal{C}$ be an accessible localization of a presentable $\infty$-category $\mathcal{C}$. Then $\pi_0^{L\mathcal{C}}\simeq {L_{\leq 0}\pi_0^{\mathcal{C}}}_{|L\mathcal{C}}$.
\end{lem}
\begin{proof}
    This follows at once from the identification $\tau_{\leq 0}^{L\mathcal{C}}=\pi_0^{L\mathcal{C}}$ and \Cref{formulae for tau under localization}.
\end{proof}

\begin{cor}\label[cor]{homotopy groups of localization}
    Let $L\mathcal{C}\subset \mathcal{C}$ be an accessible localization of a presentable $\infty$-category $\mathcal{C}$. Then $\pi_i^{L\mathcal{C}}: L\mathcal{C}_*\to \mathrm{Disc}(L\mathcal{C})_*$ is equivalent to ${L_{\leq 0}\pi_i^{\mathcal{C}}}_{|L\mathcal{C}_*}$.
\end{cor}
\begin{proof}
  First of all, since $L\mathcal{C}\subset\mathcal{C}$ is closed under limits and $\Omega$ is constructed from limits, we find that $\Omega_{L\mathcal{C}}^i=(\Omega^i_\mathcal{C})_{\vert{}L\mathcal{C}_*}$. By \Cref{formulae for tau under localization}, for any $X\in L\mathcal{C}_*$, we have $$\pi_i^{L\mathcal{C}}X:=\tau^{L\mathcal{C}}_{\leq 0}\Omega^i_{L\mathcal{C}}X\simeq L_{\leq 0}\tau^{\mathcal{C}}_{\leq 0}\Omega_\mathcal{C}^iX\simeq L_{\leq 0} \pi_i^{\mathcal{C}}X.$$
\end{proof}
\begin{rem}
Note that the above corollary does not say that $L$ takes homotopy groups to homotopy groups. Rather, it says that one must also localize at the level of discrete (a.k.a.\ $0$-truncated) objects to obtain the correct notion of internal homotopy ``groups.''
\end{rem}

\begin{lem}\label[lem]{omega reduce truncation level}
    Let $\mathcal{C}$ be an $\infty$-category with finite limits, and let $X\in \mathcal{C}_*$ be an $n$-truncated object for some $n$. Then $\Omega_{\mathcal{C}}X$ is $(n-1)$-truncated.
\end{lem}
\begin{proof}
    Recall that an object $X$ in a presentable $\infty$-category $\mathcal{C}$ is $n$-truncated if and only if the diagonal $X\xrightarrow{\Delta}X\times X$ is $(n-1)$-truncated [\cite{lurie2009higher}, Lemma 5.5.6.15]. From the pullback diagram on the right of \Cref{loop object}, we deduce that $\Omega_{\mathcal{C}}X$ is an $(n-1)$-truncated object. To see this, note that the base change functor $(x,x)^*: \mathcal{C}_{/X\times X}\to \mathcal{C}$ is left-exact (in fact, it is a right adjoint). So by [\cite{lurie2009higher}, Proposition 5.5.6.16], $(x,x)^*$ preserves $(n-1)$-truncated objects. In particular, because the right vertical map of the right square of \Cref{loop object} is $(n-1)$-truncated, the pullback condition yields that $\Omega_{\mathcal{C}}X\to *$ is $(n-1)$-truncated. But this is equivalent to saying that $\Omega_\mathcal{C}X$ is $(n-1)$-truncated.
\end{proof}
\begin{rem}
    The slice condition (i.e., $\Omega \tau_{\leq n}\simeq \tau_{\leq n-1}\Omega$) need not hold unless $\mathcal{C}$ is an $\infty$-topos. In other words, the canonical map $\tau_{\leq n-1}\Omega \to \Omega\tau_{\leq n}$ need not be an equivalence in general. Hence, it need not be the case that $\pi_i^{L\mathcal{C}}\to \pi_i^{L\mathcal{C}}\tau^{L\mathcal{C}}_{\leq i}$ is an equivalence. This holds for $n$-fundamental localizations as we shall show in \Cref{tau preserves connectivity}.
\end{rem}

\begin{prop}\label[prop]{htpy groups of n truncated objects}
    Let $\mathcal{C}$ be a presentable $\infty$-category and let $X\in \mathcal{C}$ be an $n$-truncated object for some $n\geq 0$. Then for every choice of basepoint $x:*\to X$ and every $i> n$, $\pi_i^\mathcal{C}(X,x)=*\in \mathrm{Disc}(\mathcal{C})$.
\end{prop}
\begin{proof}
 Using either of the pullback squares of \Cref{loop object}, observe that when $X$ is pointed, so is $\Omega_{\mathcal{C}}X$. Fixing a basepoint of $X$, we will thus drop the basepoint from the notation. Using \Cref{omega reduce truncation level}, it follows by induction that $\Omega^{i}X$ is $(n-i)$-truncated. So when $i>n+1$, we are done (since $(-2)$-truncated objects are contractible). It remains to show that $\pi_{n+1}^\mathcal{C}(X)=*$. Now, $\Omega^{n+1}X$ is $(-1)$-truncated (\Cref{omega reduce truncation level}). So it suffices to know that a pointed $(-1)$-truncated object is contractible. This follows from the similar statement in $\mathrm{Spc}$, the Yoneda lemma, and the definition of truncated objects.
\end{proof}
\begin{rem}\label[rem]{truncated condition}
Once again, the converse of the proposition need not be true in general. In fact, this need not even be true in arbitrary $\infty$-toposes. A minimal hypothesis under which the converse holds is hypercompleteness. Indeed, in any $\infty$-topos, vanishing of the higher homotopy sheaves induces an $\infty$-connective morphism $X\to \tau_{\leq n}^\mathcal{C}X$. So the converse is true, for example, in $\mathcal{P}(S)$ [\cite{lurie2009higher}, Example 7.2.1.9], or in $\mathcal{P}_{nis}(S)$ when $S$ is qcqs of finite Krull dimension (use [\cite{MR4296353}, Theorem 3.18] and [\cite{lurie2009higher}, Corollary 7.2.1.12] along with the fact that the collection of functors $\{p^*:\mathcal{P}_{nis}(S)\to \mathcal{P}_{nis}(\acute{e}t_X)\}_{p:X\to S\in Sm_S}$ is conservative).
\end{rem}
We end this subsection with a definition that we will use in the sequel:
\begin{defn}\label[defn]{pinL}
When $L\mathcal{C}\subset \mathcal{C}$ is an accessible localization of presentable $\infty$-categories, we define $\pi_0^L:\mathcal{C}\to \mathrm{Disc}(\mathcal{C})$ as the composition $\pi_0^{\mathcal{C}}\circ L$ (and similarly $\pi_i^L := \pi_i^{\mathcal{C}}\circ L$).
\end{defn}
\subsection{Connective and grouplike objects}
Throughout this subsection, $\mathcal{C}$ is a presentable $\infty$-category. In a general presentable $\infty$-category, there could possibly be many ways to define the notion of $n$-connective objects. We choose to define it dually to \Cref{htpy groups of n truncated objects}:

\begin{defn}\label[defn]{L n connective}
Suppose $\mathcal{C}$ is a presentable $\infty$-category.
    \begin{enumerate}
        \item We say that an object $X\in\mathcal{C}$ is $(n+1)$-connective, for some $n\geq -2$, if $\tau_{\leq n}^\mathcal{C}X\simeq *$. Every object is $(-1)$-connective. In the special case $n=0$ we say that the space is ($0$-)connected ($\pi_0^\mathcal{C}X=\tau^\mathcal{C}_{\leq 0}X\simeq *$). We denote the full subcategory of $(n+1)$-connective objects by $\tau^\mathcal{C}_{\geq n+1}\mathcal{C}$.
        \item We say that a pointed object $x:*\to X$ in a presentable $\infty$-category is globally $n$-connected (for some $n\geq 0$), if it is connected and the object $\pi_i^\mathcal{C}(X,x)$ is the terminal object for all $0\leq i\leq n$. For $m\geq 1$, we denote the full subcategory of $(m-1)$-connected (i.e., $m$-connective) objects by $\mathcal{C}_{\geq m}$ (and also write $\mathcal{C}_{>n} := \mathcal{C}_{\geq n+1}$).
    \end{enumerate}
\end{defn}

\begin{rem}
    Generally, $0$-connectiveness is not equivalent to the condition of existence of a global point. The latter is equivalent to saying that the canonical map $X\to *$ splits.
\end{rem}

\begin{rem}
    However, the above definition (2) is not a generalization of [\cite{lurie2009higher}, Definition 6.5.1.10], where the author defines it for an $\infty$-topos. The problem is that the local homotopy group object $\pi_iX$ cannot always be assembled from the global homotopy groups $\pi_i(X,x)$. This is possible in an $\infty$-topos. Indeed, a global basepoint is an effective epimorphism, so the claim follows via induction using [\cite{lurie2009higher}, Proposition 6.5.1.16(5)] applied on the left-hand square defining $\Omega_\mathcal{C}(X, x)$. Moreover, a version of [\cite{lurie2009higher}, Proposition 6.5.1.12(6)], i.e., ``an object $X$ is $(n+1)$-connective if and only if it is $n$-connected'', might be false unless $\mathcal{C}$ is an $\infty$-topos, although we shall show that both of these properties remain valid in the motivic homotopy category.
\end{rem}

\begin{rem}
    A key feature of an $\infty$-topos (for example, $\mathrm{Spc}$ of spaces, $\mathcal{P}(S)$, $\mathcal{P}_\Sigma(S)$, $\mathcal{P}_{nis}(S)$, etc.) is that in an $\infty$-topos an object is $n$-connective if and only if it is $(n-1)$-connected. In other words, $\mathcal{C}_{\geq n}$ coincides with the fiber of $\tau_{\leq n-1}^\mathcal{C}:\mathcal{C}\to \mathcal{C}$ at the final object [\cite{lurie2009higher}, Proposition 6.5.1.12], provided $\mathcal{C}$ is an $\infty$-topos. This could be a test for the failure of the $\infty$-topos property. For the motivic localization, however, this works (see \Cref{mot conn if and only if trunc trivial}), which we shall extend below (\Cref{tau preserves connectivity}) through a few formal definitions.
\end{rem}

\begin{defn}\label[defn]{L connected and LC connected}
Let $X\in \mathcal{C}$ be an object and let $L$ be a localization of $\mathcal{C}$. We say that $X$ is $L$-locally $n$-connected (equivalently $L$-connected, when $n=0$) for $n\geq 0$ (resp.\ $L$-locally $(n+1)$-connective, for $n\geq -2$) if $LX$ is an $n$-connected object of $\mathcal{C}$, i.e., $\pi_i^LX=\pi_i^\mathcal{C}LX\simeq *$ for all $0\leq i\leq n$ (resp.\ $(n+1)$-connective in $\mathcal{C}$). We say that $X$ is $L\mathcal{C}$-$n$-connected (equivalently, $L\mathcal{C}$-connected, when $n=0$) (resp.\ $L\mathcal{C}$-$(n+1)$-connective) if $LX$ is $n$-connected (resp.\ $(n+1)$-connective) internally in $L\mathcal{C}$, i.e., $\pi_i^{L\mathcal{C}}LX\simeq *$ for all $0\leq i\leq n$.
\end{defn}

\begin{lem}
    An $L$-connected object $X\in \mathcal{C}$ is $L\mathcal{C}$-connected.
\end{lem}
\begin{proof}
    This follows at once from \Cref{pi0 to Lpi0}.
\end{proof}

\begin{defn}\label[defn]{path injectivity}
    We say that the localization $L$ of the $\infty$-topos $\mathcal{C}$ is path-injective if the discrete localization functor $L_{\leq 0}: \tau_{\leq 0}^{\mathcal{C}}L\mathcal{C}\to \mathrm{Disc}(L\mathcal{C})$ detects the terminal object.
\end{defn}
In other words, a localization $L$ is path-injective if an object $X\in \mathcal{C}$ is $L$-connected if and only if it is $L\mathcal{C}$-connected.

Let $\tau_{\leq 0}^\mathcal{C}(L\mathcal{C})$ denote the image of the composite $L\mathcal{C}\subset \mathcal{C}\xrightarrow[]{\tau_{\leq 0}^\mathcal{C}}\mathrm{Disc}(\mathcal{C})$.

\begin{lem}\label[lem]{connective objects of ordinary localization}
    For any localization $L\mathcal{C}\subset \mathcal{C}$, there is a containment
    \[
        L\mathcal{C}\mcap \mathcal{C}_{\geq n} \subset (L\mathcal{C})_{\geq n}.
    \]
\end{lem}
\begin{proof}
    Let $X\in L\mathcal{C}\mcap \mathcal{C}_{\geq n}$. We shall show that $X\in(L\mathcal{C})_{\geq n}$, that is, $\pi_i^{L\mathcal{C}}X\simeq *$ for all $0\leq i\leq n-1$. But then \Cref{homotopy groups of localization} implies that $\pi_i^{L\mathcal{C}}X\simeq L_{\leq 0}\pi^\mathcal{C}_{i}X\simeq *$ (since $X\in\mathcal{C}_{\geq n}$) as desired.
\end{proof}

\begin{rem}\label[rem]{rem tau connective vs global}
    Similar statements are not necessarily true for the connective subcategories $\tau^{L\mathcal{C}}_{\geq n}L\mathcal{C}$.
\end{rem}

\begin{prop}\label[prop]{conn of localization}
    Let $L$ be a path-injective localization of an $\infty$-topos $\mathcal{C}$. Then there is an equality
    \[
        L\mathcal{C}\mcap \mathcal{C}_{\geq n} = (L\mathcal{C})_{\geq n}.
    \]
\end{prop}
\begin{proof}
    By \Cref{connective objects of ordinary localization}, it suffices to show that if $X\in (L\mathcal{C})_{\geq n}$, then $X\in \mathcal{C}_{\geq n}\mcap L\mathcal{C}$. This is equivalent to showing that $\pi_i^\mathcal{C}X\simeq *$ for all $i\leq n-1$. Since $\Omega_\mathcal{C}^iX\in L\mathcal{C}$, the object $\pi_i^\mathcal{C}X = \tau_{\leq 0}^\mathcal{C}\Omega_\mathcal{C}^iX$ lies in $\tau_{\leq 0}^\mathcal{C}(L\mathcal{C})$. Because $L_{\leq 0}$ detects the terminal object on $\tau_{\leq 0}^\mathcal{C}(L\mathcal{C})$, it suffices to show that $L_{\leq 0}\pi_i^\mathcal{C}X\simeq *$. But by \Cref{homotopy groups of localization}, we have
    \[L_{\leq 0}\pi_i^{\mathcal{C}}X\simeq \pi_i^{L\mathcal{C}}X,\]
    which is $*$ for all $i\leq n-1$ since $X\in (L\mathcal{C})_{\geq n}$.
\end{proof}
\begin{cor}\label[cor]{tau preserves connectivity}
Let $n \geq 1$ and let $L$ be a path-injective localization of an $\infty$-topos $\mathcal{C}$ that is $j$-fundamental on $(j-1)$-connected objects for all $1 \leq j \leq n-1$ (a vacuous condition when $n=1$). Then for any $X \in L\mathcal{C}_*$,
\[
    X \in (L\mathcal{C})_{\geq n} \iff \tau_{\leq n-1}^{L\mathcal{C}} X \simeq * \iff \tau_{\leq n-1}^{\mathcal{C}} X \simeq *,
\]
so that $(L\mathcal{C})_{\geq n} = \tau_{\geq n}^{L\mathcal{C}} L\mathcal{C}_* = L\mathcal{C} \mcap \mathcal{C}_{\geq n}$. Consequently:
\begin{enumerate}
    \item $(L\mathcal{C})_{\geq n}$ is closed under colimits (in particular, wedge sums) in $L\mathcal{C}_*$.
    \item For every $m \geq n$, $\tau_{\leq m}^{L\mathcal{C}}$ preserves $(L\mathcal{C})_{\geq n}$.
\end{enumerate}
\end{cor}
\begin{proof}
By \Cref{conn of localization}, $X \in (L\mathcal{C})_{\geq n} $ if and only if $X \in L\mathcal{C}$ and $\tau_{\leq n-1}^{\mathcal{C}} X \simeq *$. If $\tau_{\leq n-1}^{\mathcal{C}} X \simeq *$, then $\tau_{\leq n-1}^{L\mathcal{C}} X \simeq L_{\leq n-1}\tau_{\leq n-1}^{\mathcal{C}} X \simeq *$ by \Cref{formulae for tau under localization}. Conversely, suppose $\tau_{\leq n-1}^{L\mathcal{C}} X \simeq *$. We show $X \in L\mathcal{C} \mcap \mathcal{C}_{\geq j}$ by induction on $1 \leq j \leq n$. For $j=1$, $\tau_{\leq 0}^{L\mathcal{C}} X \simeq \tau_{\leq 0}^{L\mathcal{C}}\tau_{\leq n-1}^{L\mathcal{C}} X \simeq *$, so $X \in L\mathcal{C} \mcap \mathcal{C}_{\geq 1}$ by path-injectivity. For $1 \leq j \leq n-1$, if $X \in L\mathcal{C} \mcap \mathcal{C}_{\geq j}$, then $j$-fundamentality on $(j-1)$-connected objects yields $\tau_{\leq j}^{\mathcal{C}} X \simeq \tau_{\leq j}^{L\mathcal{C}} X \simeq \tau_{\leq j}^{L\mathcal{C}}\tau_{\leq n-1}^{L\mathcal{C}} X \simeq *$, so $X \in L\mathcal{C} \mcap \mathcal{C}_{\geq j+1}$.

Claims (1) and (2) now follow immediately because $\tau_{\leq n-1}^{L\mathcal{C}} \colon L\mathcal{C}_* \to (L\mathcal{C})_{*,\leq n-1}$ preserves colimits and satisfies $\tau_{\leq n-1}^{L\mathcal{C}} \tau_{\leq m}^{L\mathcal{C}} \simeq \tau_{\leq n-1}^{L\mathcal{C}}$ for $m \geq n$.
\end{proof}

Let us recall that an $\infty$-topos has the remarkable property that every pointed connected object is the delooping of some internal group object (and vice versa). We shall now explain what happens to this property under a non-exact localization. To set the stage for the loop-deloop machinery, let us briefly recall the bar construction for monoids in a presentable $\infty$-category $\mathcal{C}$. Roughly, a monoid object in $\mathcal{C}$ is a Segal object in the cartesian monoidal category $\mathcal{C}^\times$, i.e., a simplicial object $M_\bullet$ equipped with a compatible family of equivalences $M_n\simeq M_1\times_{M_0}\cdots \times_{M_0}M_1$. The bar construction $\mathrm{B}M_\bullet$ is defined as the colimit of the simplicial object $M_\bullet$ regarded as a simplicial diagram in $\mathcal{C}$. Because monoids are canonically pointed by the unit morphism, the resulting colimit is also pointed. Thus, the bar construction is given by a functor $\mathrm{B}_\mathcal{C}:\mathrm{Mon}(\mathcal{C})\to \mathcal{C}_*$.

Now, let us return to the internal loop space construction $\Omega_\mathcal{C}:\mathcal{C}_* \to \mathcal{C}$ (\Cref{loop object}). Given a pointed object $x: *\to X$, one can define the monoid structure on $\Omega_\mathcal{C} X$ by identifying it with the $1$-simplex of the \v{C}ech nerve $\check{C}x \in Aug(\mathcal{C}):=\mathrm{Fun}(\Delta_+^{op},\mathcal{C})$, taking the unaugmented part $(\Omega_{\mathcal{C}}X)_\bullet =(\check{C}x)_{\geq 0}$.

\begin{prop}\label[prop]{bar loop adjunction}
    Let $\mathcal{C}$ be a presentable $\infty$-category. Then there exists an adjunction:
    \[
        \mathrm{B}_{\mathcal{C}}: \mathrm{Mon}(\mathcal{C}) \leftrightarrows \mathcal{C}_*:\Omega_{\mathcal{C}}.
    \]
\end{prop}
\begin{proof}
    For a monoid object $M_\bullet$ and a pointed object $x:*\to X$, we have
    \begin{flalign*}
        \mathrm{Map}_{\mathrm{Mon}(\mathcal{C})}(M_\bullet, \Omega_{\mathcal{C}}X) &\simeq \mathrm{Map}_{\mathrm{Mon}(\mathcal{C})}(M_\bullet, \check{C}x_{\geq 0}) &\\
        &\simeq \mathrm{Map}_{Aug(\mathcal{C})}(M_\bullet\to |M_\bullet|, \check{C}x) &\\
        &\simeq \mathrm{Map}_{\mathrm{Fun}((\Delta^{\leq 0}_{+})^{op}, \mathcal{C})}(*\xrightarrow{1}|M_\bullet|, *\xrightarrow{x} X) \text{ [\cite{lurie2009higher}, Proposition 6.1.2.11]} &\\
        &= \mathrm{Map}_{\mathcal{C}_{*}}(|M_\bullet|, X). &
    \end{flalign*}
\end{proof}

More generally, given a cartesian $\infty$-category $\mathcal{C}$, an $n$-group object is a grouplike $\mathbb{E}_n$-monoid object of $\mathcal{C}$. We denote the $\infty$-category of $n$-group objects in $\mathcal{C}$ (i.e., the full subcategory of $\mathrm{Mon}_{\mathbb{E}_n}(\mathcal{C})$ spanned by grouplike objects) by $n\mbox{-}\mathrm{Grp}(\mathcal{C})$ (with $1\mbox{-}\mathrm{Grp}(\mathcal{C}) = \mathrm{Grp}(\mathcal{C})$). By formal arguments, the loop construction $\Omega_\mathcal{C}$ in an arbitrary presentable $\infty$-category yields a grouplike monoid in $\mathcal{C}$ (and $\Omega^i_\mathcal{C}$ takes values in $i$-group objects): switching $p_1$ and $p_2$ in \Cref{loop object} induces an inverse operation $\iota : \Omega_{\mathcal{C}}X\to \Omega_{\mathcal{C}}X$ for the monoid structure of $\Omega_{\mathcal{C}}X$. Therefore, the adjunction of \Cref{bar loop adjunction} restricts to an adjunction:
\[
    \mathrm{B}_{\mathcal{C}}: \mathrm{Grp}(\mathcal{C}) \leftrightarrows \mathcal{C}_*:\Omega_{\mathcal{C}}.
\]

\begin{rem}\label[rem]{em equiv of a topos}
    If $\mathcal{C}$ is an $\infty$-topos, then the adjunction $\mathrm{B}_\mathcal{C}\dashv\Omega_\mathcal{C}$ restricts to an equivalence of $\infty$-categories $\mathrm{Grp}(\mathcal{C}) \simeq \mathcal{C}_{*,\geq 1}$ [\cite{lurie2009higher}, Lemma 7.2.2.11(1)].
\end{rem}

When we are in the situation of an accessible localization $L\mathcal{C}\subset \mathcal{C}$, the equivalence from the remark above does not immediately restrict to a similar equivalence for the presentable $\infty$-category $L\mathcal{C}$. Under the hypothesis of \Cref{conn of localization}, we have $(L\mathcal{C})_{\geq 1} = \mathcal{C}_{\geq 1}\mcap L\mathcal{C}$, but unfortunately $\mathrm{B}_{\mathcal{C}}$ need not take local objects to local objects (even if they are discrete), unless the localization preserves finite products and geometric realizations; when $L$ is cartesian, we have $\mathrm{B}_{L\mathcal{C}}\simeq {L\mathrm{B}_\mathcal{C}}_{|L\mathcal{C}}$, which need not agree with ${\mathrm{B}_\mathcal{C}}_{|L\mathcal{C}}$.

\begin{lem}\label[lem]{bar of cart localization}
    Suppose $L$ is a cartesian localization of an $\infty$-topos. Then $\mathrm{B}_{L\mathcal{C}}L\simeq L\mathrm{B}_\mathcal{C}$ and $\Omega_{L\mathcal{C}}\simeq {\Omega_\mathcal{C}}_{|L\mathcal{C}}$.
\end{lem}
\begin{proof}
    This is clear because the bar construction is given by finite products and geometric realizations (and $L:\mathcal{C}\to L\mathcal{C}$ commutes with both), while $L\mathcal{C}\subset \mathcal{C}$ is closed under limits and, in particular, under fiber products.
\end{proof}

Let $\mathrm{Grp}_{L}(L\mathcal{C})$ denote the full subcategory of $\mathrm{Grp}(L\mathcal{C})$ consisting of objects $X$ such that $\mathrm{B}_\mathcal{C}X\in L\mathcal{C}$, i.e., $\mathrm{B}_{L\mathcal{C}}X\simeq \mathrm{B}_{\mathcal{C}}X$. Following [\cite{elmanto2021motivic}, \S3.4], we will call such objects strongly $L$-local monoids.

\begin{thm}\label{loop bar for localization}
    Suppose $L$ is a path-injective localization of an $\infty$-topos $\mathcal{C}$. Then the adjunction
    \[
        \mathrm{B}_{L\mathcal{C}} : \mathrm{Grp}(L\mathcal{C})\leftrightarrows L\mathcal{C}_*: \Omega_{L\mathcal{C}}\simeq {\Omega_\mathcal{C}}_{|L\mathcal{C}}
    \]
    restricts to an equivalence
    \[
        \mathrm{B}_{L\mathcal{C}} : \mathrm{Grp}^L(L\mathcal{C})\simeq (L\mathcal{C})_{*,\geq 1}: \Omega_{L\mathcal{C}}\simeq \Omega_\mathcal{C}.
    \]
\end{thm}
\begin{proof}
    Under the hypothesis, we know that $(L\mathcal{C})_{\geq 1}= L\mathcal{C}\mcap \mathcal{C}_{\geq 1}$ (\Cref{conn of localization}). So it suffices to show that if $X\in L\mathcal{C}\mcap \mathcal{C}_{*,\geq 1}$, then $\Omega_{\mathcal{C}}X \in \mathrm{Grp}^L(L\mathcal{C})$. But we know that $\mathrm{B}_\mathcal{C}\Omega_\mathcal{C}X\simeq X\in L\mathcal{C}$ (the first equivalence holds because $X\in \mathcal{C}_{*,\geq 1}$ [\cite{lurie2009higher}, Lemma 7.2.2.11(1)]), and we are done.
\end{proof}
\subsection{Eilenberg-MacLane and discrete group objects}
In this subsection, we specialize the preceding analysis to the discrete case, studying Eilenberg-MacLane objects and discrete group objects under localization.

\begin{defn}\label[defn]{defn em object}
An object in $\mathcal{C}_*$ is called $n$-Eilenberg-MacLane if it is $n$-connective and $n$-truncated in $\mathcal{C}$. We denote by $\mathcal{EM}_n(\mathcal{C})$ the full subcategory of $n$-Eilenberg-MacLane objects in $\mathcal{C}$.
\end{defn}

\begin{rem}
    Because \Cref{connective objects of ordinary localization} does not necessarily hold for $\tau^{L\mathcal{C}}_{\geq n}L\mathcal{C}$ (\Cref{rem tau connective vs global}), we define our Eilenberg-MacLane objects as $\mathcal{C}_{\leq n}\mcap \mathcal{C}_{\geq n}$ (with $\mathcal{C}_{\geq 0} := \mathcal{C}_*$) instead of $\tau^{\mathcal{C}}_{\geq n}\mathcal{C}_{\leq n}$.
\end{rem}

\begin{lem}\label[lem]{connective em objects of ordinary localization}
    For any localization $L\mathcal{C}\subset \mathcal{C}$, there are containments
    \[
        L\mathcal{C}\mcap \mathcal{C}_{\geq n}\subset (L\mathcal{C})_{\geq n}\quad \text{and}\quad L\mathcal{C}\mcap \mathcal{EM}_n(\mathcal{C})\subset \mathcal{EM}_n(L\mathcal{C}).
    \]
\end{lem}
\begin{proof}
    The first containment is \Cref{connective objects of ordinary localization}. Using \Cref{truncation of localization}, we have
    \begin{flalign*}
        L\mathcal{C}\mcap \mathcal{EM}_n(\mathcal{C}) &= L\mathcal{C}\mcap \mathcal{C}_{\leq n}\mcap\mathcal{C}_{\geq n} \\
        &= (L\mathcal{C})_{\leq n}\mcap L\mathcal{C}\mcap \mathcal{C}_{\geq n} \\
        &\subset (L\mathcal{C})_{\leq n}\mcap (L\mathcal{C})_{\geq n} = \mathcal{EM}_n(L\mathcal{C}).
    \end{flalign*}
\end{proof}

\begin{prop}\label[prop]{em of localization}
    Let $L$ be a path-injective localization of an $\infty$-topos $\mathcal{C}$. Then there are equalities
    \[
        L\mathcal{C}\mcap \mathcal{C}_{\geq n} = (L\mathcal{C})_{\geq n}\quad \text{and}\quad \mathcal{EM}_n(L\mathcal{C}) = L\mathcal{C}\mcap \mathcal{EM}_n(\mathcal{C}).
    \]
\end{prop}
\begin{proof}
    Using \Cref{truncation of localization} and \Cref{conn of localization}, we have
    \[
        L\mathcal{C}\mcap \mathcal{EM}_n(\mathcal{C}) = L\mathcal{C}\mcap \mathcal{C}_{\leq n}\mcap\mathcal{C}_{\geq n} = (L\mathcal{C})_{\leq n}\mcap (L\mathcal{C}\mcap \mathcal{C}_{\geq n}) = (L\mathcal{C})_{\leq n}\mcap (L\mathcal{C})_{\geq n} = \mathcal{EM}_n(L\mathcal{C}).
    \]
\end{proof}

\begin{rem}
    Notice that if $Y \in (L\mathcal{C})_{\geq m} = L\mathcal{C} \mcap \mathcal{C}_{\geq m}$ (assuming $L$ is path-injective) and $L$ is $m$-fundamental on $Y$, then $\tau_{\leq m}^{L\mathcal{C}} Y \simeq \tau_{\leq m}^{\mathcal{C}} Y \in L\mathcal{C} \mcap \mathcal{EM}_m(\mathcal{C}) = \mathcal{EM}_m(L\mathcal{C})$ by \Cref{em of localization}, and $\pi_m^{L\mathcal{C}}(\tau_{\leq m}^{L\mathcal{C}} Y) \simeq \pi_m^{\mathcal{C}}(\tau_{\leq m}^{\mathcal{C}} Y) \simeq \pi_m^{\mathcal{C}} Y \simeq \pi_m^{L\mathcal{C}} Y \in m\mbox{-}\mathrm{Grp}^L$ (\Cref{strictly n local}). (For $m=1$, $\tau_{\leq 0}^{L\mathcal{C}}\tau_{\leq 1}^{L\mathcal{C}} Y \simeq \tau_{\leq 0}^{L\mathcal{C}} Y \simeq *$ ensures $\tau_{\leq 1}^{L\mathcal{C}} Y \in \mathcal{EM}_1(L\mathcal{C})$ unconditionally.)
\end{rem}

When $\mathcal{D}$ is an ordinary $1$-category, $0\mbox{-}\mathrm{Grp}(\mathcal{D})\simeq \mathcal{D}_*$, $1\mbox{-}\mathrm{Grp}(\mathcal{D})\simeq \mathrm{Grp}(\mathcal{D})$, and $n\mbox{-}\mathrm{Grp}(\mathcal{D})\simeq \mathrm{Ab}(\mathcal{D})$ for every $n\geq 2$. Because $\Omega^n_\mathcal{C}$ takes values in $n$-group objects and $\tau_{\leq 0}^{\mathcal{C}}$ in an $\infty$-topos $\mathcal{C}$ is cartesian [\cite{lurie2009higher}, Lemma 6.5.1.2], it follows that $\pi_n^\mathcal{C}X$ is an $n$-group object in $\mathrm{Disc}(\mathcal{C})$: $\pi_n^\mathcal{C}: \mathcal{C}_*\to n\mbox{-}\mathrm{Grp}(\mathrm{Disc}(\mathcal{C}))$. However, in a general presentable $\infty$-category $\pi_0^{\mathcal{C}}$ need not be cartesian; thus, our internal homotopy groups of a localization need not be true group objects.

\begin{lem}\label[lem]{path cart for groups}
    Let $\mathcal{C}$ be an $\infty$-topos and $L$ an accessible localization. Assume that $L_{\leq 0}$ is cartesian. Then for every $X\in L\mathcal{C}_*$, $\pi_n^{L\mathcal{C}}X$ is an $n$-group object of $\mathrm{Disc}(L\mathcal{C})$.
\end{lem}
\begin{proof}
    Since $\pi_n^{L\mathcal{C}}\simeq {L_{\leq 0}\pi_n^\mathcal{C}}_{|L\mathcal{C}_*}$ (\Cref{homotopy groups of localization}), this follows from the fact that $\pi_n^\mathcal{C}$ has the same property when $\mathcal{C}$ is an $\infty$-topos [\cite{lurie2009higher}, the paragraph after Lemma 6.5.1.2].
\end{proof}

We shall therefore make the following definition:
\begin{defn}\label[defn]{path cart}
    We say that the localization is path-cartesian if $L_{\leq 0}$ is cartesian.
\end{defn}

Thus, a path-cartesian localization $L$ of an $\infty$-topos $\mathcal{C}$ ensures that $L\mathcal{C}$ retains the $\infty$-topos property that internal homotopy groups are true group objects. It follows that, when $L$ is path-cartesian, there is an adjunction $$i:\mathrm{Grp}(\mathrm{Disc}(L\mathcal{C}))\leftrightarrows \mathrm{Grp}(L\mathcal{C}):\pi_0^{L\mathcal{C}}.$$ Composing this with an iteration of the adjunction of \Cref{loop bar for localization} we thus obtain the familiar adjunction:
\begin{align}\label{bar pi adjunct}
    \mathrm{B}_{L\mathcal{C}}^n:n\mbox{-}\mathrm{Grp}(\mathrm{Disc}(L\mathcal{C}))\leftrightarrows L\mathcal{C}_*: \pi_n^{L\mathcal{C}}.
\end{align}
When $L\mathcal{C}$ is an $\infty$-topos this famously restricts to an equivalence on the internal Eilenberg-MacLane objects:

\begin{thm}[\textup{[\cite{lurie2009higher}, Proposition 7.2.2.12]}]\label{em in infinity topos}
    In an $\infty$-topos $\mathcal{C}$, the restriction of the homotopy group functors
    \[
        {\pi_n^\mathcal{C}}_{|\mathcal{EM}_n}: \mathcal{EM}_n(\mathcal{C})\to n\mbox{-}\mathrm{Grp}(\mathrm{Disc}(\mathcal{C}))
    \]
    induces equivalences of categories.
\end{thm}

The inverse to this equivalence is given by $n$-fold iteration of the bar construction $\mathrm{B}_\mathcal{C}:\mathrm{Mon}(\mathcal{C})\to \mathcal{C}_*$ on $\mathbb{E}_n$-monoids (with $\mathrm{B}^0_\mathcal{C}$ the canonical inclusion), yielding $\mathrm{B}^n_\mathcal{C}: n\mbox{-}\mathrm{Grp}(\mathrm{Disc}(\mathcal{C})) \to \mathcal{C}_*$. As noted in the previous subsection, $\mathrm{B}_{\mathcal{C}}$ need not take local objects to local objects even when they are discrete, so $\mathrm{B}^n_{L\mathcal{C}}$ need not agree with ${\mathrm{B}^n_\mathcal{C}}_{|n\mbox{-}\mathrm{Grp}(\mathrm{Disc}(L\mathcal{C}))}$.

\begin{defn}\label[defn]{strictly n local}
    Let $G\in n\mbox{-}\mathrm{Grp}(\mathrm{Disc}(L\mathcal{C}))$. We say that $G$ is $n$-strictly $L$-local if the canonical map $\mathrm{B}^n_\mathcal{C}G\to \mathrm{B}^n_{L\mathcal{C}}G$ is an equivalence (equivalently, $\mathrm{B}^n_\mathcal{C}G \in L\mathcal{C}$). Denote the full subcategory of $n\mbox{-}\mathrm{Grp}(\mathrm{Disc}(L\mathcal{C}))$ consisting of $n$-strictly $L$-local objects by $n\mbox{-}\mathrm{Grp}^L$.
\end{defn}

\begin{prop}\label[prop]{grp em}
    Let $L$ be a path-injective localization of an $\infty$-topos. Then for every $n\geq 0$, there is an equivalence of categories $$\mathrm{B}^n_{L\mathcal{C}}: n\mbox{-}\mathrm{Grp}^L\simeq \mathcal{EM}_n(L\mathcal{C}):\pi_n^{L\mathcal{C}}.$$ In fact, when restricted to $\mathcal{EM}_n(L\mathcal{C})$, there is an equivalence $\pi_n^{L\mathcal{C}}\simeq \pi_n^\mathcal{C}$. (Note that this bypasses the requirement of path-cartesianness needed in establishing the adjunction of \Cref{bar pi adjunct}.)
\end{prop}
\begin{proof}
    Since, by definition, ${\mathrm{B}^n_{L\mathcal{C}}}_{|n\mbox{-}\mathrm{Grp}^L}\simeq {\mathrm{B}^n_\mathcal{C}}_{|n\mbox{-}\mathrm{Grp}^L}$, it follows that the functor $\mathrm{B}^n_{L\mathcal{C}}: n\mbox{-}\mathrm{Grp}^L\to\mathcal{EM}_n(L\mathcal{C})$ is well-defined. Now the equivalence
    \[
        \mathrm{B}^n_{\mathcal{C}}: n\mbox{-}\mathrm{Grp}(\mathrm{Disc}(\mathcal{C}))\simeq \mathcal{EM}_n(\mathcal{C})
    \]
    (\Cref{em in infinity topos}) guarantees that the above functor is fully faithful. Moreover, by the definition of $n\mbox{-}\mathrm{Grp}^L$, the image of ${\mathrm{B}^n_\mathcal{C}}_{|n\mbox{-}\mathrm{Grp}^L}$ consists of $L$-local objects and hence $n$-Eilenberg-MacLane objects in $L\mathcal{C}$ by \Cref{em of localization}. (Observe that the inverse is given by ${\pi_n^\mathcal{C}}_{|\mathcal{EM}_n(L\mathcal{C})}$, i.e., $\pi_n^\mathcal{C}$ has $L$-local $n$-groups as images. Therefore, the inverse is identical to $\pi_n^{L\mathcal{C}}$.)
\end{proof}

\begin{cor}\label[cor]{not a topos localization}
    Suppose $L$ is a path-injective (and path-cartesian) localization of an $\infty$-topos such that the inclusion $n\mbox{-}\mathrm{Grp}^L\subset n\mbox{-}\mathrm{Grp}(\mathrm{Disc}(L\mathcal{C}))$ is strict for some $n$. Then $L\mathcal{C}$ is not an $\infty$-topos.
\end{cor}
\begin{proof}
    If $L\mathcal{C}$ is an $\infty$-topos, then by \Cref{em in infinity topos}, the functor $$\pi_n^{L\mathcal{C}}:\mathcal{EM}_n(L\mathcal{C})\to n\mbox{-}\mathrm{Grp}(\mathrm{Disc}(L\mathcal{C}))$$ must be an equivalence. But by the above proposition, this would imply $n\mbox{-}\mathrm{Grp}^L=n\mbox{-}\mathrm{Grp}(\mathrm{Disc}(L\mathcal{C}))$.
\end{proof}
\subsection{Covering theory and localizations}

\subsubsection{Spatial \texorpdfstring{$n$}{n}-coverings (\texorpdfstring{$L$}{L}-local \texorpdfstring{$n$}{n}-truncated morphisms)}
We now generalize the notion of truncated objects to its relative version:

\begin{defn}\label[defn]{defn for covering}
Let $\mathcal{C}$ be a presentable $\infty$-category and let $n \geq -2$. A morphism $f \colon Y \to X$ in $\mathcal{C}$ is called a \emph{$\mathcal{C}$-$n$-covering} (an equivalence when $n=-2$ \textup{[\cite{lurie2009higher}, Example 5.5.6.13]}, a \emph{monomorphism} when $n=-1$, and simply a \emph{$\mathcal{C}$-covering} when $n=0$) if $f \in (\mathcal{C}_{/X})_{\leq n}$, i.e., $f \to \tau_{\leq n}^{\mathcal{C}_{/X}} f$ is an equivalence. We denote the full subcategory of $n$-coverings of $X$ by $Cov^n_{\mathcal{C}}(X) := (\mathcal{C}_{/X})_{\leq n}$ (and write $Cov_{\mathcal{C}}(X) := Cov^0_{\mathcal{C}}(X)$). For a pointed object $X_* = (x \colon * \to X) \in \mathcal{C}_*$, we similarly define $Cov^n_{\mathcal{C}}(X)_* := Cov^n_{\mathcal{C}_*}(X_*)$.
\end{defn}

\begin{lem}\label[lem]{cov is presentable}
Let $\mathcal{C}$ be a presentable $\infty$-category and let $n \geq -2$. Then $Cov^n_{\mathcal{C}}(X)$ is a presentable $\infty$-category formed as an accessible localization of $\mathcal{C}_{/X}$, with left adjoint $\tau_{\leq n}^{\mathcal{C}_{/X}} \colon \mathcal{C}_{/X} \to Cov^n_{\mathcal{C}}(X)$. Similarly, for a pointed object $X_* = (x \colon * \to X)$ with $\mathcal{C}_{*/X} := (\mathcal{C}_*)_{/X_*}$,
\[
    Cov^n_{\mathcal{C}}(X)_* = (\mathcal{C}_{*/X})_{\leq n} \simeq Cov^n_{\mathcal{C}}(X) \times_{\mathcal{C}_{/X}} \mathcal{C}_{*/X}
\]
is presentable, arising as an accessible localization of $\mathcal{C}_{*/X}$. In particular, $Cov^n_{\mathcal{C}}(X)$ and $Cov^n_{\mathcal{C}}(X)_*$ have all limits and colimits.
\end{lem}
\begin{proof}
Presentability of both categories follows from [\cite{lurie2009higher}, Proposition 5.5.6.18]. The identification $(\mathcal{C}_{*/X})_{\leq n} \simeq Cov^n_{\mathcal{C}}(X) \times_{\mathcal{C}_{/X}} \mathcal{C}_{*/X}$ is an easy exercise using the definitions and the forgetful right adjoint $\mathcal{C}_{*/X}\to \mathcal{C}_{/X}$.
\end{proof}

\begin{rem}\label[rem]{initial object of cov in a topos}
Assuming the initial object $\emptyset \in \mathcal{C}$ is strict, the unique map $\emptyset \to X$ is already $n$-truncated for any $n \geq -1$ and is thus trivially an initial object in $Cov^n_{\mathcal{C}}(X)$ (whereas for $n = -2$, $Cov^{-2}_{\mathcal{C}}(X) \simeq *$ has $\mathrm{id}_X \colon X \to X$ as its unique object). In the pointed setting, however, the basepoint $* \to X$ is generally no longer an $n$-covering for $n \geq -2$. Instead, because $* \to X$ is initial in $\mathcal{C}_{*/X}$ and the left adjoint $\tau_{\leq n}^{\mathcal{C}_{*/X}} \colon \mathcal{C}_{*/X} \to Cov^n_{\mathcal{C}}(X)_*$ preserves colimits, $Cov^n_{\mathcal{C}}(X)_*$ admits a nontrivial initial object given by $\tau_{\leq n}^{\mathcal{C}_{*/X}}(* \to X)$ for every $n \geq -2$. When $\mathcal{C}$ is an $\infty$-topos and $n \geq -2$, computing homotopy groups identifies this initial pointed $n$-covering with the $(n+1)$-connected cover $\tau_{>n+1}^{\mathcal{C}} X \to X$ (which for $n = -2$ is $\mathrm{id}_X \colon X \to X$, since $\tau_{\leq -1}^{\mathcal{C}} X \simeq *$ for pointed $X$), sitting in the fiber sequence
\[
    \tau_{>n+1}^{\mathcal{C}} X \to X \to \tau_{\leq n+1}^{\mathcal{C}} X.
\]
\end{rem}

\begin{prop}\label[prop]{Covering of Localization}
Let $L\mathcal{C} \subset \mathcal{C}$ be an accessible localization of a presentable $\infty$-category and let $n \geq -2$. A morphism $Y \to X$ in $L\mathcal{C}$ is an $n$-covering in $L\mathcal{C}$ if and only if it is an $n$-covering in $\mathcal{C}$, so that
\[
    (L\mathcal{C}_{/X})_{\leq n} = (\mathcal{C}_{/X})_{\leq n} \mcap L\mathcal{C}_{/X}.
\]
\end{prop}
\begin{proof}
For $X \in L\mathcal{C}$, the limit of any diagram $D \colon \mathcal{I} \to L\mathcal{C}_{/X}$ in $\mathcal{C}_{/X}$ is computed by the limit in $\mathcal{C}$ of the extended cone $D^\triangleright \colon \mathcal{I}^\triangleright \to \mathcal{C}$ sending the cone point to $X$. Since $L\mathcal{C} \subset \mathcal{C}$ is closed under limits, $L\mathcal{C}_{/X} \subset \mathcal{C}_{/X}$ is closed under limits. As both slice categories are presentable, $L\mathcal{C}_{/X} \subset \mathcal{C}_{/X}$ is a localization, so the claim follows from \Cref{truncation of localization}.
\end{proof}

\begin{lem}\label[lem]{covering lemma}
Let $\mathcal{C}$ be an $\infty$-topos, $L$ an accessible localization of $\mathcal{C}$, and $n \geq -2$.
\begin{enumerate}
    \item Any morphism between $L\mathcal{C}$-$n$-coverings of $X$ is an $L\mathcal{C}$-$n$-covering.
    \item The composition of $L\mathcal{C}$-$n$-coverings is an $L\mathcal{C}$-$n$-covering.
    \item For any morphism $f \colon Y \to X$ in $\mathcal{C}$, pullback along $Lf \colon LY \to LX$ carries $L\mathcal{C}$-$n$-coverings to $L\mathcal{C}$-$n$-coverings.
    \item If $L$ is effective on $X$, the pullback functor $\eta_X^* \colon \mathcal{C}_{/LX} \to \mathcal{C}_{/X}$ along $\eta_X \colon X \to LX$ preserves and detects $\mathcal{C}$-$n$-coverings.
    \item If $L$ is effective on $X$, a morphism $Z \to LX$ is an $L\mathcal{C}$-$n$-covering if and only if $Z$ is $L$-local and $Z \times_{LX} X \to X$ is a $\mathcal{C}$-$n$-covering.
\end{enumerate}
\end{lem}
\begin{proof}
For (1) and (2), \Cref{Covering of Localization} reduces the claim to $L = \mathrm{id}_{\mathcal{C}}$, where it follows by applying [\cite{lurie2009higher}, Lemma 5.5.6.14] in the slice category. For (3), pullback along $Lf$ is left-exact and thus preserves $n$-truncated morphisms by [\cite{lurie2009higher}, Proposition 5.5.6.16], while $L\mathcal{C}$ is closed under fiber products in $\mathcal{C}$. For (4), since $\eta_X \colon X \to LX$ is an effective epimorphism, [\cite{lurie2009higher}, Proposition 6.2.3.17] together with (3) implies that $\eta_X^*$ preserves and detects $n$-coverings. Part (5) follows immediately from (4) and \Cref{Covering of Localization}.
\end{proof}

\begin{rem}\label[rem]{covering pullback description}
When $L$ is effective on $X$, \Cref{covering lemma}(5) gives a pullback square of $\infty$-categories for every $n \geq -2$:
\[
\xymatrix{
Cov^n_{L\mathcal{C}}(LX) \ar@{^(->}[r] \ar[d] & L\mathcal{C}_{/LX} \ar[d] \\
Cov^n_{\mathcal{C}}(X) \ar@{^(->}[r] & \mathcal{C}_{/X}.
}
\]
\end{rem}

\begin{lem}\label[lem]{covering if and only if fiber is discrete}
Let $X$ be a pointed, connected object of an $\infty$-topos $\mathcal{C}$, and let $n \geq -2$. A morphism $f \colon Z \to X$ in $\mathcal{C}$ is an $n$-covering if and only if $\mathrm{fib}(f) \in \mathcal{C}_{\leq n}$.
\end{lem}
\begin{proof}
Since $X$ is connected, the basepoint map $* \to X$ is an effective epimorphism, so [\cite{lurie2009higher}, Proposition 6.2.3.17] implies that $f$ is $n$-truncated if and only if $\mathrm{fib}(f) \to *$ is $n$-truncated.
\end{proof}
\begin{defn}\label[defn]{defn extension closed}
Let $B \in L\mathcal{C}_*$ be a pointed $L$-local object. We say that $L\mathcal{C}$ is \emph{closed under extensions over $B$} (or \emph{extension-closed over $B$}) if for any morphism $f \colon Z \to B$ in $\mathcal{C}$ whose fiber $\mathrm{fib}(f) := Z \times_B *$ lies in $L\mathcal{C}$, we have $Z \in L\mathcal{C}$.
\end{defn}
\begin{cor}\label[cor]{L covering if and only if fiber is L discrete}
Suppose $L\mathcal{C}$ is closed under extensions over $LX$. For $LX$ pointed and connected in $\mathcal{C}$ and $n \geq -2$, a morphism $f \colon Z \to LX$ is an $L\mathcal{C}$-$n$-covering if and only if $\mathrm{fib}(f) \in (L\mathcal{C})_{\leq n}$.
\end{cor}
\begin{proof}
Extension-closure over $LX$ ensures that $Z$ is $L$-local whenever $\mathrm{fib}(f) \in (L\mathcal{C})_{\leq n}$, so the result follows from \Cref{Covering of Localization} and \Cref{covering if and only if fiber is discrete}.
\end{proof}

\begin{defn}\label[defn]{definition for L covering}
Let $\mathcal{C}$ be a presentable $\infty$-category, $L$ an accessible localization, $X \in \mathcal{C}$, and $n \geq -2$. The category of \emph{$L$-fibrations over $X$}, denoted $L\mathcal{F}ib_X$, is the full subcategory of $\mathcal{C}_{/X}$ spanned by objects that are local with respect to $L$-equivalences in $\mathcal{C}_{/X}$. The $\infty$-category of \emph{$L$-$n$-coverings} (equivalences over $X$ when $n=-2$, and \emph{$L$-monomorphisms} when $n=-1$) of $X$ is $Cov^n_L(X) := (L\mathcal{F}ib_X)_{\leq n}$ (with $Cov_L(X) := Cov^0_L(X)$). When $X$ is pointed, we set $Cov^n_L(X)_* := Cov^n_L(X) \times_{\mathcal{C}_{/X}} \mathcal{C}_{*/X}$.
\end{defn}

\begin{rem}
Because the forgetful functor $\mathcal{C}_{/X} \to \mathcal{C}$ preserves and creates colimits, the class of $L$-equivalences in $\mathcal{C}_{/X}$ is strongly saturated. Hence $L\mathcal{F}ib_X \subset \mathcal{C}_{/X}$ is an accessible localization, with reflector denoted $L^{fib}$, and for every $n \geq -2$, \Cref{truncation of localization} yields
\[
    Cov^n_L(X) = (L\mathcal{F}ib_X)_{\leq n} = (\mathcal{C}_{/X})_{\leq n} \mcap L\mathcal{F}ib_X,
\]
which is an accessible localization of the presentable $\infty$-category $\mathcal{C}_{/X}$ and is hence presentable (and similarly for $Cov^n_L(X)_* \simeq ((L\mathcal{F}ib_X)_{L^{fib}(x)/})_{\leq n}$ inside $\mathcal{C}_{*/X}$ when $x \colon * \to X$ is a basepoint). Note that for $n = -2$, since $\mathrm{id}_X \in L\mathcal{F}ib_X$, we simply have $Cov^{-2}_L(X) = (\mathcal{C}_{/X})_{\leq -2} \simeq *$.
\end{rem}

\begin{rem}
In general, $L^{fib}$ need not preserve truncated morphisms, just like the absolute case. Dually, the slice truncation $\tau_{\leq n}^{\mathcal{C}_{/X}}$ need not preserve $L$-fibrations in general.
\end{rem}

\begin{lem}\label[lem]{L fib over L local object}
A morphism $f \colon X \to Y$ over an $L$-local object $Y$ is an $L$-fibration if and only if $X$ is $L$-local. In particular, any morphism between $L$-local objects is an $L$-fibration.
\end{lem}
\begin{proof}
Factor $f$ via its $L^{fib}$-localization as
\[
\xymatrix{
X \ar[r]^-{e} \ar[dr]_f & \mathrm{dom}(L^{fib}f) \ar[d]^{L^{fib}f} \\
& Y.
}
\]
Since $Y$ is $L$-local and $L^{fib}f$ is an $L$-fibration, $\mathrm{dom}(L^{fib}f)$ is $L$-local. If $X$ is $L$-local, the $L$-equivalence $e$ is a morphism between $L$-local objects and thus an equivalence, so $f \simeq L^{fib}f$ is an $L$-fibration. Conversely, if $f$ is an $L$-fibration, then $e$ is an equivalence, so $X \simeq \mathrm{dom}(L^{fib}f)$ is $L$-local.
\end{proof}

\begin{lem}\label[lem]{lemma for L covering}
In any presentable $\infty$-category $\mathcal{C}$ with an accessible localization $L$ and $n \geq -2$:
\begin{enumerate}
    \item The pullback of an $L$-fibration (resp.\ $L$-$n$-covering) is an $L$-fibration (resp.\ $L$-$n$-covering).
    \item The composition of $L$-fibrations (resp.\ $L$-$n$-coverings) is an $L$-fibration (resp.\ $L$-$n$-covering).
    \item Any morphism over $X$ between $L$-fibrations (resp.\ $L$-$n$-coverings) is an $L$-fibration (resp.\ $L$-$n$-covering).
\end{enumerate}
\end{lem}
\begin{proof}
For (1), given $f \colon Y \to X$, the left adjoint $f_! \colon \mathcal{C}_{/Y} \to \mathcal{C}_{/X}$ preserves $L$-equivalences, so its right adjoint $f^*$ preserves $L$-fibrations; moreover, $f^*$ is left-exact and thus preserves $n$-truncated objects [\cite{lurie2009higher}, Proposition 5.5.6.16]. Part (2) follows directly from the definition of $L$-fibrations and [\cite{lurie2009higher}, Lemma 5.5.6.14]. For (3), given $f_1 \colon Y_1 \to X$ and $f_2 \colon Y_2 \to X$ in $L\mathcal{F}ib_X$ and a morphism $g \colon Y_1 \to Y_2$ over $X$, applying \Cref{L fib over L local object} to the accessible localization $L^{fib} \colon \mathcal{C}_{/X} \to L\mathcal{F}ib_X$ over the $L^{fib}$-local object $f_2$ shows that $g \in L\mathcal{F}ib_{Y_2}$, while [\cite{lurie2009higher}, Lemma 5.5.6.14] in $\mathcal{C}_{/X}$ shows that $g$ is $n$-truncated whenever $f_1$ and $f_2$ are.
\end{proof}

\begin{prop}\label[prop]{L covering of L local object}
For any $L$-local object $X \in L\mathcal{C}$ and $n \geq -2$, we have $Cov^n_L(X) = Cov^n_{L\mathcal{C}}(X)$ (and $Cov^n_L(X)_* = Cov^n_{L\mathcal{C}}(X)_*$ when $X$ is pointed).
\end{prop}
\begin{proof}
This follows immediately from \Cref{Covering of Localization}, \Cref{L fib over L local object}, and \Cref{cov is presentable}.
\end{proof}

For the remainder of this subsection, assume that $\mathcal{C}$ is an $\infty$-topos and $L$ is an accessible localization of $\mathcal{C}$.

\begin{cor}\label[cor]{torsors are covering}
Let $n \geq 0$, let $1 \leq m \leq n+1$, and let $G \in m\mbox{-}\mathrm{Grp}^L$. Then any principal $\mathrm{B}^{m-1}_{\mathcal{C}}G$-bundle $Y \to X$ (i.e., a morphism pulled back from $* \to \mathrm{B}^m_{\mathcal{C}}G$; in particular, for $m=1$, any $G$-torsor \textup{[\cite{MR3423073}, \S 3]}) is an $L$-$n$-covering.
\end{cor}
\begin{proof}
Since $G \in m\mbox{-}\mathrm{Grp}^L$, $\mathrm{B}^m_{\mathcal{C}}G$ is $L$-local and connected with fiber $\mathrm{B}^{m-1}_{\mathcal{C}}G \in (L\mathcal{C})_{\leq m-1} \subset \mathcal{C}_{\leq n}$, so the morphism $* \to \mathrm{B}^m_{\mathcal{C}}G$ between $L$-local objects is an $L\mathcal{C}$-$n$-covering by \Cref{Covering of Localization} and \Cref{covering if and only if fiber is discrete}, and hence an $L$-$n$-covering by \Cref{L covering of L local object}. Since $Y \to X$ is pulled back from $* \to \mathrm{B}^m_{\mathcal{C}}G$, \Cref{lemma for L covering}(1) yields the result.
\end{proof}

\begin{prop}\label[prop]{L covering is covering of L}
Suppose $L$ is locally cartesian on $X$. Then pullback along $\eta_X \colon X \to LX$ induces equivalences of $\infty$-categories with inverse $L$:
\begin{enumerate}
    \item $\eta_X^* \colon L\mathcal{C}_{/LX} = L\mathcal{F}ib_{LX} \xrightarrow{\sim} L\mathcal{F}ib_X$ (and $\eta_X^* \colon L\mathcal{C}_{*/LX} \xrightarrow{\sim} (L\mathcal{F}ib_X)_*$ when $X$ is pointed).
    \item If $L$ is also effective on $X$, then for every $n \geq -2$, $\eta_X^*$ restricts to an equivalence
    \[
        \eta_X^* \colon Cov^n_{L\mathcal{C}}(LX) = Cov^n_L(LX) \xrightarrow{\sim} Cov^n_L(X)
    \]
    (and $\eta_X^* \colon Cov^n_{L\mathcal{C}}(LX)_* \xrightarrow{\sim} Cov^n_L(X)_*$ when $X$ is pointed).
\end{enumerate}
\end{prop}
\begin{proof}
Given $f \colon Y \to X$ in $L\mathcal{F}ib_X$, consider the commutative diagram
\[
\xymatrix{
Y \ar[r] \ar[dr]_f & LY \times_{LX} X \ar[r] \ar[d] & X \ar[d]^{\eta_X} \\
& LY \ar[r]^{Lf} & LX.
}
\]
Because $L$ is locally cartesian on $X$, the map $Y \to LY \times_{LX} X$ is an $L$-equivalence. By \Cref{L fib over L local object} and \Cref{lemma for L covering}(1), $LY \times_{LX} X \to X$ is an $L$-fibration; since $f$ is also an $L$-fibration, $Y \to LY \times_{LX} X$ is an equivalence. Thus $L \colon L\mathcal{F}ib_X \to L\mathcal{C}_{/LX}$ satisfies $\eta_X^* L \simeq \mathrm{id}$, while $L \eta_X^* \simeq \mathrm{id}$ holds by local cartesianity of $L$ on $X$, proving (1). For (2), when $L$ is effective on $X$, \Cref{covering lemma}(5) implies that $Z \in L\mathcal{C}_{/LX}$ lies in $Cov^n_{L\mathcal{C}}(LX)$ if and only if $\eta_X^* Z \in Cov^n_L(X)$, so (1) restricts to (2). Finally, when $x \colon * \to X$ is a basepoint, $\eta_{X!}(x) = Lx$ gives $\mathrm{Map}_{\mathcal{C}_{/X}}(x, \eta_X^* Z) \simeq \mathrm{Map}_{L\mathcal{C}_{/LX}}(Lx, Z)$ naturally in $Z \in L\mathcal{C}_{/LX}$, lifting the equivalence to the pointed setting.
\end{proof}

\begin{rem}[Relation to fundamental $m$-groupoids]\label[rem]{rem fundamental groupoid}
Classically, for $m \geq 1$, the truncation $\tau_{\leq m}^{\mathcal{C}} Y$ of an object $Y \in \mathcal{C}$ is viewed as its \emph{fundamental $m$-groupoid} (and its \emph{fundamental groupoid} when $m=1$). Under an accessible localization $L \colon \mathcal{C} \to L\mathcal{C}$, one compares the external fundamental $m$-groupoid $\tau_{\leq m}^{\mathcal{C}} LY$ (such as Morel's $\mathbb{A}^1$-fundamental groupoid $\tau_{\leq 1}^{nis} L_{mot}\esc{X}$ when $L = L_{mot}$ and $m=1$) with its internal counterpart $\tau_{\leq m}^{L\mathcal{C}} LY \simeq L_{\leq m}\tau_{\leq m}^{\mathcal{C}} LY$. This motivates calling $L$ \emph{$m$-fundamental on $X$} in \Cref{defn m fundamental} (or simply \emph{fundamental on $X$} when $m=1$) when $\tau_{\leq m}^{\mathcal{C}} LX \simeq \tau_{\leq m}^{L\mathcal{C}} LX$. We say that $L$ is \emph{$m$-fundamental} if it is $m$-fundamental on every object.
\end{rem}

\begin{lem}\label[lem]{fundamental loop pi}
Suppose $L$ is $(n+1)$-fundamental on a pointed object $y \colon * \to Y$ for some $n \geq 0$. Then
\[
    \tau_{\leq n}^{L\mathcal{C}}\Omega(LY, Ly) \simeq \tau_{\leq n}^{\mathcal{C}}\Omega(LY, Ly) \simeq \Omega\tau_{\leq n+1}^{\mathcal{C}}(LY, Ly) \simeq \Omega\tau_{\leq n+1}^{L\mathcal{C}}(LY, Ly),
\]
and consequently
\[
    \pi_{n+1}^L(Y,y) = \pi_{n+1}^{\mathcal{C}}(LY, Ly) \simeq \pi_{n+1}^{L\mathcal{C}}(LY, Ly)
\]
(so that $\pi_i^L(Y,y) \simeq \pi_i^{L\mathcal{C}}(LY, Ly)$ for all $1 \leq i \leq n+1$ whenever $L$ is $i$-fundamental on $Y$ for $1 \leq i \leq n+1$).
\end{lem}
\begin{proof}
Since $\mathcal{C}$ is an $\infty$-topos and $(n+1)$-fundamentality gives $\tau_{\leq n+1}^{\mathcal{C}} LY \simeq \tau_{\leq n+1}^{L\mathcal{C}} LY$, we have $\tau_{\leq n}^{\mathcal{C}}\Omega(LY,Ly) \simeq \Omega \tau_{\leq n+1}^{\mathcal{C}}(LY, Ly) \simeq \Omega \tau_{\leq n+1}^{L\mathcal{C}}(LY, Ly)$. Because $L\mathcal{C}$ is closed under loop spaces, this object already lies in $L\mathcal{C} \mcap \mathcal{C}_{\leq n} = (L\mathcal{C})_{\leq n}$, whence $\tau_{\leq n}^{L\mathcal{C}}\Omega(LY, Ly) \simeq L_{\leq n}\tau_{\leq n}^{\mathcal{C}}\Omega(LY, Ly) \simeq \tau_{\leq n}^{\mathcal{C}}\Omega(LY, Ly)$. Taking $n$-fold loops then yields $\pi_{n+1}^{\mathcal{C}}(LY, Ly) \simeq \Omega^n\tau_{\leq n}^{\mathcal{C}}\Omega(LY, Ly) \in (L\mathcal{C})_{\leq 0}$, so $\pi_{n+1}^{\mathcal{C}}(LY, Ly) \simeq L_{\leq 0}\pi_{n+1}^{\mathcal{C}}(LY, Ly) \simeq \pi_{n+1}^{L\mathcal{C}}(LY, Ly)$.
\end{proof}

If $Y \in \mathcal{C}$ admits a global section $y \colon * \to Y$, then for every $n \geq -2$, \Cref{covering lemma}(3) shows that pullback along $Ly$ induces a fiber functor $y^* \colon Cov^n_{L\mathcal{C}}(LY) \to Cov^n_{L\mathcal{C}}(*) = (L\mathcal{C})_{\leq n}$. For $G \in \mathrm{Grp}(\mathcal{C})$ and a full subcategory $\mathcal{D} \subseteq \mathcal{C}$, we write $Act_G(\mathcal{D}) := Act_G(\mathcal{C}) \times_{\mathcal{C}} \mathcal{D}$ for the $\infty$-category of objects of $\mathcal{D}$ equipped with a coherent $G$-action in $\mathcal{C}$.

\begin{lem}\label[lem]{action of the fundamental groupoid}
Let $n \geq 0$ and let $(Y,y) \in \mathcal{C}_*$ be a pointed object. Then the fiber functor $y^*$ on $Cov^n_{L\mathcal{C}}(LY)$ factors through a canonical functor
\[
    \Gamma_y^{L,n} \colon Cov^n_{L\mathcal{C}}(LY) \longrightarrow Act_{\tau_{\leq n}^{\mathcal{C}}\Omega(LY,Ly)}\bigl((L\mathcal{C})_{\leq n}\bigr).
\]
\end{lem}
\begin{proof}
We may easily reduce to $L = \mathrm{id}_{\mathcal{C}}$. Applying $\mathrm{Map}_{\mathcal{C}}(W, -)$ and using the full faithfulness of the Yoneda embedding $\mathcal{C} \to \mathrm{Fun}(\mathcal{C}^{op},\mathrm{Spc})$ reduces the claim to spaces, where $\Omega(Y,y)$ acts coherently on the $n$-truncated fiber $y^*Z$ for any $(Z \to Y) \in Cov^n_{\mathcal{C}}(Y)$ and thus factors through an action of $\tau_{\leq n}(\Omega(Y,y))$.
\end{proof}

\begin{prop}\label[prop]{L galois theory}
Let $\mathcal{C}$ be an $\infty$-topos, $L$ an accessible localization of $\mathcal{C}$, and $n \geq 0$. Suppose $(Y,y) \in \mathcal{C}_*$ is a pointed object such that $Y$ is $L$-connected and $L\mathcal{C}$ is closed under extensions over $LY$. Then the fiber functor
\[
    \Gamma_y^{L,n} \colon Cov^n_{L\mathcal{C}}(LY) \longrightarrow Act_{\tau_{\leq n}^{\mathcal{C}}\Omega(LY,Ly)}\bigl((L\mathcal{C})_{\leq n}\bigr)
\]
is an equivalence of $\infty$-categories. When $L$ is effective and locally cartesian on $Y$, this extends via \Cref{L covering is covering of L} to an equivalence $Cov^n_L(Y) \simeq Act_{\tau_{\leq n}^{\mathcal{C}}\Omega(LY,Ly)}\bigl((L\mathcal{C})_{\leq n}\bigr)$.
\end{prop}
\begin{proof}
When $L = \mathrm{id}_{\mathcal{C}}$, the inverse sends $F \in Act_{\tau_{\leq n}^{\mathcal{C}}\Omega(Y,y)}(\mathcal{C}_{\leq n})$ (viewed as an $\Omega Y$-module in $\mathcal{C}_{\leq n}$ via restriction along $\Omega Y \to \tau_{\leq n}^{\mathcal{C}}\Omega(Y,y)$) to the morphism $\mathrm{B}_{\mathcal{C}}(F, \Omega Y, *) \to \mathrm{B}_{\mathcal{C}}(*, \Omega Y, *) \simeq Y$ (where $\mathrm{B}_{\mathcal{C}}(F, \Omega Y, *)$ denotes the two-sided bar construction), which has $n$-truncated fiber $F$ and is therefore a $\mathcal{C}$-$n$-covering by \Cref{covering if and only if fiber is discrete} (see [\cite{constantin2026higher}] for a detailed proof).

For general $L$, applying the $L = \mathrm{id}_{\mathcal{C}}$ case to the pointed connected object $(LY, Ly)$ yields an equivalence $$\Gamma_y^{\mathcal{C},n} \colon Cov^n_{\mathcal{C}}(LY) \xrightarrow{\sim} Act_{\tau_{\leq n}^{\mathcal{C}}\Omega(LY,Ly)}(\mathcal{C}_{\leq n}),$$ with inverse sending $F$ to $\mathrm{B}_{\mathcal{C}}(F, \Omega LY, *) \to LY$. Since $L\mathcal{C}$ is closed under extensions over $LY$, \Cref{L covering if and only if fiber is L discrete} implies that a $\mathcal{C}$-$n$-covering of $LY$ belongs to the full subcategory $Cov^n_{L\mathcal{C}}(LY) \subseteq Cov^n_{\mathcal{C}}(LY)$ if and only if its fiber $F$ belongs to $(L\mathcal{C})_{\leq n} = \mathcal{C}_{\leq n} \mcap L\mathcal{C}$. Thus $\Gamma_y^{\mathcal{C},n}$ restricts to the desired equivalence $\Gamma_y^{L,n}$.
\end{proof}

\begin{prop}\label[prop]{Universal cover for fundamental localization}
Let $n \geq -2$ and let $X \in \mathcal{C}_*$ be a pointed object.
\begin{enumerate}
    \item $Cov^n_{L\mathcal{C}}(LX)_*$ admits an initial object $$\tau^{L\mathcal{C}}_{>n+1}{X} := \tau_{\leq n}^{L\mathcal{C}_{*/LX}}(* \to LX) \to LX,$$ satisfying $\tau_{\leq n+1}^{L\mathcal{C}}\tau^{L\mathcal{C}}_{>n+1}{X} \simeq *$. If $L$ is locally cartesian and effective on $X$, then $$\tau_{>n+1}^L {X} := \eta_X^*\tau^{L\mathcal{C}}_{>n+1}{X} \to X$$ is initial in $Cov^n_L(X)_*$ and satisfies $\tau_{\leq n+1}^{L\mathcal{C}}L\tau_{>n+1}^L {X} \simeq *$.
    \item Suppose $L$ is $(n+1)$-fundamental on $X$. Then the inclusion $Cov^n_{L\mathcal{C}}(LX)_* \subset Cov^n_{\mathcal{C}}(LX)_*$ preserves the initial object and is thus uniquely determined as $$(Z \to LX) \in Cov^n_{L\mathcal{C}}(LX)_*\quad \text{with}\quad \tau_{\leq n+1}^{\mathcal{C}} Z \simeq *$$ (equivalently, $\tau_{\leq n+1}^{L\mathcal{C}} Z \simeq *$ when $L$ is path-injective and $j$-fundamental on $(j-1)$-connected objects for $1 \leq j \leq n+1$). If $L$ is locally cartesian and effective on $X$, then $\tau_{>n+1}^L {X} \to X$ is the unique object $$(Y \to X) \in Cov^n_L(X)_*\quad  \text{with}\quad \tau_{\leq n+1}^{\mathcal{C}} LY \simeq *.$$ 
\end{enumerate}
\end{prop}
\begin{proof}
For (1), $\tau^{L\mathcal{C}}_{>n+1}{X}$ is initial in $Cov^n_{L\mathcal{C}}(LX)_*$ by \Cref{initial object of cov in a topos} applied to $L\mathcal{C}$. Set $W := \tau^{L\mathcal{C}}_{>n+1}{X}$ and $Z := \tau_{\leq n+1}^{L\mathcal{C}} W$, pointed by $z \colon * \to W \to Z$. Since $Z \in (L\mathcal{C})_{\leq n+1}$, its diagonal is $n$-truncated [\cite{lurie2009higher}, Lemma 5.5.6.15], and hence so are its pullbacks $z \colon * \to Z$ and $F := W \times_Z * \to W$ (\Cref{covering lemma}(3)). By \Cref{covering lemma}(2), the composite $F \to W \to LX$ lies in $Cov^n_{L\mathcal{C}}(LX)_*$; initiality of $W$ gives a section of $F \to W$, so the truncation unit $W \to Z$ factors through $*$, forcing $Z \simeq *$. The claims for $\tau_{>n+1}^L {X}$ follow via \Cref{L covering is covering of L}.

For (2), in the fiber sequence $$\tau_{>n+1}^{\mathcal{C}} LX \to LX \to \tau_{\leq n+1}^{\mathcal{C}} LX,$$ the base $\tau_{\leq n+1}^{\mathcal{C}} LX \simeq \tau_{\leq n+1}^{L\mathcal{C}} LX$ is $L$-local, and hence so is $\tau_{>n+1}^{\mathcal{C}} LX$. Thus the initial object $\tau_{>n+1}^{\mathcal{C}} LX \to LX$ of $Cov^n_{\mathcal{C}}(LX)_*$ (\Cref{initial object of cov in a topos}) lies in $Cov^n_{L\mathcal{C}}(LX)_*$, giving $\tau^{L\mathcal{C}}_{>n+1}{X} \simeq \tau_{>n+1}^{\mathcal{C}} LX$. For any $(Z \to LX) \in Cov^n_{\mathcal{C}}(LX)_*$ with $\tau_{\leq n+1}^{\mathcal{C}} Z \simeq *$, the initial map $f \colon \tau_{>n+1}^{\mathcal{C}} LX \to Z$ defines an object of $Cov^n_{\mathcal{C}}(Z)_*$ (\Cref{covering lemma}(1)), where $\mathrm{id}_Z \simeq \tau_{>n+1}^{\mathcal{C}} Z$ is initial; the resulting pointed section of $f$ is inverse to $f$ by initiality of $\tau_{>n+1}^{\mathcal{C}} LX$ (and \Cref{tau preserves connectivity} gives the equivalence with $\tau_{\leq n+1}^{L\mathcal{C}} Z \simeq *$). \Cref{L covering is covering of L} then transfers uniqueness to $\tau_{>n+1}^L {X}$. 
\end{proof}

Recall that for any $Y \in \mathcal{C}$, the slice $\infty$-category $\mathcal{C}_{/Y}$ is tensored over $\mathcal{C}$ via $Z \otimes (A \to Y) := (Z \times A \to A \to Y)$. For $A, B \in \mathcal{C}_{/Y}$, we denote by $\underline{\mathrm{Map}}_{\mathcal{C}_{/Y}}(A,B) \in \mathcal{C}$ the enriched mapping object characterized by $$\mathrm{Map}_{\mathcal{C}}(Z, \underline{\mathrm{Map}}_{\mathcal{C}_{/Y}}(A,B)) \simeq \mathrm{Map}_{\mathcal{C}_{/Y}}(Z \times A, B)$$ for all $Z \in \mathcal{C}$, and we set $\underline{\mathrm{End}}_{\mathcal{C}_{/Y}}(A) := \underline{\mathrm{Map}}_{\mathcal{C}_{/Y}}(A,A)$ (with $\underline{\mathrm{Aut}}_{\mathcal{C}_{/Y}}(A) \subset \underline{\mathrm{End}}_{\mathcal{C}_{/Y}}(A)$ the subobject of equivalences). When $A \in Cov^n_{\mathcal{C}}(Y)$ ($n \geq -2$), $\underline{\mathrm{End}}_{\mathcal{C}_{/Y}}(A)$ is $n$-truncated in $\mathcal{C}$.

\begin{lem}\label[lem]{aut=pi_1}
Let $n \geq 0$ and suppose $L$ is cartesian and $(n+1)$-fundamental on a pointed object $X \in \mathcal{C}_*$. Then for the initial pointed $L\mathcal{C}$-$n$-covering $\tau^{L\mathcal{C}}_{>n+1}{X} \simeq \tau_{>n+1}^{\mathcal{C}} LX \to LX$, the canonical map $\Omega\tau_{\leq n+1}^{\mathcal{C}} LX \to \underline{\mathrm{End}}_{L\mathcal{C}_{/LX}}(\tau^{L\mathcal{C}}_{>n+1}{X})$ in $\mathcal{C}$ is an equivalence, inducing further equivalences in $\mathcal{C}$:
\begin{align*}
    \tau_{\leq n}^{L\mathcal{C}}\Omega LX \simeq \tau_{\leq n}^{\mathcal{C}}\Omega LX
    &\simeq \Omega\tau_{\leq n+1}^{L\mathcal{C}} LX \simeq \Omega\tau_{\leq n+1}^{\mathcal{C}} LX \\
    &\simeq \underline{\mathrm{Aut}}_{L\mathcal{C}_{/LX}}(\tau^{L\mathcal{C}}_{>n+1}{X}) \simeq \underline{\mathrm{End}}_{L\mathcal{C}_{/LX}}(\tau^{L\mathcal{C}}_{>n+1}{X}) \\
    &\simeq \underline{\mathrm{Aut}}_{\mathcal{C}_{/LX}}(\tau^{L\mathcal{C}}_{>n+1}{X}) \simeq \underline{\mathrm{End}}_{\mathcal{C}_{/LX}}(\tau^{L\mathcal{C}}_{>n+1}{X}).
\end{align*}
If $L$ is moreover locally cartesian on $X$, then these are further equivalent to $\underline{\mathrm{Aut}}_{\mathcal{C}_{/X}}(\tau_{>n+1}^L {X}) \simeq \underline{\mathrm{End}}_{\mathcal{C}_{/X}}(\tau_{>n+1}^L {X})$. In particular, when $X$ is $L$-locally $n$-connected, these automorphism objects are equivalent to $\mathrm{B}^n_{L\mathcal{C}}\pi_{n+1}^{L\mathcal{C}} LX \simeq \mathrm{B}^n_{\mathcal{C}}\pi_{n+1}^L X \simeq \mathrm{B}^n_{\mathcal{C}}\pi_{n+1}^{L\mathcal{C}} LX$.
\end{lem}
\begin{proof}
Set $W := \tau^{L\mathcal{C}}_{>n+1}{X} \simeq \tau_{>n+1}^{\mathcal{C}} LX$ (\Cref{Universal cover for fundamental localization}(2)). Since $W \in L\mathcal{C}_{/LX}$ and $L$ is cartesian, for any $Z \in \mathcal{C}$ we have
\begin{flalign*}
    \mathrm{Map}_{\mathcal{C}}(Z, \underline{\mathrm{End}}_{L\mathcal{C}_{/LX}}(W)) \simeq \mathrm{Map}_{L\mathcal{C}_{/LX}}(LZ \times W, W)
    &\simeq \mathrm{Map}_{L\mathcal{C}_{/LX}}(L(Z \times W), W) \\
    &\simeq \mathrm{Map}_{\mathcal{C}_{/LX}}(Z \times W, W) \\
    &\simeq \mathrm{Map}_{\mathcal{C}}(Z, \underline{\mathrm{End}}_{\mathcal{C}_{/LX}}(W)),
\end{flalign*}
so $\underline{\mathrm{End}}_{L\mathcal{C}_{/LX}}(W) \simeq \underline{\mathrm{End}}_{\mathcal{C}_{/LX}}(W)$. Furthermore, when $L$ is locally cartesian on $X$, setting $W^L := \tau_{>n+1}^L {X} = \eta_X^* W = X \times_{LX} W$, the map
\[
    Z \times W^L \simeq (Z \times X) \times_{LX} W \to (LZ \times LX) \times_{LX} W \simeq LZ \times W
\]
is an $L$-equivalence, whence
\begin{flalign*}
    \mathrm{Map}_{\mathcal{C}}(Z, \underline{\mathrm{End}}_{\mathcal{C}_{/X}}(W^L)) \simeq \mathrm{Map}_{\mathcal{C}_{/X}}(Z \times W^L, W^L)
    &\simeq \mathrm{Map}_{\mathcal{C}_{/LX}}(Z \times W^L, W) \\
    &\simeq \mathrm{Map}_{L\mathcal{C}_{/LX}}(L(Z \times W^L), W) \\
    &\simeq \mathrm{Map}_{L\mathcal{C}_{/LX}}(LZ \times W, W) \\
    &\simeq \mathrm{Map}_{\mathcal{C}}(Z, \underline{\mathrm{End}}_{L\mathcal{C}_{/LX}}(W)).
\end{flalign*}

Thus we reduce to the case $L = \mathrm{id}_{\mathcal{C}}$ for a pointed object $Y \in \mathcal{C}_*$. Let $\tau_{>n+1}^{\mathcal{C}} Y \xrightarrow{i} Y \xrightarrow{u} \tau_{\leq n+1}^{\mathcal{C}} Y$ be the defining fiber sequence of $\tau_{>n+1}^{\mathcal{C}} Y$ over the basepoint $f \colon * \to \tau_{\leq n+1}^{\mathcal{C}} Y$. Mapping $\tau_{>n+1}^{\mathcal{C}} Y$ into this fiber sequence yields the fiber sequence in $\mathcal{C}$:
\[
    \underline{\mathrm{Map}}_{\mathcal{C}}(\tau_{>n+1}^{\mathcal{C}} Y, \tau_{>n+1}^{\mathcal{C}} Y) \xrightarrow{i \circ -} \underline{\mathrm{Map}}_{\mathcal{C}}(\tau_{>n+1}^{\mathcal{C}} Y, Y) \xrightarrow{u \circ -} \underline{\mathrm{Map}}_{\mathcal{C}}(\tau_{>n+1}^{\mathcal{C}} Y, \tau_{\leq n+1}^{\mathcal{C}} Y) \simeq \tau_{\leq n+1}^{\mathcal{C}} Y
\]
(here the last equivalence is obtained by the composite equivalence
\[
    \underline{\mathrm{Map}}_{\mathcal{C}}(\tau_{>n+1}^{\mathcal{C}} Y, \tau_{\leq n+1}^{\mathcal{C}} Y) \simeq \underline{\mathrm{Map}}_{\mathcal{C}}(\tau_{\leq n+1}^{\mathcal{C}}\tau_{>n+1}^{\mathcal{C}} Y, \tau_{\leq n+1}^{\mathcal{C}} Y) \simeq \underline{\mathrm{Map}}_{\mathcal{C}}(*, \tau_{\leq n+1}^{\mathcal{C}} Y),
\]
which takes $u \circ i$ to $f$). On the other hand, by definition of morphisms in the slice category, we have another fiber sequence
\[
    \underline{\mathrm{End}}_{\mathcal{C}_{/Y}}(\tau_{>n+1}^{\mathcal{C}} Y) \to \underline{\mathrm{Map}}_{\mathcal{C}}(\tau_{>n+1}^{\mathcal{C}} Y, \tau_{>n+1}^{\mathcal{C}} Y) \xrightarrow{i \circ -} \underline{\mathrm{Map}}_{\mathcal{C}}(\tau_{>n+1}^{\mathcal{C}} Y, Y)
\]
over $i = (i \circ -)(\mathrm{id}_{\tau_{>n+1}^{\mathcal{C}} Y})$. Comparing the two fiber sequences we identify $\underline{\mathrm{End}}_{\mathcal{C}_{/Y}}(\tau_{>n+1}^{\mathcal{C}} Y)$ with $\Omega(\tau_{\leq n+1}^{\mathcal{C}} Y, f) \simeq \tau_{\leq n}^{\mathcal{C}}\Omega Y$. Because this endomorphism monoid is equivalent to the grouplike object $\Omega(\tau_{\leq n+1}^{\mathcal{C}} Y, f)$, every endomorphism is an equivalence, so $\underline{\mathrm{Aut}} \simeq \underline{\mathrm{End}}$ in each slice category above, and \Cref{fundamental loop pi} identifies the four loop-truncation terms for $LX$. Finally, when $\tau_{\leq n}^{\mathcal{C}} LX \simeq *$,
\[
    \Omega\tau_{\leq n+1}^{\mathcal{C}} LX \simeq \Omega \mathrm{B}^{n+1}_\mathcal{C}\pi_{n+1}^\mathcal{C}LX\simeq \mathrm{B}^n_{\mathcal{C}}\pi_{n+1}^L X \simeq \mathrm{B}^n_{\mathcal{C}}\pi_{n+1}^{L\mathcal{C}} LX \simeq \mathrm{B}^n_{L\mathcal{C}}\pi_{n+1}^{L\mathcal{C}} LX
\]
(\Cref{bar of cart localization}, \Cref{fundamental loop pi}).
\end{proof}

\begin{thm}[The $n$-th $L$-local Galois correspondence]\label{LC local galois corr}
Let $n \geq 0$ and let $L$ be a path-injective localization of an $\infty$-topos $\mathcal{C}$ that is $(n+1)$-fundamental on $L$-locally $n$-connected objects. For any pointed $L$-locally $n$-connected object $X \in \mathcal{C}_*$ such that $L\mathcal{C}$ is extension-closed over $LX$:
\begin{enumerate}
    \item There are equivalences of posets between $(n+1)$-strictly $L$-local subgroups of $\pi_{n+1}^L(X) \simeq \pi_{n+1}^{L\mathcal{C}}(LX)$ and pointed $L\mathcal{C}$-$n$-coverings of $LX$ with $\mathcal{C}$- (or $L\mathcal{C}$-) $n$-connected total object:
    \begin{flalign*}
        Sub_{(n+1)\mbox{-}\mathrm{Grp}^L}(\pi_{n+1}^L(X)) \simeq Cov^n_{L\mathcal{C}}(LX)_*^{>_\mathcal{C} n} = Cov^n_{L\mathcal{C}}(LX)_*^{>_{L\mathcal{C}} n},
    \\
    \text{where}\quad Cov^n_{L\mathcal{C}}(LX)_*^{>_{\mathcal{D}} n} := Cov^n_{L\mathcal{C}}(LX)_* \times_{\mathcal{D}_*} \mathcal{D}_{>n} \quad \text{for}\quad \mathcal{D} \in \{\mathcal{C}, L\mathcal{C}\}
    \end{flalign*}
    (spanned by $(Y \to LX) \in Cov^n_{L\mathcal{C}}(LX)_*$ with $\tau_{\leq n}^{\mathcal{C}} Y \simeq *$, resp.\ $Y \in (L\mathcal{C})_{>n}$).
    \item If $L$ is locally cartesian and effective on $X$, pullback along $\eta_X \colon X \to LX$ induces an equivalence of posets
    \[
        Cov^n_{L\mathcal{C}}(LX)_*^{>_\mathcal{C} n} \xrightarrow{\sim} Cov^n_L(X)_*^{>_L n},
    \]
    where $Cov^n_L(X)_*^{>_L n} \subset Cov^n_L(X)_*$ is the full subcategory spanned by $(Y \to X) \in Cov^n_L(X)_*$ with $\tau_{\leq n}^{\mathcal{C}} LY \simeq *$ (i.e., $LY \in \mathcal{C}_{>n}$).
\end{enumerate}
\end{thm}
\begin{proof}
First consider the case $L = \mathrm{id}_{\mathcal{C}}$. For any $(p \colon Y \to X) \in Cov^n_{\mathcal{C}}(X)_*^{>_\mathcal{C} n}$, $\mathrm{fib}(p)$ is $n$-truncated by \Cref{covering if and only if fiber is discrete}, while $Y$ and $X$ lie in $\mathcal{C}_{>n}$. The long exact sequence of homotopy groups therefore implies that $\mathrm{fib}(p) \in \mathcal{EM}_n(\mathcal{C})$ and that $p_* \colon \pi_{n+1}^{\mathcal{C}}(Y) \to \pi_{n+1}^{\mathcal{C}}(X)$ is a monomorphism. Moreover, for any $Y_1, Y_2 \in Cov^n_{\mathcal{C}}(X)_*^{>_\mathcal{C} n}$, since $Y_1 \in \mathcal{C}_{>n}$ and $Y_2 \to X$ is pulled back from $\tau_{\leq n+1}^{\mathcal{C}} Y_2 \to \tau_{\leq n+1}^{\mathcal{C}} X$, we have
\[
    \mathrm{Map}_{\mathcal{C}_{*/X}}(Y_1, Y_2) \simeq \mathrm{Map}_{\mathcal{EM}_{n+1}(\mathcal{C})_{/\tau_{\leq n+1}^{\mathcal{C}} X}}(\tau_{\leq n+1}^{\mathcal{C}} Y_1, \tau_{\leq n+1}^{\mathcal{C}} Y_2) \simeq \mathrm{Hom}_{/ \pi_{n+1}^{\mathcal{C}}(X)}(\pi_{n+1}^{\mathcal{C}}(Y_1), \pi_{n+1}^{\mathcal{C}}(Y_2)).
\]
Thus $Cov^n_{\mathcal{C}}(X)_*^{>_\mathcal{C} n}$ is a poset, and the functor
\[
    {}_*\pi_{n+1}^{\mathcal{C}} \colon Cov^n_{\mathcal{C}}(X)_*^{>_\mathcal{C} n} \to Sub_{(n+1)\mbox{-}\mathrm{Grp}(\mathrm{Disc}(\mathcal{C}))}(\pi_{n+1}^{\mathcal{C}}(X)), \quad (p \colon Y \to X) \mapsto p_*(\pi_{n+1}^{\mathcal{C}}(Y)),
\]
is fully faithful. Its inverse $p_{-}$ sends a subgroup $i \colon H \leq \pi_{n+1}^{\mathcal{C}}(X)$ to the pullback
\[
\xymatrix{
E_H \ar[r] \ar[d]_{p_H} & \mathrm{B}^{n+1}_{\mathcal{C}}H \ar[d]^{\mathrm{B}^{n+1}_{\mathcal{C}}i} \\
X \ar[r]_-{q} & \mathrm{B}^{n+1}_{\mathcal{C}}\pi_{n+1}^{\mathcal{C}}(X)
}
\]
(sending the trivial subgroup $*$ to $\tau_{>n+1}^{\mathcal{C}} X$). Since $i$ is a monomorphism, $\mathrm{fib}(p_H) \simeq \mathrm{fib}(\mathrm{B}^{n+1}_{\mathcal{C}}i) \in \mathcal{C}_{\leq n}$, so $p_H$ is a $\mathcal{C}$-$n$-covering by \Cref{covering if and only if fiber is discrete}. By the long exact sequence of homotopy groups, $E_H \in \mathcal{C}_{>n}$ and $(p_H)_*\pi_{n+1}^{\mathcal{C}}(E_H) \simeq H$.

For general $L$, since $LX \in \mathcal{C}_{>n}$ and $L$ is path-injective and $(n+1)$-fundamental on $L$-locally $n$-connected objects, we have $\pi_{n+1}^L(X) \simeq \pi_{n+1}^{L\mathcal{C}}(LX) \simeq \pi_{n+1}^{\mathcal{C}}(LX)$ (\Cref{fundamental loop pi}). Moreover, since $L\mathcal{C} \mcap \mathcal{C}_{>n} = (L\mathcal{C})_{>n}$ (\Cref{conn of localization}) we have $$Cov^n_{L\mathcal{C}}(LX)_*^{>_\mathcal{C} n} = Cov^n_{L\mathcal{C}}(LX)_*^{>_{L\mathcal{C}} n}.$$ Any $(p \colon Y \to LX)$ in this category satisfies $Y \in L\mathcal{C} \mcap \mathcal{C}_{>n}$, so path-injectivity and $(n+1)$-fundamentality give $$\tau_{\leq n+1}^{\mathcal{C}} Y \simeq \tau_{\leq n+1}^{L\mathcal{C}} Y \in \mathcal{EM}_{n+1}(L\mathcal{C})$$ (\Cref{em of localization}) and thus $\pi_{n+1}^{\mathcal{C}}(Y) \simeq \pi_{n+1}^{L\mathcal{C}}(Y) \in (n+1)\mbox{-}\mathrm{Grp}^L$ by \Cref{grp em}. Hence ${}_*\pi_{n+1}^{\mathcal{C}}$ restricts to
\[
    {}_*\pi_{n+1}^{L\mathcal{C}} \colon Cov^n_{L\mathcal{C}}(LX)_*^{>_\mathcal{C} n} \to Sub_{(n+1)\mbox{-}\mathrm{Grp}^L}(\pi_{n+1}^L(X)).
\]
Conversely, for any $H \in Sub_{(n+1)\mbox{-}\mathrm{Grp}^L}(\pi_{n+1}^L(X))$, both $\mathrm{B}^{n+1}_{\mathcal{C}}H$ and $\mathrm{B}^{n+1}_{\mathcal{C}}\pi_{n+1}^{\mathcal{C}}(LX) \simeq \tau_{\leq n+1}^{\mathcal{C}} LX$ are $L$-local, so $$\mathrm{fib}(p_H) \simeq \mathrm{fib}(\mathrm{B}^{n+1}_{\mathcal{C}}i) \in (L\mathcal{C})_{\leq n}.$$ Extension-closure of $L\mathcal{C}$ over $LX$ then implies that $E_H$ is $L$-local, so $p_{-}$ restricts to the required inverse (carrying $*$ to $\tau^{L\mathcal{C}}_{>n+1}{X}$).

Finally, when $L$ is locally cartesian and effective on $X$, the pointed pullback equivalence $\eta_X^* \colon Cov^n_{L\mathcal{C}}(LX)_* \xrightarrow{\sim} Cov^n_L(X)_*$ of \Cref{L covering is covering of L} restricts to $$Cov^n_{L\mathcal{C}}(LX)_*^{>_\mathcal{C} n} \xrightarrow{\sim} Cov^n_L(X)_*^{>_L n}$$ (carrying $\tau^{L\mathcal{C}}_{>n+1}{X}$ to $\tau_{>n+1}^L {X}$), as local cartesianity gives $L(\eta_X^* Y) \simeq Y$ for every $Y \in Cov^n_{L\mathcal{C}}(LX)_*$.
\end{proof}

\subsubsection{Geometric \texorpdfstring{$n$}{n}-coverings (\texorpdfstring{$L$}{L}-local \texorpdfstring{$n$}{n}-connective morphisms)}

We next extend the notions of effective epimorphisms [\cite{lurie2009higher}, paragraph following Corollary 6.2.3.5] and $n$-connective morphisms [\cite{lurie2009higher}, Definition 6.5.1.10] from $\infty$-toposes to general $\infty$-categories with fiber products and geometric realizations. Recall that in our indexing convention (\Cref{L n connective}), effective epimorphisms correspond to $0$-connective morphisms.

\begin{defn}[\textup{[\cite{nlab:effective_epimorphism_in_ainfinity1category}]}]\label[defn]{defn eff epi}
Let $\mathcal{C}$ be an $\infty$-category admitting fiber products and geometric realizations. A morphism $f \colon X \to Y$ in $\mathcal{C}$ is an \emph{effective epimorphism} (or \emph{$0$-connective morphism}) if its \v{C}ech nerve $\check{C}(f)$ is a colimit diagram, so that $|\check{C}(f)| \xrightarrow{\sim} Y$.
\end{defn}

\begin{lem}\label[lem]{pi0 preserves eff epi}
Let $\mathcal{C}$ be a presentable $\infty$-category. Then $\pi_0^{\mathcal{C}} \colon \mathcal{C} \to \mathrm{Disc}(\mathcal{C})$ preserves effective epimorphisms.
\end{lem}
\begin{proof}
Since $\pi_0^{\mathcal{C}} \colon \mathcal{C} \to \mathrm{Disc}(\mathcal{C})$ is a left adjoint, it preserves colimits. Given an effective epimorphism $f \colon X \to Y$ in $\mathcal{C}$, applying $\pi_0^{\mathcal{C}}$ to $|\check{C}(f)| \xrightarrow{\sim} Y$ and truncating to the $1$-skeleton in the $1$-category $\mathrm{Disc}(\mathcal{C})$ yields a coequalizer diagram
\[
    \pi_0^{\mathcal{C}}(X \times_Y X) \rightrightarrows \pi_0^{\mathcal{C}}(X) \to \pi_0^{\mathcal{C}}(Y).
\]
Factoring the pair of projections through the canonical map $\pi_0^{\mathcal{C}}(X \times_Y X) \to \pi_0^{\mathcal{C}}(X) \times_{\pi_0^{\mathcal{C}}(Y)} \pi_0^{\mathcal{C}}(X)$ shows that
\[
    \pi_0^{\mathcal{C}}(X) \times_{\pi_0^{\mathcal{C}}(Y)} \pi_0^{\mathcal{C}}(X) \rightrightarrows \pi_0^{\mathcal{C}}(X) \to \pi_0^{\mathcal{C}}(Y)
\]
is also a coequalizer diagram in $\mathrm{Disc}(\mathcal{C})$. Because $\mathrm{Disc}(\mathcal{C}) \subset \mathcal{C}$ is a reflective subcategory (hence closed under fiber products) and a $1$-category, this coequalizer computes $|\check{C}(\pi_0^{\mathcal{C}}(f))|$ in $\mathrm{Disc}(\mathcal{C})$, so $\pi_0^{\mathcal{C}}(f)$ is an effective epimorphism.
\end{proof}

\begin{prop}\label[prop]{pullback stability eff epi}
Let $\mathcal{C}$ be an $\infty$-category admitting fiber products and geometric realizations. If geometric realizations in $\mathcal{C}$ are universal (for instance, if $\mathcal{C}$ is presentable with universal colimits), then effective epimorphisms in $\mathcal{C}$ are stable under pullbacks.
\end{prop}
\begin{proof}
Let $f \colon X \to Y$ be an effective epimorphism and let $f' \colon X' \coloneqq X \times_Y Y' \to Y'$ be its pullback along a morphism $Y' \to Y$. Since pullback along $Y' \to Y$ preserves fiber products and, by universality, geometric realizations, we have $\check{C}(f') \simeq Y' \times_Y \check{C}(f)$ and hence $|\check{C}(f')| \simeq Y' \times_Y |\check{C}(f)| \xrightarrow{\sim} Y'$.
\end{proof}

\begin{lem}\label[lem]{epi mono factorization in a presentable infty cat}
In any presentable $\infty$-category $\mathcal{C}$ with universality of colimits, every morphism $f \colon X \to Y$ admits a factorization through a $0$-connective morphism (an effective epimorphism) followed by a $(-1)$-truncated morphism (a monomorphism).
\end{lem}
\begin{proof}
Consider the canonical factorization $X \xrightarrow{p} |\check{C}(f)| \xrightarrow{i} Y$. By universality of colimits and the split augmentation of the shifted \v{C}ech nerve $\check{C}(f)_{\bullet+1} \simeq X \times_Y \check{C}(f)$ over $X$, we have $X \times_Y |\check{C}(f)| \simeq |\check{C}(f)_{\bullet+1}| \simeq X$. Pulling back along $\check{C}(f)_m \to X$ and applying universality of colimits again yields $|\check{C}(f)| \times_Y |\check{C}(f)| \simeq |\check{C}(f)|$, so $i$ is a monomorphism [\cite{lurie2009higher}, Lemma 5.5.6.15]. Consequently, fiber products over $|\check{C}(f)|$ agree with those over $Y$, whence $\check{C}(p) \simeq \check{C}(f)$ and $|\check{C}(p)| \simeq |\check{C}(f)|$, so $p$ is an effective epimorphism.
\end{proof}

We refer to $X \to |\check{C}(f)| \to Y$ as the \emph{image factorization} of $f$, and to $\mathrm{im}_{\mathcal{C}}(f) \coloneqq |\check{C}(f)|$ as the \emph{$\mathcal{C}$-image} of $f$.
\begin{prop}\label[prop]{eff epi sub injective}
Let $\mathcal{C}$ be a presentable $\infty$-category with universal colimits. A morphism $f \colon U \to X$ in $\mathcal{C}$ is an effective epimorphism if and only if $f^* \colon Sub(X) \to Sub(U)$ is injective (where $Sub(X)$ denotes the poset of subobjects of $X$, i.e., the skeleton of $Cov^{-1}_{\mathcal{C}}(X)$).
\end{prop}
\begin{proof}
The proof of [\cite{lurie2009higher}, Proposition 6.2.3.10] applies verbatim using \Cref{pullback stability eff epi,epi mono factorization in a presentable infty cat}.
\end{proof}
\begin{prop}[Right cancellation of effective epimorphisms]\label[prop]{prop right cancellation eff epi}
Let $\mathcal{C}$ be an $\infty$-category admitting fiber products and geometric realizations in which effective epimorphisms are stable under pullback. For composable morphisms $X \xrightarrow{f} Y \xrightarrow{g} Z$ in $\mathcal{C}$, if $g \circ f$ is an effective epimorphism, then $g$ is an effective epimorphism.
\end{prop}
\begin{proof}
This is immediate from \Cref{eff epi sub injective} and the left cancellation property of injections. Below we give a direct proof.

Consider the bisimplicial object $W_{m,n} \coloneqq \check{C}(g \circ f)_m \times_Z \check{C}(g)_n$ augmented over $Z$. For each $m \ge 0$, $W_{m,\bullet}$ identifies with the \v{C}ech nerve of the projection $\check{C}(g \circ f)_m \times_Z Y \to \check{C}(g \circ f)_m$, which admits a section $(\mathrm{id}, f \circ \mathrm{pr}_0)$ and thus satisfies $|W_{m,\bullet}| \xrightarrow{\sim} \check{C}(g \circ f)_m$ [\cite{lurie2009higher}, Lemma~6.1.3.16 and Remark~6.1.3.17]. Dually, for each $n \ge 0$, $W_{\bullet,n}$ identifies with the \v{C}ech nerve of the pullback $X \times_Z \check{C}(g)_n \to \check{C}(g)_n$ of $g \circ f$, yielding $|W_{\bullet,n}| \xrightarrow{\sim} \check{C}(g)_n$ by pullback stability. Hence $|\check{C}(g)| \simeq \colim_{n} |W_{\bullet,n}| \simeq \colim_{m}|W_{m,\bullet}| \simeq |\check{C}(g \circ f)| \xrightarrow{\sim} Z$.
\end{proof}
\begin{prop}\label[prop]{rem eff epi vs -1 trunc}
Let $\mathcal{C}$ be a presentable $\infty$-category with universal colimits. For any morphism $f \colon X \to Y$ in $\mathcal{C}$, the canonical comparison map $u \colon \tau_{\leq -1}^{\mathcal{C}_{/Y}} f \to |\check{C}(f)|$ over $Y$ is an equivalence.
\end{prop}
\begin{proof}
Let $X \xrightarrow{q} \tau_{\leq -1}^{\mathcal{C}_{/Y}} f \xrightarrow{j} Y$ be the $(-1)$-truncation factorization of $f$ in $\mathcal{C}_{/Y}$, where $j$ is $(-1)$-truncated (a monomorphism). By \Cref{epi mono factorization in a presentable infty cat}, the map $i \colon |\check{C}(f)| \to Y$ in the image factorization $X \xrightarrow{p} |\check{C}(f)| \xrightarrow{i} Y$ is also $(-1)$-truncated. The universal property of $(-1)$-truncation in $\mathcal{C}_{/Y}$ therefore yields a canonical comparison map $u \colon \tau_{\leq -1}^{\mathcal{C}_{/Y}} f \to |\check{C}(f)|$ over $Y$ such that $u \circ q \simeq p$ and $i \circ u \simeq j$. Since $p$ is an effective epimorphism, right cancellation (\Cref{pullback stability eff epi,prop right cancellation eff epi}) applied to $p \simeq u \circ q$ implies that $u$ is an effective epimorphism, so $|\check{C}(u)| \xrightarrow{\sim} |\check{C}(f)|$. On the other hand, since $j \simeq i \circ u$ is a monomorphism, $u$ is also a monomorphism by [\cite{lurie2009higher}, Lemma 5.5.6.14], whence $\tau_{\leq -1}^{\mathcal{C}_{/Y}} f \xrightarrow{\sim} |\check{C}(u)|$. Combining these equivalences shows that $u \colon \tau_{\leq -1}^{\mathcal{C}_{/Y}} f \xrightarrow{\sim} |\check{C}(f)|$ is an equivalence (compare [\cite{lurie2009higher}, Corollary 6.2.3.5] in the case of a semi-topos).
\end{proof}

\begin{prop}\label[prop]{L preserves effective epi}
Let $L \colon \mathcal{C} \to L\mathcal{C}$ be an accessible localization of a presentable $\infty$-category $\mathcal{C}$ such that $L\mathcal{C}$ has universal colimits. Then $L$ preserves effective epimorphisms.
\end{prop}
\begin{proof}
Let $f \colon X \to Y$ be an effective epimorphism in $\mathcal{C}$, so that $|\check{C}^{\mathcal{C}}_\bullet f| \xrightarrow{\sim} Y$. Although $L$ need not preserve fiber products, there is a canonical comparison map $L(\check{C}^{\mathcal{C}}_\bullet f) \to \check{C}^{L\mathcal{C}}_\bullet Lf$ over $LY$. Since $L$ preserves colimits, taking geometric realizations in $L\mathcal{C}$ yields a section of the canonical map $|\check{C}^{L\mathcal{C}}_\bullet Lf| \to LY$. Because $L\mathcal{C}$ has universal colimits, $|\check{C}^{L\mathcal{C}}_\bullet Lf| \to LY$ is a monomorphism by \Cref{epi mono factorization in a presentable infty cat}, and a monomorphism admitting a section is an equivalence [\cite{lurie2009higher}, paragraph after Example 5.5.6.13].
\end{proof}

\begin{rem}\label[rem]{Effective localization remark}
We say that a localization $L$ is \emph{effective on $X$} if the localization unit $\eta_X \colon X \to LX$ is an effective epimorphism ($0$-connective) in $\mathcal{C}$, and that $L$ is \emph{effective} if it is effective on every object. For example, the localization of a sheaf topos $\mathcal{T}$ at an interval object $I$ in the sense of [\cite{morel19991}] is effective: the unit induces an epimorphism $\tau_{\leq 0}^{\mathcal{T}}X \to \tau_{\leq 0}^{\mathcal{T}}L_I X$ [\cite{morel19991}, Corollary 2.3.22], so $X \to L_I X$ is an effective epimorphism by [\cite{lurie2009higher}, Proposition 7.2.1.14]. More generally, any cartesian localization at a split object (an object admitting a global section) is effective (see [\cite{p1algtop}]).
\end{rem}

\begin{prop}\label[prop]{L n connective factorization}
Let $\mathcal{C}$ be an $\infty$-topos and suppose $L$ is locally cartesian and effective on $Y \in \mathcal{C}$. For every $n \geq -1$, every morphism $f \colon X \to Y$ in $\mathcal{C}$ factors as $f = g_n \circ h_n$, where $g_n \in Cov^{n-1}_L(Y)$ is an $L$-$(n-1)$-covering (an equivalence when $n=-1$, and an $L$-monomorphism when $n=0$) and $L(h_n)$ has terminal $(n-1)$-truncation in $L\mathcal{C}_{/L(\mathrm{dom}(g_n))}$.
\end{prop}
\begin{proof}
Pull back the slice $(n-1)$-truncation factorization $LX \to \tau_{\leq n-1}^{L\mathcal{C}_{/LY}}(Lf) \to LY$ along $\eta_Y \colon Y \to LY$ and apply \Cref{L covering is covering of L}. Note that local cartesianity of $L$ identifies $Lh_n$ with $LX \to \tau_{\leq n-1}^{L\mathcal{C}_{/LY}}(Lf)$.
\end{proof}

Thus, in any presentable $\infty$-category $\mathcal{D}$, it is meaningful for $n \geq -1$ to declare a morphism $f$ to be \emph{$n$-connective} if $\tau_{\leq n-1}^{\mathcal{D}_{/\mathrm{cod}(f)}} f \to \mathrm{cod}(f)$ is an equivalence (so that every morphism is $(-1)$-connective, while for $n=0$ this agrees with effective epimorphisms when $\mathcal{D}$ has universal colimits, by \Cref{rem eff epi vs -1 trunc}). Henceforth, a morphism $f \colon X \to Y$ in $\mathcal{C}$ is called an \emph{$L$-local $n$-connective morphism} (and for $n=0$, an \emph{$L$-effective epimorphism}) if $Lf$ is $n$-connective in $L\mathcal{C}$.

\begin{lem}\label[lem]{eff loc preserves detects eff epi}
Viewed as an endofunctor of $\mathcal{C}$, an effective localization $L \colon \mathcal{C} \to \mathcal{C}$ of an $\infty$-topos preserves effective epimorphisms in $\mathcal{C}$ and detects them over $L\mathcal{C}$. In particular, an object $X \in \mathcal{C}$ is $L$-locally $(-1)$-connected if and only if $X$ is $(-1)$-connected in $\mathcal{C}$.
\end{lem}
\begin{proof}
For any effective epimorphism $f \colon X \to Y$ in $\mathcal{C}$, the composite $Lf \circ \eta_X \simeq \eta_Y \circ f$ is a composite of effective epimorphisms in an $\infty$-topos, hence an effective epimorphism [\cite{lurie2009higher}, Corollary 6.2.3.12]. Right cancellation (\Cref{prop right cancellation eff epi}) then implies that $Lf$ is an effective epimorphism in $\mathcal{C}$.

For detection over $L\mathcal{C}$, suppose $Y \in L\mathcal{C}$ is local, so that $\eta_Y$ is an equivalence and $f \simeq Lf \circ \eta_X$. Since $\eta_X$ is an effective epimorphism, $f$ is an effective epimorphism if and only if $Lf$ is an effective epimorphism by [\cite{lurie2009higher}, Corollary 6.2.3.12] and \Cref{prop right cancellation eff epi}. The final claim follows by taking $Y = *$.
\end{proof}

\begin{defn}[Geometric covering]\label[defn]{defn geom cov}
Let $\mathcal{C}$ be a presentable $\infty$-category and $L \colon \mathcal{C} \to L\mathcal{C}$ an accessible localization. We say that a morphism $f \colon Y \to X$ in $\mathcal{C}$ is:
\begin{enumerate}
    \item an \emph{external $L$-effective epimorphism} if $L\check{C}^{\mathcal{C}}_\bullet f$ is a colimit diagram in $L\mathcal{C}$ (equivalently, the $\mathcal{C}$-image map $\mathrm{im}_{\mathcal{C}}(f) \to X$ is an $L$-equivalence);
    \item an \emph{internal $L$-effective epimorphism} (or simply an \emph{$L$-effective epimorphism}) if $Lf$ is an effective epimorphism in $L\mathcal{C}$.
\end{enumerate}
We say that a collection of morphisms $\{p_i \colon U_i \to X\}_{i \in I}$ in $\mathcal{C}$ is an \emph{(internal/external) geometric $L$-covering} if the induced morphism $p \colon U \coloneqq \bigsqcup_{i \in I} U_i \to X$ is an (internal/external) $L$-effective epimorphism.
\end{defn}

\begin{lem}\label[lem]{comparison of various geom cov}
Let $\mathcal{C}$ be a presentable $\infty$-category and $L \colon \mathcal{C} \to L\mathcal{C}$ an accessible localization.
\begin{enumerate}
    \item Every geometric $\mathcal{C}$-covering is a geometric external $L$-covering. If $L\mathcal{C}$ has universal colimits, then every geometric external $L$-covering is a geometric internal $L$-covering.
    \item A collection of morphisms in $L\mathcal{C}$ is a geometric $L\mathcal{C}$-covering if and only if, viewed in $\mathcal{C}$, it is a geometric external $L$-covering, if and only if it is a geometric internal $L$-covering.
\end{enumerate}
\end{lem}
\begin{proof}
In (1), the first statement follows because $L$ preserves colimits, while the second statement follows by the same argument as in the proof of \Cref{L preserves effective epi}. Part (2) follows from the fact that $L\mathcal{C}$ is closed under fiber products (and hence \v{C}ech nerves) in $\mathcal{C}$.
\end{proof}

\begin{lem}[Local-to-global $L$-connectivity]\label[lem]{L geom L connectivity}
Suppose $L$ is a path-injective localization of an $\infty$-topos $\mathcal{C}$. Let $\{p_i \colon U_i \to X\}_{i \in I}$ be an internal geometric $L$-covering (for example, a geometric $\mathcal{C}$-covering when $L\mathcal{C}$ has universal colimits) such that each $U_i$ is $L$-locally $0$-connected and there exists $i_0 \in I$ such that $U_{i_0} \times_X U_i$ is $L$-locally $(-1)$-connected for all $i \in I$ (for example, when each $U_{i_0} \times_X U_i$ is pointed). Then $X$ is $L$-locally $0$-connected.
\end{lem}
\begin{proof}
Since $Lp \colon \bigsqcup_{i \in I} LU_i \to LX$ is an effective epimorphism in $L\mathcal{C}$ and $\pi_0^{L\mathcal{C}}$ preserves effective epimorphisms (\Cref{pi0 preserves eff epi}), the induced map $$(f_i)_{i \in I} \colon \bigsqcup_{i \in I} \pi_0^{L\mathcal{C}} LU_i \simeq \bigsqcup_{i \in I} * \twoheadrightarrow \pi_0^{L\mathcal{C}} LX$$ is an epimorphism in $\mathrm{Disc}(L\mathcal{C})$. For each $i \in I$, both projections from $P_i \coloneqq \pi_0^{L\mathcal{C}} L(U_{i_0} \times_X U_i)$ to $\pi_0^{L\mathcal{C}} LU_{i_0} \simeq *$ and $\pi_0^{L\mathcal{C}} LU_i \simeq *$ coincide with the unique map $e_i \colon P_i \to *$, which is an epimorphism by \Cref{pi0 preserves eff epi} since $L(U_{i_0} \times_X U_i) \to *$ is an effective epimorphism in $L\mathcal{C}$. From $f_{i_0} \circ e_i = f_i \circ e_i$ we obtain $f_i = f_{i_0}$ for all $i \in I$. Hence the epimorphism $(f_i)_{i \in I} \colon \bigsqcup_{i \in I} * \twoheadrightarrow \pi_0^{L\mathcal{C}} LX$ factors through $f_{i_0} \colon * \to \pi_0^{L\mathcal{C}} LX$, so $f_{i_0}$ is both an epimorphism and a split monomorphism (having terminal domain $*$), and thus an isomorphism. Path-injectivity of $L$ then yields $\pi_0^L X \simeq *$.
\end{proof}

\begin{thm}[Higher $L$-local Van Kampen theorem]\label{L van campen}
Let $n \geq 0$, and let $L$ be an accessible localization of an $\infty$-topos $\mathcal{C}$ such that (when $n \geq 1$) $L$ is path-injective and $j$-fundamental on $(j-1)$-connected objects for $1 \leq j \leq n$. Let $\{p_i \colon U_i \to X\}_{i \in I}$ be a pointed geometric external $L$-covering in $\mathcal{C}$ such that for every $1 \leq m \leq n+2$ and $i_1, \dots, i_m \in I$, the $m$-fold intersection $U_{i_1 \dots i_m} \coloneqq U_{i_1} \times_X \cdots \times_X U_{i_m}$ is $L$-locally $(n-1)$-connected (and canonically pointed by the base points of $U_i$ by the universal property of fiber-products). Then $X$ is $L$-locally $(n-1)$-connected, and there is a coequalizer diagram in $n\mbox{-}\mathrm{Grp}^L$:
\[
    \coprod_{i,j} \pi_n^{L\mathcal{C}}\bigl(L(U_i \times_X U_j)\bigr) \rightrightarrows \coprod_i \pi_n^{L\mathcal{C}}(LU_i) \to \pi_n^{L\mathcal{C}}(LX) \to *,
\]
where $\pi_n^{L\mathcal{C}}(L-) \simeq \pi_n^L(-)$ when $n \geq 1$ (and $\pi_0^{L\mathcal{C}}(L-) \simeq L_{\leq 0}\pi_0^L(-)$ in $0\mbox{-}\mathrm{Grp}^L = \mathrm{Disc}(L\mathcal{C})_*$ when $n = 0$), the basepoints of the intersections are induced via the universal property of pullbacks from the pointings of $U_i \to X$, and $\coprod$ denotes the coproduct in $n\mbox{-}\mathrm{Grp}^L$.
\end{thm}
\begin{proof}
We proceed by induction on $n \geq 0$. For $n=0$, $X$ is pointed and hence $L$-locally $(-1)$-connected. For $n \geq 1$, $\{p_i \colon U_i \to X\}_{i \in I}$ is also an internal geometric $L$-covering by \Cref{comparison of various geom cov}(1); when $n=1$, $L$-local $0$-connectivity of $X$ is immediate from \Cref{L geom L connectivity}, while when $n \geq 2$, the inductive hypothesis applied to $n-1$ shows that $LX \in (L\mathcal{C})_{\geq n-1}$ and that $\pi_{n-1}^L(X) \simeq \pi_{n-1}^{L\mathcal{C}}(LX)$ (via $(n-1)$-fundamentality) is the coequalizer in $(n-1)\mbox{-}\mathrm{Grp}^L$ of $$\coprod_{i,j} \pi_{n-1}^L(U_i \times_X U_j) \rightrightarrows \coprod_i \pi_{n-1}^L(U_i),$$ which vanishes since each $LU_i$ and $L(U_i \times_X U_j)$ belongs to $(L\mathcal{C})_{\geq n}$. Thus $LX \in (L\mathcal{C})_{\geq n}$, and $n$-fundamentality yields $\pi_n^{L\mathcal{C}}(L-) \simeq \pi_n^L(-)$ on $X$ and on the intersections (\Cref{fundamental loop pi}).

Let $p \colon U \coloneqq \bigsqcup_{i \in I} U_i \to X$ and let $q \colon \bigsqcup_{i \in I} * \to *$ be the canonical map. The compatible pointings $* \to U_i \to X$ assemble into a morphism $\bigsqcup_{i \in I} * \to U$ over $* \to X$, inducing a map of \v{C}ech nerves $\check{C}_\bullet(q) \to \check{C}_\bullet(p)$ in $\mathcal{C}$. Since fiber products distribute over coproducts in $\mathcal{C}$, we have $$\check{C}_{m-1}(q) \simeq \bigsqcup_{(i_1,\dots,i_m) \in I^m} *\quad \text{and} \quad \check{C}_{m-1}(p) \simeq \bigsqcup_{(i_1,\dots,i_m) \in I^m} U_{i_1 \dots i_m}.$$ Since $|\check{C}_\bullet(q)| \simeq *$ and $|L\check{C}_\bullet(p)| \simeq LX$ in $L\mathcal{C}$, taking the cofiber of $L\check{C}_\bullet(q) \to L\check{C}_\bullet(p)$ expresses $LX$ as the geometric realization in $L\mathcal{C}_*$ of the pointed simplicial object $W_\bullet \coloneqq L\check{C}_\bullet(p) / L\check{C}_\bullet(q)$, whose $(m-1)$-simplices are the pointed wedge sums
\[
    W_{m-1} \simeq \bigvee_{(i_1,\dots,i_m) \in I^m} LU_{i_1 \dots i_m} \in L\mathcal{C}_*.
\]
Applying the left adjoint $\tau_{\leq n}^{L\mathcal{C}} \colon L\mathcal{C}_* \to (L\mathcal{C})_{*,\leq n}$ yields $\tau_{\leq n}^{L\mathcal{C}} LX \simeq |\tau_{\leq n}^{L\mathcal{C}} W_\bullet|$. Because $(L\mathcal{C})_{*,\leq n}$ is an $(n+1,1)$-category, this geometric realization is computed by the colimit over the $(n+1)$-skeleton $\Delta_{\leq n+1}^{\mathrm{op}}$ (comprising $1 \leq m \leq n+2$) [\cite{lurie2009higher}, Lemma 6.5.3.9]. Both the colimit $\tau_{\leq n}^{L\mathcal{C}} LX$ and the terms $$\tau_{\leq n}^{L\mathcal{C}} W_{m-1} \simeq \bigvee_{(i_1,\dots,i_m) \in I^m} \tau_{\leq n}^{L\mathcal{C}} LU_{i_1 \dots i_m}\quad \text{for}\quad 1 \leq m \leq n+2$$ lie in the full subcategory $\mathcal{EM}_n(L\mathcal{C}) \subset (L\mathcal{C})_{*,\leq n}$: for $n=0$ this is automatic since $\mathcal{EM}_0(L\mathcal{C}) = (L\mathcal{C})_{*,\leq 0}$, while for $n \geq 1$ each $W_{m-1}$ belongs to $(L\mathcal{C})_{\geq n}$ by \Cref{tau preserves connectivity}(1), and $\tau_{\leq n}^{L\mathcal{C}}$ preserves $(L\mathcal{C})_{\geq n}$ by \Cref{tau preserves connectivity}(2). Because $\mathcal{EM}_n(L\mathcal{C})$ is full in $(L\mathcal{C})_{*,\leq n}$, $\tau_{\leq n}^{L\mathcal{C}} LX$ remains the colimit of $\tau_{\leq n}^{L\mathcal{C}} W_\bullet|_{\Delta_{\leq n+1}^{\mathrm{op}}}$ inside $\mathcal{EM}_n(L\mathcal{C})$. Applying the equivalence $\pi_n^{L\mathcal{C}} \colon \mathcal{EM}_n(L\mathcal{C}) \xrightarrow{\sim} n\mbox{-}\mathrm{Grp}^L$ (\Cref{grp em}) and reducing the colimit over $\Delta_{\leq n+1}^{\mathrm{op}}$ in the $1$-category $n\mbox{-}\mathrm{Grp}^L$ to its $1$-skeleton $\Delta_{\leq 1}^{\mathrm{op}}$ yields the stated coequalizer diagram.
\end{proof}

\begin{rem}
An internal $L\mathcal{C}$ Van Kampen theorem for internal geometric $L$-coverings in $L\mathcal{C}$ follows immediately from \Cref{L van campen} by \Cref{comparison of various geom cov}(2) (without requiring $L\mathcal{C}$ to have universal colimits if one assumes the covering is already internal). Specifically, for $n \geq 0$ and any pointed geometric internal $L$-covering $\{p_i \colon U_i \to X\}_{i \in I}$ whose $m$-fold fiber products $LU_{i_1} \times_{LX} \cdots \times_{LX} LU_{i_m}$ in $L\mathcal{C}$ are $(n-1)$-connected for $1 \leq m \leq n+2$, the same argument yields a coequalizer diagram in $n\mbox{-}\mathrm{Grp}^L$:
\[
    \coprod_{i,j} \pi_n^{L\mathcal{C}}(LU_i \times_{LX} LU_j) \rightrightarrows \coprod_i \pi_n^{L\mathcal{C}}(LU_i) \to \pi_n^{L\mathcal{C}}(LX) \to *.
\]
Because $L$ need not preserve fiber products, however, \Cref{L van campen} cannot be recovered directly from this internal version.
\end{rem}

\begin{cor}[Higher local-to-global $L$-connectivity]\label[cor]{higher L geom L connectivity}
Let $n \geq 1$, and let $L$ be a path-injective localization of an $\infty$-topos $\mathcal{C}$ such that $L\mathcal{C}$ has universal colimits and $L$ is $j$-fundamental on $(j-1)$-connected objects for $1 \leq j \leq n$. Let $\{p_i \colon U_i \to X\}_{i \in I}$ be a pointed geometric external $L$-covering in $\mathcal{C}$ such that each $U_i$ is $L$-locally $n$-connected and for every $2 \leq m \leq n+2$ and $i_1, \dots, i_m \in I$, the $m$-fold intersection $U_{i_1 \dots i_m}$ is $L$-locally $(n-1)$-connected. Then $X$ is $L$-locally $n$-connected.
\end{cor}
\begin{proof}
Since each $U_i$ is in particular $L$-locally $(n-1)$-connected, \Cref{L van campen} implies that $X$ is $L$-locally $(n-1)$-connected and yields an epimorphism $$*\simeq \coprod_i \pi_n^L(U_i) \to \pi_n^L(X)$$ in $n\mbox{-}\mathrm{Grp}^L$. Thus $\pi_n^L(X) \simeq *$, so $X$ is $L$-locally $n$-connected.
\end{proof}
\section{The internal homotopy theory of \texorpdfstring{$\mathcal{H}^{mot}$}{H\^{}mot}}

In this section, we apply the definitions and results from the previous section to the motivic localization functor $L_{mot}:\mathcal{P}_{nis}(S)\to\mathcal{H}^{mot}(S) $. Recall that for a qcqs scheme $S$, the full subcategory $\mathcal{P}_{nis}(S)\subset \mathcal{P}(S)$ of Nisnevich-local $S$-spaces (presheaves sending Nisnevich distinguished squares to pullback squares and $\emptyset$ to $*$) is a left-exact accessible localization $L_{nis}$, and thus an $\infty$-topos [\cite{hoyois2017six}, Corollary 3.9]. Similarly, the full subcategory of $\mathbb{A}^1$-local objects (presheaves $\esc{X}$ such that $\esc{X}(X)\xrightarrow{p^*}\esc{X}(\mathbb{A}^1\times X)$ is an equivalence for all $X\in Sm_S$) forms an accessible localization $L_{\mathbb{A}^1}$ of $\mathcal{P}(S)$. The intersection $\mathcal{H}^{mot}(S):=\mathcal{P}_{nis}(S)\mcap L_{\mathbb{A}^1}\mathcal{P}(S)$ of motivic $S$-spaces is an accessible localization $L_{mot}$ of $\mathcal{P}(S)$ (and of $\mathcal{P}_{nis}(S)$) that is cartesian and locally cartesian [\cite{hoyois2017six}, Proposition 3.15].

Morel defines the $\mathbb{A}^1$-homotopy groups of a space as the Nisnevich homotopy group of its motivic localization, i.e., $$\pi_i^{\mathbb{A}^1}:=\pi_i^{nis}L_{mot}:=L_{nis}\pi_iL_{mot}.$$ In terms of \Cref{pinL}, this is precisely $\pi_i^{L_{mot}}$. Unfortunately, the $\mathbb{A}^1$-homotopy path components construction $\pi_0^{\mathbb{A}^1}$ fails to be the internal $0$-truncation functor for the motivic homotopy category, and consequently does not preserve (motivic) colimits. The agenda of this section is to elaborate on this issue and remedy it by introducing the internal notions of motivic truncations and motivic homotopy groups.

\subsection{Motivic truncations and motivic homotopy groups}

To simplify notation, we shall drop the scheme $S$, assuming it to be fixed until \Cref{pi_i mot=nis}, after which we shall assume that the base is a field (though not necessarily perfect, unless otherwise specified).

\underline{\textit{Motivic truncations:}}

\begin{defn}\label[defn]{motivically n truncated space}
Let $n\geq -2$. A motivic space $\esc{X}\in \mathcal{H}^{mot}$ is \emph{motivically $n$-truncated} if $\mathrm{Map}_{\mathcal{H}^{mot}}(\esc{Y},\esc{X})$ is $n$-truncated for every motivic space $\esc{Y}$ (\Cref{defn n truncated}). We denote the full subcategory of $n$-truncated motivic spaces by $\mathcal{H}^{mot}_{\leq n}$ or $\tau_{\leq n}^{mot}\mathcal{H}^{mot}$. Since $\mathcal{H}^{mot}$ is a presentable $\infty$-category, the inclusion $\mathcal{H}^{mot}_{\leq n}\subset \mathcal{H}^{mot}$ is a localization $\tau_{\leq n}^{mot}$.
\end{defn}

From \Cref{truncation of localization} and \Cref{L trunc is a localization}, it follows that:
\begin{prop}\label[prop]{motivic trunc is nis trunc}
A motivic space is motivically $n$-truncated if and only if it is Nisnevich-locally $n$-truncated, i.e., $$\mathcal{H}^{mot}_{\leq n}=\mathcal{H}^{mot}\mcap(\mathcal{P}_{nis})_{\leq n}.$$ Moreover, $\mathcal{H}^{mot}_{\leq n}\subset \mathcal{P}$ is a localization $L_{\leq n}^{mot}: \mathcal{P}\to \mathcal{H}^{mot}_{\leq n}$, which restricts to the internal motivic truncation functor $\tau_{\leq n}^{mot}: \mathcal{H}^{mot}\to \mathcal{H}^{mot}_{\leq n}$.
\end{prop}

\begin{rem}\label[rem]{sectionwise and stalkwise checking}
Since $\mathcal{P}_{nis}(S)$ is generated under colimits by $Sm_S$, it is easy to see that $\esc{X}\in\mathcal{P}_{nis}(S)$ is $n$-truncated if and only if for every $X\in Sm_S$, $\esc{X}(X)$ is an $n$-truncated space. Recall that when $n\geq 0$ and $S$ has finite Krull dimension, by \Cref{truncated condition} this is also equivalent to the requirement that $\pi_i^{nis}\esc{X}\simeq *$ for all $i>n$.
\end{rem}

\begin{prop}\label[prop]{-1-2 mot trunc}
For $n \in \{-2, -1\}$, every $\esc{F} \in \mathcal{P}_{\leq n}$ is $\mathbb{A}^1$-invariant; hence $\mathcal{H}^{mot}_{\leq n} = (\mathcal{P}_{nis})_{\leq n}$ and $\tau_{\leq n}^{nis}\simeq L_{\leq n}^{mot}$. In other words, $L_{mot}$ is $(-1)$- and $(-2)$-fundamental on every space.
\end{prop}

\begin{proof}
For $n = -2$, $\esc{F} \simeq *$ is terminal, and the claim is trivial. For $n = -1$, $\esc{F}$ takes values in the poset $\mathrm{Spc}_{\leq -1} \simeq \{\emptyset \to *\}$ (\Cref{sectionwise and stalkwise checking}). Since the projection $p : X \times_S \mathbb{A}^1_S \to X$ is split by the zero section, $p^* : \esc{F}(X) \to \esc{F}(X \times_S \mathbb{A}^1_S)$ admits a retraction in the poset $\mathrm{Spc}_{\leq -1}$ and is therefore an equivalence. The equivalence $\tau_{\leq n}^{nis}\simeq L_{\leq n}^{mot}$ then follows from \Cref{motivic trunc is nis trunc}.
\end{proof}

\begin{prop}\label[prop]{nis trunc preserve mot equiv}
For all $n$, the Nisnevich truncation functors $\tau_{\leq n}^{nis}:\mathcal{P}_{nis}\to (\mathcal{P}_{nis})_{\leq n}$ preserve motivic equivalences (i.e., the respective saturated classes generated by $\mathbb{A}^1$-projections).
\end{prop}

\begin{proof}
We use the fact that $\tau_{\leq n}^{nis}$ preserves colimits, so that it suffices to know that it preserves the $\mathbb{A}^1$-projection maps. If $n=-1,-2$, there is nothing to prove. When $n\geq 0$, we use the fact that schemes are Nisnevich-local and discrete, so that $\tau_{\leq n}^{nis}$ is the identity on $\mathbb{A}^1$-projections of schemes.
\end{proof}

\begin{rem}
It must, however, be noted that $\tau_{\leq n}^{nis}:\mathcal{P}_{nis}\to \mathcal{P}_{nis}$ (viewed as an endofunctor of $\mathcal{P}_{nis}$) does not preserve motivic equivalences. Indeed, for $n=0$, this would mean that $\mathbb{P}^1\simeq\pi_0^{nis}\mathbb{P}^1\to \pi_0^{nis}(\Sigma_{nis}\mathbb{G}_m)\simeq *$ is a motivic equivalence in $\mathcal{P}_{nis}$, which is absurd (although the $n=-1, -2$ cases are special).
\end{rem}

Thus we obtain a canonical map $\tau_{\leq n}^{nis}L_{mot}\to L^{mot}_{\leq n}$. However, it is not quite the case that $\tau_{\leq n}^{nis}$ preserves motivic local objects. For example, when $n=0$, this is the famous conjecture of Morel, later disproven by Ayoub in [\cite{ayoubcounterexamples}] (see the paragraph after \Cref{pi0 of h space} below). However, a transfinite iteration of $\tau_{\leq n}^{nis}L_{mot}$ does actually recover $L^{mot}_{\leq n}$ as we shall now discuss.

\begin{con}\label[con]{motivic truncation formulae}
Since $\mathcal{H}^{mot}_{\leq n} = \mathcal{P}_{nis} \mcap L_{\mathbb{A}^1}\mathcal{P} \mcap \mathcal{P}_{\leq n}$ (\Cref{motivic trunc is nis trunc}), we consider the sequential colimit of endofunctors on $\mathcal{P}$ given by
\[
    \colim_{m} (L_{nis}\tau_{\leq n}L_{\mathbb{A}^1})^{\circ m},
\]
where $L_{\mathbb{A}^1} \colon \mathcal{P} \to L_{\mathbb{A}^1}\mathcal{P}$, $\tau_{\leq n} \colon \mathcal{P} \to \mathcal{P}_{\leq n}$, and $L_{nis} \colon \mathcal{P} \to \mathcal{P}_{nis}$ denote the respective localization functors (note that $L_{nis}\tau_{\leq n} \simeq \tau_{\leq n}^{nis}L_{nis}$ since $L_{nis}$ is left-exact).
\end{con}

\begin{thm}\label{formula for motivic truncation}
The functor of \Cref{motivic truncation formulae} computes the localization $L_{\leq n}^{mot} \colon \mathcal{P} \to \mathcal{H}^{mot}_{\leq n}$; in particular, $\tau_{\leq n}^{mot} \simeq {L_{\leq n}^{mot}}_{|\mathcal{H}^{mot}}$.
\end{thm}
\begin{proof}
Each of the subcategories $\mathcal{P}_{nis}$, $L_{\mathbb{A}^1}\mathcal{P}$, and $\mathcal{P}_{\leq n}$ is closed under filtered colimits in $\mathcal{P}$: $\mathcal{P}_{nis}$ because the Nisnevich topology is a cd-topology; $L_{\mathbb{A}^1}\mathcal{P}$ because it is closed under all small colimits (checked sectionwise); and $\mathcal{P}_{\leq n}$ by [\cite{lurie2009higher}, Corollary 5.5.7.4]. By cofinality, $\colim_{m}(L_{nis}\tau_{\leq n}L_{\mathbb{A}^1})^{\circ m}\esc{X}$ therefore lies in $\mathcal{P}_{nis} \mcap L_{\mathbb{A}^1}\mathcal{P} \mcap \mathcal{P}_{\leq n} = \mathcal{H}^{mot}_{\leq n}$ for every $\esc{X} \in \mathcal{P}$. Moreover, since $\mathcal{H}^{mot}_{\leq n}$ is contained in all three subcategories, each transition map in the sequential diagram $\esc{X} \to (L_{nis}\tau_{\leq n}L_{\mathbb{A}^1})^{\circ m}\esc{X}$ is an $L_{\leq n}^{mot}$-equivalence, and hence so is the induced map to the colimit.
\end{proof}

\begin{rem}\label[rem]{remark on truncation of localization}
By the same cofinality argument, one obtains the equivalent formulas
\[
    L_{\leq n}^{mot} \simeq \colim_{m} (\tau^{nis}_{\leq n}L_{mot})^{\circ m} \simeq \colim_{m} (\tau_{\leq n}L_{mot})^{\circ m}.
\]
In particular, $\tau_{\leq 0}^{mot} \simeq \colim_{m}(\pi_0^{\mathbb{A}^1})^{\circ m}$ on $\mathcal{H}^{mot}$.
\end{rem}

Although internal truncations of an arbitrary localization need not preserve finite products, $\tau_{\leq n}^{mot}$ inherits this $\infty$-topos property [\cite{lurie2009higher}, Lemma 6.5.1.2]:

\begin{prop}\label[prop]{mot truncation cart}
The localization $L_{\leq n}^{mot}:\mathcal{P}\to \mathcal{H}^{mot}_{\leq n}$ has the following properties:
\begin{enumerate}
    \item It preserves finite products for all $n \geq -2$. Consequently, $\tau_{\leq n}^{mot}$ preserves finite products and $L_{mot}$ is path-cartesian (\Cref{path cart}).
    \item Its restriction to $\mathcal{P}_{nis}$ is effective for all $n \geq -1$ (i.e., $\esc{X} \to L_{\leq n}^{mot}\esc{X}$ is an effective epimorphism in $\mathcal{P}_{nis}$ for every $\esc{X} \in \mathcal{P}_{nis}$).
\end{enumerate}
\end{prop}
\begin{proof}
Since filtered colimits commute with finite products in $\mathcal{P}$, finite-product preservation follows factorwise from \Cref{motivic truncation formulae}: $L_{nis}$ is left-exact, $L_{\mathbb{A}^1}$ preserves finite products because $\mathbb{A}^1$-projections are stable under base change, and $\tau_{\leq n}$ preserves finite products by [\cite{lurie2009higher}, Lemma 6.5.1.2].

For effectivity, the case $n = -1$ holds by \Cref{-1-2 mot trunc}. For $n \geq 0$ and $\esc{X} \in \mathcal{P}_{nis}$, [\cite{lurie2009higher}, Proposition 7.2.1.14] and \Cref{remark on truncation of localization} reduce the claim to showing that
\[
    \pi_0^{nis}\esc{X} \longrightarrow \pi_0^{nis}L_{\leq n}^{mot}\esc{X} \simeq \colim_{m} \pi_0^{nis}\bigl((\tau_{\leq n}^{nis}L_{mot})^{\circ m}\esc{X}\bigr)
\]
is an epimorphism of Nisnevich sheaves. Since $\pi_0^{nis}\tau_{\leq n}^{nis} \simeq \pi_0^{nis}$ for $n \geq 0$ and the unit $\mathrm{id} \to L_{mot}$ induces an epimorphism on $\pi_0^{nis}$ [\cite{morel19991}, Corollary 2.3.22], each step in the sequential system is an epimorphism, and hence so is the map to the colimit.
\end{proof}

For $n = 0$, we have $\mathrm{Disc}(\mathcal{H}^{mot}) = \mathcal{H}^{mot} \mcap Sh_{nis} = Sh_{nis}^{\mathbb{A}^1}$, and the discrete localization ${L_{\leq 0}^{mot}}_{|Sh_{nis}}$ recovers the iterated $\mathbb{A}^1$-chain connected components functor $\mathcal{L} \colon Sh_{nis} \to Sh_{nis}^{\mathbb{A}^1}$ of [\cite{amit}]:

\begin{cor}
We have $L_{\leq 0}^{mot} \simeq \mathcal{L}\pi_0^{nis}$ on $\mathcal{P}$; in particular, $\tau_{\leq 0}^{mot}\esc{X} \cong \mathcal{L}\pi_0^{nis}\esc{X}$ for every motivic space $\esc{X}$.
\end{cor}
\begin{proof}
Using the $\mathbb{A}^1$-singular resolution $L_{\mathbb{A}^1} \simeq |(-)^{\Delta^\bullet_{\mathbb{A}^1}}|$ on $\mathcal{P}$, the $0$-truncation of $L_{\mathbb{A}^1}$ on discrete presheaves is $$\tau_{\leq 0}|(-)^{\Delta^\bullet_{\mathbb{A}^1}}| \cong \mathrm{coeq}\bigl((-)^{\mathbb{A}^1} \rightrightarrows (-)\bigr) = \mathcal{S}^{pre}$$ in the notation of [\cite{amit}]. On $Sh_{nis}$, one step $L_{nis}\tau_{\leq 0}L_{\mathbb{A}^1}$ therefore coincides with $\mathcal{S} = L_{nis}\mathcal{S}^{pre}$, so ${L_{\leq 0}^{mot}}_{|Sh_{nis}} \cong \colim_{m}\mathcal{S}^{\circ m} = \mathcal{L}$. Since $L_{\leq 0}^{mot}$ factors through $\pi_0^{nis} = L_{nis}\tau_{\leq 0}$, this yields $L_{\leq 0}^{mot} \simeq \mathcal{L}\pi_0^{nis}$ and $\tau_{\leq 0}^{mot} \simeq {\mathcal{L}\pi_0^{nis}}_{|\mathcal{H}^{mot}}$ (\Cref{pi0 to Lpi0}).
\end{proof}

\underline{\textit{Motivic homotopy groups:}}

\begin{defn}\label[defn]{defn motivic homotopy groups}
Following \Cref{LC homotopy group}, the \emph{internal motivic connected components functor} is defined as $\pi_0^{mot}:=\tau_{\leq 0}^{mot}:\mathcal{H}^{mot}(S)\to \mathrm{Disc}(\mathcal{H}^{mot}(S))=Sh_{nis}^{\mathbb{A}^1}(S)$. For $i\geq 1$, the \emph{$i$-th internal motivic homotopy group functor} is defined on pointed motivic spaces by (\Cref{mot truncation cart} and \Cref{path cart for groups})
\[
\pi_i^{mot}:=\tau_{\leq 0}^{mot}\Omega^i=\pi_0^{mot}\Omega^i: \mathcal{H}^{mot}(S)_*\longrightarrow i\mbox{-}\mathrm{Grp}(Sh_{nis}^{\mathbb{A}^1}(S)).
\]
\end{defn}

\begin{cor}\label[cor]{piimot is Lpiinis}
For all $i\geq 0$ and for $\esc{X}$ a motivic space, we have $$\pi_i^{mot}\esc{X}=\tau_{\leq 0}^{mot}\pi_i^{nis}\esc{X}\simeq \mathcal{L}\pi_i^{nis}\esc{X},$$ where $\mathcal{L}$ is the localization functor for the inclusion $Sh^{\mathbb{A}^1}_{nis}\subset Sh_{nis}$.
\end{cor}
\begin{proof}
This is \Cref{homotopy groups of localization} applied to $L_{mot}$, together with the observation ${L_{\leq 0}^{mot}}\simeq \mathcal{L}\circ{\pi_0^{nis}}$.
\end{proof}

(For the rest of this section, we shall work over fields. Thus, we let $S=\mathrm{Spec}(k)$ for a field $k$.)

\begin{prop}\label[prop]{pi_i mot=nis}
Let $k$ be a perfect field. Then for all $i\geq 1$ and for every pointed $k$-motivic space $\esc{X}$, there are equivalences $$\pi_i^{mot}\esc{X}\simeq \pi_i^{nis}\esc{X}\simeq \pi_i^{\mathbb{A}^1}\esc{X}.$$
\end{prop}
\begin{proof}
By \Cref{piimot is Lpiinis}, this is an immediate consequence of the fact that for every $i\geq 1$ the sheaf $\pi_i^{\mathbb{A}^1}\esc{X}$ is $\mathbb{A}^1$-invariant [\cite{morel2012a1}, Corollary 6.3].
\end{proof}

Following this, Morel conjectured that the same would hold even for $\pi_0$, i.e., $\pi_0^{mot}\simeq {\mathcal{L}\pi_0^{nis}}_{|\mathcal{H}^{mot}}\simeq {\pi_0^{nis}}_{|\mathcal{H}^{mot}}$ [\cite{morel2012a1}, Conjecture 1.12]. It was known that the above conjecture is very close to being true: the pointed map $\pi_0^{nis}\esc{X}(X)\to \pi_0^{nis}\esc{X}(k(X))$ has a trivial kernel for any pointed motivic space $\esc{X}$ [\cite{bachmann2024strongly}, Proposition 2.9]. Building on this observation, Choudhury proved in [\cite{choudhury2014connectivity}] that, for motivic spaces with a minimal amount of group structure, the Nisnevich homotopy path components sheaf $\pi_0^{nis}$ is indeed $\mathbb{A}^1$-invariant:

\begin{prop}\label[prop]{pi0 of h space}
Let $\esc{X}$ be a $k$-motivic $H$-group space (i.e., a group object in the homotopy $1$-category $h\mathcal{H}^{mot}(k)$) over a perfect field $k$. Then $\pi_0^{mot}\esc{X}\simeq \pi_0^{nis}\esc{X}$. In other words, a motivic $H$-group space is $0$-fundamental.
\end{prop}
\begin{proof}
This is [\cite{choudhury2014connectivity}, Theorem 4.18] combined with \Cref{piimot is Lpiinis}.
\end{proof}

Despite all this, it is now known, thanks to Ayoub, that there are motivic spaces $\esc{X}$ such that $\pi_0^{nis}\esc{X}$ is not $\mathbb{A}^1$-invariant [\cite{ayoubcounterexamples}], i.e., $\pi_0^{nis}\esc{X}\to \pi_0^{mot}\esc{X}$ is not an isomorphism. At best, what is known in general is the following:

\begin{prop}\label[prop]{pi0 A1 inv in one variable}
Let the base field be $k$. Then for a $k$-motivic space $\esc{X}$, the canonical map $\pi_0^{\mathbb{A}^1}\esc{X}\simeq \pi_0^{nis}\esc{X}\to \pi_0^{mot}\esc{X}$ is an isomorphism on $\mathcal{F}_k$ and on $\mathbb{A}^1_{\mathcal{F}_k}$.
\end{prop}
\begin{proof}
This is a restatement of [\cite{choudhury2022characterisation}, Corollary 2.15].
\end{proof}

\begin{defn}\label[defn]{defn motivic connectivity}
Following \Cref{L n connective} and \Cref{L connected and LC connected} for $L_{mot} \colon \mathcal{P}_{nis}(S) \to \mathcal{H}^{mot}(S)$, we say that:
\begin{enumerate}
    \item For $n \geq -2$, an $S$-space $\esc{X}$ is \emph{$\mathbb{A}^1$-$(n+1)$-connective} if $\tau_{\leq n}^{nis}(L_{mot}\esc{X}) \simeq *$, and \emph{motivically $(n+1)$-connective} if $\tau_{\leq n}^{mot}(L_{mot}\esc{X}) \simeq *$ (for $n = 0$, \emph{$\mathbb{A}^1$-connected} and \emph{motivically connected}, i.e., $\pi_0^{\mathbb{A}^1}(\esc{X}) \simeq *$ and $\pi_0^{mot}(L_{mot}\esc{X}) \simeq *$). For $m \geq -1$, $\tau_{\geq m}^{mot}\mathcal{H}^{mot}(S)$ and $\tau_{\geq m}^{nis}\mathcal{H}^{mot}(S) := \mathcal{H}^{mot}(S) \mcap \tau_{\geq m}^{nis}\mathcal{P}_{nis}(S)$ denote the full subcategories of motivically $m$-connective and $\mathbb{A}^1$-$m$-connective motivic spaces.
    \item For $n \geq 0$, a pointed $S$-space $\esc{X}$ is \emph{$\mathbb{A}^1$-$n$-connected} if $\pi_i^{\mathbb{A}^1}(\esc{X}) \simeq *$ for all $0 \leq i \leq n$, and \emph{motivically $n$-connected} if $\pi_i^{mot}(L_{mot}\esc{X}) \simeq *$ for all $0 \leq i \leq n$. For $m \geq 1$, $\mathcal{H}^{mot}(S)_{\geq m}$ (or $\mathcal{H}_*^{mot}(S)_{\geq m}$) denotes the full subcategory of $\mathcal{H}^{mot}(S)_*$ spanned by motivically $(m-1)$-connected spaces (with $\mathcal{H}^{mot}(S)_{\geq 0} := \mathcal{H}^{mot}(S)_*$).
\end{enumerate}
\end{defn}

\begin{lem}\label[lem]{connected vs connective mot}
\begin{enumerate}
    \item For $-1\leq i\leq 0$, an $S$-space $\esc{X}$ is $\mathbb{A}^1$-$i$-connective if and only if it is motivically $i$-connective, if and only if $\esc{X}$ is Nisnevich-$i$-connective.
    \item For $n\geq 0$, a pointed $S$-space is $\mathbb{A}^1$-$(n+1)$-connective if and only if it is $\mathbb{A}^1$-$n$-connected.
\end{enumerate}
\end{lem}
\begin{proof}
(1) follows from \Cref{-1-2 mot trunc}. (2) is a statement in $\mathcal{P}_{nis}$ which is well known.
\end{proof}

The above implies, in particular, that a motivic space is motivically $(-2)$-, $(-1)$-connected if and only if it is Nisnevich-locally $(-2)$-, $(-1)$-connected over an arbitrary scheme. This brings us to the $n=0$ case, which leads to one of the most important observations of this paper. At this point, we must specialize to the case where the base is a field and establish the comforting fact that a motivic space is motivically $0$-connected if and only if it is Nisnevich-locally $0$-connected. In other words, the motivic localization functor is path-injective (see \Cref{path injectivity}).

\begin{lem}\label[lem]{motivic 0 truncation is injective}
Let $k$ be a perfect field. Then any motivic $k$-space is motivically connected if and only if it is Nisnevich-connected, if and only if it is $\mathbb{A}^1$-connected. In detail, let $\tau_{\leq 0}^{nis}\mathcal{H}^{mot}(k)$ denote the image of $\mathcal{H}^{mot}(k)$ under $\tau_{\leq 0}^{nis}$ (we call it the category of $\mathbb{A}^1$-homotopy path components). Then the restriction functor $$L_{\leq 0}^{mot}\simeq \mathcal{L}:\tau_{\leq 0}^{nis}\mathcal{H}^{mot}(k)\to Sh_{nis}^{\mathbb{A}^1}(k)$$ detects the terminal object. In other words, the motivic localization is path-injective (\Cref{path injectivity}).
\end{lem}
\begin{proof}
Using the isomorphism on $\mathcal{F}_k$ as in \Cref{pi0 A1 inv in one variable}, this follows immediately from [\cite{bachmann2024strongly}, Corollary 2.10].
\end{proof}

\begin{rem}
We include the full argument via ghost homotopies for completeness. First, since $\pi_0\esc{X}(K)\simeq \pi_0^{nis}\esc{X}(K)$ for every $K\in \mathcal{F}_k$, we want to show that for any $K\in \mathcal{F}_k$, any two $t,t'\in \esc{X}_0(K)$ are connected by a simplicial homotopy in $\esc{X}(K)$. Since $\mathcal{L}\pi_0^{nis}\esc{X}(K)=*$, there is an $n$-ghost homotopy $$H=(\pi_0^{nis}\esc{X}\xleftarrow{h}V\xrightarrow[]{nis} \mathbb{A}^1_K,h^W:W\mathrel{\substack{\xrightarrow{ha} \\ \xrightarrow[hb]{}}} \pi_0^{nis}\esc{X}, \tilde{\sigma}_i),$$ where the maps $a$ and $b$ are the following composites (in which the first map is a Nisnevich cover): $$a,b: W\xrightarrow{nis}V\times_{\mathbb{A}^1_K}V\mathrel{\substack{\xrightarrow{pr_1} \\ \xrightarrow[pr_2]{}}} V.$$ Note that since $V$ is a $1$-dimensional scheme, it has Nisnevich homotopy dimension $1$, and $\esc{X}$ satisfies Nisnevich descent. By [\cite{bachmann2024strongly}, Lemma 2.17], we have a surjection $[V,\esc{X}]\to [V, \pi_0^{nis}\esc{X}]$. Therefore, up to homotopy, there is a map $u: V\to \esc{X}$ lifting $h$.

Now, we claim that for every $K\in \mathcal{F}_k$, any two sections $t,t'$ as above are simplicially homotopic. We prove this by induction on $n$. As in [\cite{10.2140/akt.2022.7.385}, Theorem 2.2], we may assume that $V$ is an elementary Nisnevich cover $V=V_1\sqcup V_2$, let $h_i=h_{|V_i}$ and $u_i=u_{|V_i}$, and denote $W'=V_1\times _{\mathbb{A}^1_K}V_2$. Let $W'=\bigcup W'_j$ be the decomposition into irreducible components of $W'$, and let $\eta_j$ be the corresponding generic points. We deduce that $H'_{\eta_j}$ is an $(n-1)$-ghost homotopy between ${(u_1)}_{\eta_j}$ and $(u_2)_{\eta_j}$. By the induction hypothesis, this implies that $(u_1)_{\eta_j}$ and $(u_2)_{\eta_j}$ are simplicially homotopic in $\esc{X}$. By definition, these simplicial homotopies spread out to an open subset $W''\subset W'$. Consequently, $(u_1)_{|W''}$ and $(u_2)_{|W''}$ are simplicially homotopic in $\esc{X}$. Moreover, by further shrinking $V_2$, we may assume that $W''= V_1 \times _{\mathbb{A}^1_K}V_2$.

Therefore, we may assume that $u_1:V_1\to \esc{X}$ and $u_2: V_2\to \esc{X}$ agree on $V_1\times_{\mathbb{A}^1_K}V_2$ up to a simplicial homotopy. Since $V=V_1\sqcup V_2\to \mathbb{A}^1_K$ is an elementary Nisnevich cover (so the associated Nisnevich distinguished square is a pushout in $\mathcal{P}_{nis}(k)$) and $\esc{X}$ is a Nisnevich $\infty$-sheaf, we may therefore descend $u: V\to \esc{X}$ to a map $u': \mathbb{A}^1_K\to \esc{X}$. Since $\esc{X}$ is $\mathbb{A}^1$-local, this yields a simplicial homotopy between the endpoints, namely $t$ and $t'$.

The base case is easy. Indeed, a $0$-ghost homotopy from $t$ to $t'$ in $\pi_0^{nis}\esc{X}$ is given by an explicit map $\mathbb{A}^1_K\to \pi_0^{nis}\esc{X}$, which, by appeal to the same homotopy dimension argument [\cite{bachmann2024strongly}, Lemma 2.17], lifts to $[\mathbb{A}^1_K, \esc{X}]$ and, by the $\mathbb{A}^1$-locality of $\esc{X}$, gives a simplicial homotopy.
\end{rem}

To settle the more general cases of $n$-connected spaces, we recall Morel's famous result on the fundamentality of the motivic localization functor, which again requires the base to be a perfect field:

\begin{prop}\label[prop]{layers of connected motivic space}
Let $\esc{X}\in \mathcal{P}_{nis}(k)_{*,\geq 1}$ be a pointed $k$-space over a perfect field $k$. Then the following conditions are equivalent:
\begin{enumerate}
    \item $\esc{X}\in \mathcal{H}^{mot}(k)$.
    \item $\mathrm{B}^{i}_{nis}\pi_i^{nis}\esc{X}\in \mathcal{H}^{mot}(k)$ for all $i\geq 1$.
    \item $\tau_{\leq i}^{nis}\esc{X}\in \mathcal{H}^{mot}(k)$ for all $i\geq 0$.
\end{enumerate}
In this situation, $\tau^{nis}_{\leq i}\esc{X}\simeq \tau^{mot}_{\leq i}\esc{X}$, i.e., $L_{mot}$ is $i$-fundamental on $\esc{X}$ for all $i\geq 0$.
\end{prop}
\begin{proof}
$(1)\implies (2)$ follows from [\cite{morel2012a1}, Corollary 6.2]. Suppose (2) holds. Since $\esc{X}$ is Nisnevich-locally connected, $\tau_{\leq 0}^{nis}\esc{X}\simeq *\in \mathcal{H}^{mot}(k)$. Inductively, in the fiber sequence $$\mathrm{B}^{i+1}_{nis}\pi^{nis}_{i+1}(\esc{X})\to \tau^{nis}_{\leq i+1}\esc{X}\to\tau^{nis}_{\leq i}\esc{X},$$ the fiber lies in $\mathcal{H}^{mot}(k)$ by (2), so $\tau^{nis}_{\leq i+1}\esc{X}\in \mathcal{H}^{mot}(k)$ by [\cite{morel2012a1}, Lemma 6.51]. Finally, if (3) holds, the limit-closure of $\mathcal{H}^{mot}(k)\subset \mathcal{P}_{nis}(k)$ implies $\plim_{n}\tau^{nis}_{\leq n}\esc{X}\in \mathcal{H}^{mot}(k)$, and the claim follows from the equivalence $\esc{X}\simeq \plim_{n}\tau^{nis}_{\leq n}\esc{X}$ in $\mathcal{P}_{nis}(k)$ (due to Postnikov convergence of $\mathcal{P}_{nis}(k)$). For the final statement, since $\esc{X}$ and $\tau_{\leq i}^{nis}\esc{X}$ are motivic local, $(\tau^{nis}_{\leq i}L_{\mathbb{A}^1})^m\esc{X}\simeq \tau_{\leq i}^{nis}\esc{X}$ for all $m$, so \Cref{motivic truncation formulae} gives $\tau^{mot}_{\leq i}\esc{X}\simeq \tau_{\leq i}^{nis}\esc{X}$.
\end{proof}

\begin{rem}\label[rem]{Lmot commute trunc on connected objects}
Combined with \Cref{nis trunc preserve mot equiv}, Morel's result (\Cref{layers of connected motivic space}) leads to the following nice fact: on $\mathcal{P}_{nis}(k)_{\geq 1}$, the localization $L_{mot}:\mathcal{P}_{nis}(k)\to \mathcal{H}^{mot}(k)$ commutes with truncations whenever $n\geq 0$, i.e., for $\esc{X} \in \mathcal{P}_{nis}(k)_{\geq 1}$, $$ \tau_{\leq n}^{mot}(L_{mot}\esc{X})\simeq \tau_{\leq n}^{nis}(L_{mot}\esc{X}) \simeq L_{\leq n}^{mot}(\tau_{\leq n}^{nis}\esc{X}).$$
\end{rem}

\begin{cor}\label[cor]{mot conn if and only if trunc trivial}
Let $\esc{X}\in \mathcal{H}^{mot}(k)$ be a motivic space over a perfect field $k$. Then $\tau_{\leq n-1}^{mot}\esc{X}\simeq *$ if and only if $\tau_{\leq n-1}^{nis}\esc{X}\simeq *$. Hence, $$\tau_{\geq n}^{mot}\mathcal{H}^{mot}(k)\simeq \tau_{\geq n}^{nis}\mathcal{H}^{mot}(k):=\mathcal{H}^{mot}(k)\mcap \mathcal{P}_{nis}(k)_{\geq n}.$$
When $n\geq 1$, the pointed versions of these categories are further equivalent to $\mathcal{H}^{mot}(k)_{\geq n}$.
\end{cor}
\begin{proof}
Since $L_{mot}$ is path-injective (\Cref{motivic 0 truncation is injective}), the cases $n\geq 1$ follow from \Cref{tau preserves connectivity} using \Cref{layers of connected motivic space} (since $k$ is a point for the Nisnevich topology on $Sm_k$, the existence of a point for a Nisnevich connected $k$-space is automatic; this explains the first equivalence of the unpointed categories). It remains to deal with the cases $n=-2,-1$, but this is immediate from \Cref{-1-2 mot trunc}.
\end{proof}

\begin{rem}[On the Postnikov completeness of the connected motivic homotopy category]\label[rem]{post convg of conn}
While Postnikov completeness of the full motivic homotopy category remains open in general, the presentable $\infty$-category of pointed connected motivic spaces $\mathcal{H}^{mot}(k)_{\geq 1} = \tau^{mot}_{\geq 1}\mathcal{H}^{mot}_*(k)$ is Postnikov complete over any perfect field $k$. Indeed, $\mathcal{H}^{mot}_*(k)_{\geq 1} = \mathcal{P}_{nis}(k)_{*,\geq 1} \mcap \mathcal{H}^{mot}_*(k)$ (\Cref{mot conn if and only if trunc trivial}), and for both $\mathcal{H}^{mot}_*(k)_{\geq 1}$ and $\mathcal{P}_{nis}(k)_{*,\geq 1}$ the intrinsic categorical $n$-truncation (in the sense of [\cite{lurie2009higher}, Proposition 5.5.6.18]) coincides with $\tau^{mot}_{\leq n+1} \simeq \tau^{nis}_{\leq n+1}$ for all $n \geq 0$ (\Cref{layers of connected motivic space}), as both categories are generated under colimits by suspensions $\Sigma(U_+)$ for $U \in Sm_k$. Since $\mathcal{H}^{mot}_*(k) \subset \mathcal{P}_{nis}(k)_*$ is closed under limits, it suffices to show that $\mathcal{P}_{nis}(k)_{*,\geq 1}$ is Postnikov complete. This follows from the Postnikov completeness of $\mathcal{P}_{nis}(k)$ ([\cite{MR4296353}, Theorem 3.18]; [\cite{lurie2009higher}, Corollary 7.2.1.12]): $\mathcal{P}_{nis}(k)_{*,\geq 1}$ is closed under $\tau^{nis}_{\leq m}$ ($m \geq 1$), so every object is the limit of its Postnikov tower, and conversely the limit of any Postnikov tower $\{\esc{X}_m\}_{m \geq 1}$ in $\mathcal{P}_{nis}(k)_{*,\geq 1}$ remains connected since by definition of a Postnikov tower $\pi_0^{nis}\!\left(\plim_m \esc{X}_m\right) \simeq \pi_0^{nis}\esc{X}_1 \simeq *$.
\end{rem}
\subsection{Motivic Eilenberg-MacLane spaces}

\begin{defn}\label[defn]{defn motivic em space}
Following \Cref{defn em object}, for $n\geq 0$, we say that a pointed motivic space $\esc{X}\in \mathcal{H}^{mot}(S)_*$ is a \emph{motivic $n$-Eilenberg-MacLane space} if it is motivically $n$-truncated and motivically $n$-connective. We denote the full subcategory of motivic $n$-Eilenberg-MacLane spaces by $\mathcal{EM}_n(\mathcal{H}^{mot}(S)):=\mathcal{H}^{mot}_{\leq n}(S)_*\mcap \mathcal{H}^{mot}(S)_{\geq n}$.
\end{defn}

\begin{thm}\label[thm]{motivic em vs nis em}
For every $n\geq 1$ and any perfect field $k$,
\[
    \mathcal{H}^{mot}(k)\mcap \mathcal{EM}_{n}(\mathcal{P}_{nis}(k))=\mathcal{EM}_n(\mathcal{H}^{mot}(k))=\tau_{\geq n}^{mot}\mathcal{H}^{mot}_{\leq n}(k)_*.
\]
\end{thm}
\begin{proof}
The first equality follows from \Cref{motivic 0 truncation is injective} and \Cref{em of localization}. The second follows immediately from the definition $\mathcal{EM}_n(\mathcal{H}^{mot}(k)):=\mathcal{H}^{mot}_{\leq n}(k)_*\mcap \mathcal{H}^{mot}(k)_{\geq n}$ and the identification $\mathcal{H}^{mot}(k)_{\geq n}=\tau_{\geq n}^{mot}\mathcal{H}^{mot}(k)_*$ (\Cref{mot conn if and only if trunc trivial}).
\end{proof}

This gives a nice alternate presentation of \Cref{layers of connected motivic space}:
\begin{prop}\label[prop]{internal layers of connected motivic space}
Let $\esc{X}\in \mathcal{P}_{nis}(k)_{*,\geq 1}$ be a pointed $k$-space over a perfect field $k$. Then the following conditions are equivalent:
\begin{enumerate}
    \item $\esc{X}\in \mathcal{H}^{mot}(k)$.
    \item $\mathrm{B}^{i}_{nis}\pi_i^{nis}\esc{X}\in \mathcal{EM}_i(\mathcal{H}^{mot}(k))$ for all $i\geq 1$.
    \item $\tau_{\leq i}^{nis}\esc{X}\in \mathcal{H}^{mot}_{\leq i}(k)$ for all $i\geq 0$.
\end{enumerate}
\end{prop}

Let $Sh_S^{\mathbb{A}^1} := Sh_{nis}^{\mathbb{A}^1}(S)$, and let $L_{\mathbb{A}^1}\mathrm{Grp}_S \subset \mathrm{Grp}_S := \mathrm{Grp}(Sh_{nis}(S))$ and $L_{\mathbb{A}^1}\mathrm{Ab}_S \subset \mathrm{Ab}_S := \mathrm{Ab}(Sh_{nis}(S))$ denote the full subcategories consisting of $\mathbb{A}^1$-invariant sheaves of groups and abelian groups, respectively.

\begin{lem}\label[lem]{internal bar pi adjunct}
For an arbitrary base scheme $S$, there is an adjunction $\mathrm{B}_{mot}^n:n\mbox{-}\mathrm{Grp}(Sh_S^{\mathbb{A}^1})\leftrightarrows \mathcal{H}^{mot}(S)_*: \pi_n^{mot}$, where $\mathrm{B}_{mot}$ denotes the internal bar construction.
\end{lem}
\begin{proof}
This follows from the path-cartesianness of $L_{mot}$ (\Cref{mot truncation cart}(1)) and \Cref{bar pi adjunct}.
\end{proof}

\begin{defn}\label[defn]{defn strong and strict A1 invariance}
Following \Cref{strictly n local}, a Nisnevich sheaf of $n$-groups $G\in n\mbox{-}\mathrm{Grp}(Sh_{nis}(k))$ is \emph{$n$-strictly $\mathbb{A}^1$-invariant} if $\mathrm{B}^n_{nis}G$ is $\mathbb{A}^1$-local, or equivalently if $H^i_{nis}(-,G)$ is $\mathbb{A}^1$-invariant for all $0\leq i\leq n$. We denote the full subcategory of $n$-strictly $\mathbb{A}^1$-invariant $n$-group sheaves by $n\mbox{-}\mathrm{Grp}^{L_{mot}}$.
\end{defn}

For $n=0$, $0\mbox{-}\mathrm{Grp}^{L_{mot}}=Sh^{\mathbb{A}^1}_{k,*}$. For $n=1$, this matches [\cite{morel2012a1}, Definition 1.7] (strongly $\mathbb{A}^1$-invariant sheaves of groups), and, in Morel's notation $1\mbox{-}\mathrm{Grp}^{L_{mot}}=\mathrm{Grp}_k^{\mathbb{A}^1}$. Moreover, writing $\mathrm{Ab}_k^{\mathbb{A}^1}$ for the category of strictly $\mathbb{A}^1$-invariant sheaves of abelian groups, by [\cite{morel2012a1}, Theorem 4.46]:
\begin{lem}\label[lem]{n-strict groups}
Let $k$ be a perfect field. Then, for all $n\geq 2$, $\mathrm{Ab}_k^{\mathbb{A}^1}=2\mbox{-}\mathrm{Grp}^{L_{mot}}=n\mbox{-}\mathrm{Grp}^{L_{mot}}$.
\end{lem}
We shall henceforth use these identifications.

\begin{thm}\label{motivic em}
Let $k$ be a perfect field. Then the adjunction of \Cref{internal bar pi adjunct} restricts to equivalences of categories:
\[
    Sh_{k,*}^{\mathbb{A}^1}\simeq \mathcal{EM}_0(\mathcal{H}^{mot}(k)), \quad \mathrm{Grp}_k^{\mathbb{A}^1}\simeq \mathcal{EM}_1(\mathcal{H}^{mot}(k)): \pi_1^{mot}, \quad \text{and} \quad \mathrm{Ab}_k^{\mathbb{A}^1}\simeq \mathcal{EM}_n(\mathcal{H}^{mot}(k)):\pi_n^{mot}
\]
for all $n\geq 2$.
\end{thm}
\begin{proof}
Since $L_{mot}$ is path-injective, this follows from \Cref{grp em} and \Cref{n-strict groups}.
\end{proof}

To contrast this with \Cref{em in infinity topos}, let us more closely analyze the inclusions $$1\mbox{-}\mathrm{Grp}^{L_{mot}}=\mathrm{Grp}_k^{\mathbb{A}^1}\subset L_{\mathbb{A}^1}\mathrm{Grp}_k\subset Sh_{k,*}^{\mathbb{A}^1}\quad \textup{and }\quad 2\mbox{-}\mathrm{Grp}^{L_{mot}}=\mathrm{Ab}^{\mathbb{A}^1}_k \subset L_{\mathbb{A}^1}\mathrm{Ab}_k \subset Sh_{k,*}^{\mathbb{A}^1}.$$ It is well known that, unlike in an $\infty$-topos, the inclusion $\mathrm{Ab}^{\mathbb{A}^1}_k \subsetneq L_{\mathbb{A}^1}\mathrm{Ab}_k$ is strict. Previously, essentially only one example of this kind was available in the literature [\cite{choudhury2014connectivity}, \S5]; in the rest of this section, we shall generate additional examples of this type, including non-abelian ones. More specifically, we would like to show that group objects that come from the left adjoint to the forgetful functor $\mathrm{Grp}_k\subset Sh_{nis}(k)_*$ and lie in the smaller subcategory $\mathrm{Grp}^{\mathbb{A}^1}_k$ must necessarily be birational and hence must have a birational sheaf as source. On the other hand, we will also demonstrate that on a sufficiently larger category than the birational one, the left adjoint $Sh_{nis}(k)_*\to \mathrm{Grp}_k$ preserves $\mathbb{A}^1$-locality, despite the well-known issue that Nisnevich sheafification generally destroys $\mathbb{A}^1$-invariance. We shall explain this point first:

\underline{\textit{$\mathbb{A}^1$-invariance of freely generated sheaves.}}
We restrict our attention to discrete motivic spaces. Recall that every discrete $\mathbb{A}^1$-invariant Nisnevich sheaf of groups is weakly unramified (i.e., induces injections along dense open immersions) [\cite{bachmann2024strongly}, Corollary 2.4]. While the converse is generally false, we shall show that once the unramifiedness condition is strengthened, free constructions do preserve both properties (Nisnevich and $\mathbb{A}^1$-locality):

\begin{defn}\label[defn]{tau weakly unramified}
Let $\tau$ be a property of morphisms. We say that a presheaf $F$ is \emph{$\tau$-weakly unramified} if for every dominant $\tau$-morphism $f:X\to Y$, the restriction $F(Y)\to F(X)$ is injective.
\end{defn}

\begin{rem}
It is evident that the ordinary notion of weak unramifiedness is equivalent to open-immersion-weak unramifiedness.
\end{rem}

\begin{ex}\label[ex]{kahler}
The presheaf of global sections of the sheaf of K\"ahler differentials $X\mapsto \Omega_{X/k}(X)$ is \'etale-weakly unramified. Indeed, if $f:X\to Y$ is an \'etale morphism of smooth schemes over $k$, then $f^*\Omega_{Y/k}\simeq \Omega_{X/k}$. On the other hand, since $\Omega_{Y/k}$ is a flat $\mathcal{O}_Y$-module and $\mathcal{O}_Y\to f_*\mathcal{O}_X$ is injective (since $f$ is dominant), it follows that the morphism of $\mathcal{O}_Y$-modules $$\Omega_{Y/k}\to f_*\mathcal{O}_X\otimes _{\mathcal{O}_Y}\Omega_{Y/k}\simeq f_*f^*\Omega_{Y/k}\simeq f_*\Omega_{X/k}$$ is an injective $\mathcal{O}_Y$-module homomorphism. Taking global sections yields the injection:
\[
    \Gamma(Y,\Omega_{Y/k})\to \Gamma(Y,f_*\Omega_{X/k})\simeq \Gamma(X,\Omega_{X/k}).
\]
Note, however, that this presheaf is not $\mathbb{A}^1$-invariant.
\end{ex}

\begin{ex}\label[ex]{fully weak un}
When $\tau$ is the class of all morphisms, we say that a $\tau$-weakly unramified presheaf is \emph{fully weakly unramified}. This is equivalent to turning all dominant maps into injections. A natural first candidate for a fully weakly unramified presheaf is the assignment $X\mapsto k(X)$, where $k(X)$ denotes the ring of rational functions on $X$. By definition, this assignment sends dominant morphisms to injective maps. However, it is not functorial with respect to arbitrary morphisms and hence does not define a presheaf on $Sm_k$. Nevertheless, this idea can be suitably generalized to construct presheaves with the desired property. Specifically, let $G$ be a weakly unramified presheaf (i.e., it induces injections on generic point inclusions) whose restriction to $\mathcal{F}_k$ sends field extensions to injections; then $G$ is fully weakly unramified. Indeed, for a dominant morphism $f:X\to Y$ of connected smooth schemes, the composition $G(Y)\to G(X)\hookrightarrow G(k(X))$ is equal to the composition of injections $G(Y)\hookrightarrow G(k(Y))\hookrightarrow G(k(X))$.
\end{ex}

\begin{ex}\label[ex]{milnor is fully weakly un}
For every $n$, the Milnor $K$-theory sheaf $\underline{\mathrm{K}}^M_n$ satisfies part of the above condition. Indeed, $\mathrm{K}^M_n$ is injective on purely transcendental extensions [\cite{MR3076731}, Example 7.3.2]. Unfortunately, it is not injective on arbitrary finite extensions. However, by Corollary 7.5.3 of \textit{loc.\ cit.}, it follows that $\underline{\mathrm{K}}^M_n \otimes \mathbb{Q}$  is injective on all field extensions. Since $\mathbb{Q}$ is a flat $\mathbb{Z}$-module and $\underline{\mathrm{K}}^M_n$ is weakly unramified, $\underline{\mathrm{K}}^M_n\otimes \mathbb{Q}$ has all the required properties and is thus fully weakly unramified. (Note that the functor $-\otimes_\mathbb{Z}\mathbb{Q}$ preserves Nisnevich sheaves: the point is that $\mathbb{Q}$ is a filtered colimit $\mathbb{Q} \cong \colim\{\cdot n:\mathbb{Z}\to \mathbb{Z}\}$, so the sheaf tensor product above is computed sectionwise.)
\end{ex}

\begin{rem}
The analog of the above example fails for Milnor-Witt $K$-theory. For example, $\mathrm{GW}(\mathbb{R})\otimes \mathbb{Q}\simeq \mathbb{Q}\oplus \mathbb{Q}$, whereas $\mathrm{GW}(\mathbb{C})\otimes \mathbb{Q}\simeq \mathbb{Q}$.
\end{rem}

\begin{rem}
Using the general criterion of \Cref{fully weak un}, it is easy to show, extending the result of \Cref{kahler}, that $\Gamma(-,\Omega_{-/k})$ is fully weakly unramified when $k$ is a field of characteristic $0$.
\end{rem}

\begin{ex}[Separated schemes]\label[ex]{sep scheme is fully-weak unram}
Every separated scheme turns dominant maps into injections, i.e., is fully weakly unramified. Indeed, let $Z$ be a separated scheme over $k$, and let $h: X\to Y$ be a dominant morphism of connected smooth schemes. Suppose $f,g: Y\to Z$ are two $k$-morphisms such that $f\circ h=g\circ h$. It follows that $h$ factors through the equalizer $E\hookrightarrow Y$ (which is a closed immersion due to the separatedness of $Z$). The dominance of $h$ then implies that $E\hookrightarrow Y$ is dominant. Because $Y$ is smooth and connected (hence integral), it follows that the equalizer $E\hookrightarrow Y$ is an isomorphism, so $f=g$.
\end{ex}

Here is the key result:
\begin{lem}\label[lem]{free nis_excision_G}
Suppose $F \colon \mathrm{Sets} \to \mathrm{Sets}$ preserves intersection squares and $G$ is an \'etale-weakly unramified Nisnevich sheaf of sets on $Sm_k$. Then $L_\Sigma FG \simeq L_{nis} FG$.
\end{lem}

Before proving the lemma, we list some examples of endofunctors of the type used in the statement above.

\begin{ex}[Functors preserving intersection squares]\label[ex]{adding free structure preserve intersection square}\label[ex]{set theoretic example}
The functors $F = F_{mon}, F_{grp}, \mathbb{Z}[-]$, i.e., the free monoid, free group, and free abelian group functors (and any composition of these functors), as well as the cotensor $(-)^K$ with a fixed set $K$ and the (covariant) non-empty power set functor ${2^{(-)}}_+$, all preserve intersection squares.
\end{ex}

In order to prove \Cref{free nis_excision_G}, we shall need a crucial lemma that is very specific to the base being a field.

\begin{lem}\label[lem]{nis exc on conn}
A presheaf $\esc{F}\colon Sm_k^{op}\to\mathrm{Spc}$ satisfies Nisnevich excision if and only if $\esc{F}\in \mathcal{P}_{\Sigma}(k)$ and its restriction to the full subcategory $Sm_k^{conn}$ of connected schemes satisfies Nisnevich excision.
\end{lem}
\begin{proof}
The necessity holds by definition. We shall explain the sufficiency. For this, we first remind the reader of two important facts when the base is a field:
\begin{enumerate}
    \item A smooth $k$-scheme is irreducible if and only if it is connected, i.e., $Sm^{conn}_k=Sm_k^{irr}$.
    \item All smooth $k$-schemes (being quasi-compact by our standing convention) are Noetherian and hence any closed subscheme has finitely many connected/irreducible components.
\end{enumerate}
Now, let the following be a Nisnevich distinguished square $Q$ of non-empty schemes (the empty cases being trivial):
\[
    \xymatrix{
        U' \ar[r]^{j'} \ar[d]_{\pi'} & X' \ar[d]^\pi \\
        U \ar[r]^j & X
    }
\]
Decomposing $X$ into its connected (hence irreducible) components $X_i$, it is evident that the above square is a disjoint union of Nisnevich squares $Q_i$ over irreducible bases. Since a finite product of pullback squares in $\mathrm{Spc}$ is a pullback square and $\esc{F}\in \mathcal{P}_\Sigma(k)$, we are reduced to the case where $X$ is irreducible. In that case, since $j$ is a dense open immersion, $U$ is irreducible as well. Unfortunately, since $\pi$ is only \'etale, there is no reason for $X'$ to be irreducible. Had $\pi$ been an open immersion, $X'$ and $U'$ would also be irreducible (or empty) and we would clearly be done. In other words, $\esc{F}$ already satisfies Zariski excision. For a general \'etale $\pi$, let $Z=X\setminus U$. Any connected component of $X'$ that does not intersect $\pi^{-1}(Z)$ is mapped isomorphically onto its image by $j'$, contributing a trivial pullback factor. So we may assume that $\pi^{-1}(Z)$ intersects every connected component of $X'$. If $Z$ is connected, the isomorphism $\pi^{-1}(Z)\xrightarrow{\sim} Z$ implies that $\pi^{-1}(Z)\subset X'$ is connected, and hence $X'$ is connected by the previous reduction, so we are done. Otherwise, let $\{Z_i\}$ be the (finite) set of connected components of $Z$, and set $$U_i:=X\setminus \Big(\bigsqcup_{j\neq i}Z_j\Big).$$ Since $\{U_i\}$ is a Zariski open cover of $X$ (consisting of irreducible schemes since $X$ itself is irreducible), by Zariski excision (which we have already established), we are reduced to checking the Nisnevich excision condition after base change to each $U_i$. By construction, over $U_i$ the closed complement $Z_i$ is connected, completing the proof.
\end{proof}

We now prove \Cref{free nis_excision_G}:
\begin{proof}[Proof of \Cref{free nis_excision_G}]
We want to show that $L_\Sigma FG$ is already a Nisnevich sheaf. Since points for the coproduct topology (yielding the $L_\Sigma$-localization) on $Sm_k$ are precisely the connected schemes, $(L_\Sigma FG)(X) = F(G(X))$ for every connected $X\in Sm_k$. By \Cref{nis exc on conn}, we are reduced to checking the Nisnevich excision condition for $FG$ along Nisnevich distinguished squares of non-empty connected schemes.

Observe that every morphism in such a Nisnevich square of connected schemes is dominant. Since $G$ is an \'etale-weakly unramified Nisnevich sheaf, it turns such squares into intersection (pullback) squares of injections in $\mathrm{Sets}$. Since $F$ preserves intersection squares, $FG$ sends such squares to pullback squares of sets, and we are done.
\end{proof}

Such ``intersection-preserving'' transformations therefore preserve \'etale-weakly unramified motivic sheaves of sets. This follows from the elementary but crucial observation that $L_\Sigma$ preserves $\mathbb{A}^1$-invariance:

\begin{lem}\label[lem]{L_sigma preserves a1 local}
Over a field $k$, $L_\Sigma$ preserves $\mathbb{A}^1$-local objects.
\end{lem}
\begin{proof}
Let $\esc{F}$ be an $\mathbb{A}^1$-invariant $k$-space. Since over a field every smooth scheme is a finite disjoint union of connected schemes and $\mathbb{A}^1 \times_k -$ preserves disjoint unions and connectedness, the $\mathbb{A}^1$-invariance of $L_\Sigma\esc{F}$ reduces to that of $\esc{F}$ on connected schemes, since $(L_\Sigma\esc{F})(X) \simeq \esc{F}(X)$ on connected schemes.
\end{proof}

\begin{lem}\label[lem]{free mot_excision_G}
Suppose $F \colon \mathrm{Sets} \to \mathrm{Sets}$ preserves intersection squares and $G$ is an \'etale-weakly unramified motivic sheaf of $k$-sets. Then $L_{nis} FG\simeq L_\Sigma FG$ is a motivic sheaf.
\end{lem}
\begin{proof}
This is immediate from \Cref{L_sigma preserves a1 local} and \Cref{free nis_excision_G}.
\end{proof}

\begin{cor}\label[cor]{mot if and only if free mot}
Let $A$ be an \'etale-weakly unramified Nisnevich sheaf of sets, and let $F: \mathrm{Sets} \to \mathrm{Sets}$ be a conservative functor that preserves intersection squares. Then $A$ is $\mathbb{A}^1$-invariant if and only if $L_{nis} FA$ is $\mathbb{A}^1$-invariant. In particular, $A$ is $\mathbb{A}^1$-invariant if and only if $L_{nis} F_{mon}A$ is $\mathbb{A}^1$-invariant, if and only if $L_{nis} F_{grp}A$ is $\mathbb{A}^1$-invariant, if and only if $L_{nis}\mathbb{Z}[A]$ is $\mathbb{A}^1$-invariant.
\end{cor}
\begin{proof}
Since $F_{grp},F_{mon},\mathbb{Z}[-]:\mathrm{Sets}\to \mathrm{Sets}$ are conservative and preserve intersection squares, it suffices to establish the first statement. By \Cref{free nis_excision_G}, we know that $L_{nis} FA\simeq L_{\Sigma} FA$, and $\mathbb{A}^1$-invariance need only be checked on connected smooth schemes. So, it suffices to know that by the conservativity of $F$, for any connected $X\in Sm_k$, $A(X)\to A(\mathbb{A}^1_X)$ is a bijection if and only if $FA(X)\to FA(\mathbb{A}^1_X)$ is a bijection.
\end{proof}

\begin{ex}\label[ex]{adding free structure to G_m}
By \Cref{adding free structure preserve intersection square} and \Cref{sep scheme is fully-weak unram}, for an $\mathbb{A}^1$-rigid (separated) scheme $A$, the sheaves $L_{nis} F_{mon}A$, $L_{nis} F_{grp}A$, and $L_{nis}\mathbb{Z}[A]$ are $\mathbb{A}^1$-invariant. For example, $L_{nis} F_{mon}\mathbb{G}_m$, $L_{nis} F_{grp}\mathbb{G}_m$, $L_{nis}\mathbb{Z}[\mathbb{G}_m]$, $L_{nis}(\mathbb{G}_m^K)$, and $L_{nis}({2^{\mathbb{G}_m}}_+)$ are $\mathbb{A}^1$-local, as are $L_{nis}F_{grp}(\underline{\mathrm{K}}^M_n \otimes \mathbb{Q})$ and $L_{nis}\mathbb{Z}[\underline{\mathrm{K}}^M_n \otimes \mathbb{Q}]$ (as well as $L_{nis}\mathbb{Q}[\underline{\mathrm{K}}^M_n \otimes \mathbb{Q}]$) by \Cref{milnor is fully weakly un}.
\end{ex}

\underline{\textit{Strong $\mathbb{A}^1$-invariance of free group sheaves generated by motivic sheaves of sets.}} We now turn to the other aspect, namely when free group sheaves are strongly $\mathbb{A}^1$-invariant. To discuss this, we shall set up (and recall) a list of definitions as in the previous sub-subsection.

\begin{defn}\label[defn]{defn birational and generically free}
A presheaf of sets $F$ on $Sm_k$ is
\begin{enumerate}
    \item \emph{birational} if $F(X)\to F(U)$ is a bijection for every dense open immersion $U\hookrightarrow X$ in $Sm_k$.
    \item \emph{$1$-birational} if $F(X)\to F(U)$ is a bijection for every dense open immersion $U\hookrightarrow X$ whose complement has codimension at least $2$.
    \item \emph{generically free} if $F_{|\mathcal{F}_k}$ factors through $\mathrm{FreeGrp}\to \mathrm{Sets}$.
    \item When $F$ is pointed, its \emph{contraction} $F_{-1}$ is defined by $F_{-1}(X):=\ker(F(\mathbb{G}_m\times X)\xrightarrow{\mathrm{ev}_1}F(X))$.
\end{enumerate}
\end{defn}

\begin{lem}\label[lem]{1 bir vs strong}
If $G$ is a presheaf of sets such that $L_\Sigma F_{grp}G$ is strongly $\mathbb{A}^1$-invariant, then $L_\Sigma G$ is $1$-birational.
\end{lem}
\begin{proof}
By a theorem of Morel, strongly $\mathbb{A}^1$-invariant sheaves of groups are $1$-birational local [\cite{bachmann2024strongly}, Lemma 2.8]. Hence, $L_\Sigma F_{grp}G$ is $1$-birational. Since the $1$-birationality of $L_\Sigma G$ is equivalent to the $1$-birationality of $G$ on connected schemes, the claim follows from the conservativity of $F_{grp}$.
\end{proof}

\begin{coex}\label[coex]{main coex 1}
Let $G$ be an \'etale-weakly unramified motivic sheaf of sets that is not $1$-birational local (i.e., there is a dense open immersion $U\hookrightarrow X$ of connected smooth schemes of codimension at least $2$ such that $G(X)\to G(U)$ is not a bijection). Then $L_{nis}F_{grp}G$ is $\mathbb{A}^1$-invariant by \Cref{free mot_excision_G}, but it is not strongly $\mathbb{A}^1$-invariant by \Cref{1 bir vs strong}.

Examples of sheaves of this type are abundant. For instance, let $X$ be an $\mathbb{A}^1$-rigid variety of dimension $>1$ and $Z\subset X$ a smooth closed subvariety of codimension at least $2$. Then $X\setminus Z$ is not $1$-birational by construction, but it is an \'etale-weakly unramified (being a separated scheme) motivic sheaf (being an open subscheme of an $\mathbb{A}^1$-rigid scheme).
\end{coex}

\begin{coex}[Showing $\mathrm{Grp}^{\mathbb{A}^1}_k \subsetneq L_{\mathbb{A}^1}\mathrm{Grp}_k$, non-abelian case]\label[coex]{main coex 2}
Over a perfect field $k$, the free Nisnevich sheaf of groups generated by $\mathbb{G}_m$, denoted $L_{nis}F_{grp}(\mathbb{G}_m)$, is an $\mathbb{A}^1$-invariant sheaf of (non-abelian) groups (see \Cref{adding free structure to G_m}). It follows that the reduced free group sheaf generated by the pointed sheaf $\mathbb{G}_m=(\mathbb{G}_m,1)$, denoted $L_{nis}\tilde{F}_{grp}\mathbb{G}_m$, being a retract of $L_{nis}F_{grp}(\mathbb{G}_m)$, is again $\mathbb{A}^1$-invariant. However, we claim that $L_{nis}\tilde{F}_{grp}\mathbb{G}_m$ is not strongly $\mathbb{A}^1$-invariant. Indeed, we know that $L_{nis}\tilde{F}_{grp}\mathbb{G}_m=\pi_1^{nis}\Sigma_{nis}\mathbb{G}_m$, whereas the free strongly $\mathbb{A}^1$-invariant sheaf of groups generated by $\mathbb{G}_m$ is (see [\cite{morel2012a1}, the paragraph below Lemma 7.23]) $$\pi_1^{mot}\Sigma_{mot}\mathbb{G}_m\simeq\pi_1^{\mathbb{A}^1}\mathbb{P}^1=F_{\mathbb{A}^1}(1).$$ We claim that there is no isomorphism between the two. Indeed, the center of $L_{nis}\tilde{F}_{grp}\mathbb{G}_m$ is trivial (since on connected schemes such as $\mathbb{G}_m$, its sections form a free group on a set of cardinality $>1$), whereas the center of $F_{\mathbb{A}^1}(1)$ is highly nontrivial (its center contains $\underline{\mathrm{K}}_2^{MW}$ [\cite{morel2012a1}, Theorem 7.29(2)]).
\end{coex}

\begin{rem}\label[rem]{free group of bir is bir}
\Cref{main coex 2} is not a special case of \Cref{main coex 1} since $\mathbb{G}_m$ is $1$-birational [\cite{bachmann2024strongly}, Lemma 2.8]. Moreover, the analogue of \Cref{main coex 2} fails for birational sheaves of sets $A$: for any endofunctor $F:\mathrm{Sets}\to \mathrm{Sets}$, $L_\Sigma FA$ is birational and hence strongly $\mathbb{A}^1$-invariant when group-valued [\cite{0bat}, Corollary 3.2.6].
\end{rem}

The following lemma provides a complete converse to \Cref{1 bir vs strong}, obtained by generalizing the mechanism of \Cref{main coex 2}.

\begin{lem}\label[lem]{free groups are strong if and only if they are birational}
Let $k$ be a perfect field and let $\mathrm{F}$ be a Nisnevich sheaf of groups on $Sm_k$ which is generically free (\Cref{defn birational and generically free}). Then $\mathrm{F}$ is strongly $\mathbb{A}^1$-invariant if and only if it is birational local.
\end{lem}
\begin{proof}
Since birational sheaves of groups are automatically strongly $\mathbb{A}^1$-invariant [\cite{0bat}, Corollary 3.2.6], the ``if'' part is immediate.

For the ``only if'' part, suppose that $\mathrm{F}$ is strongly $\mathbb{A}^1$-invariant. Because such an $\mathrm{F}$ is Zariski-weakly unramified [\cite{bachmann2024strongly}, Corollary 2.4], for any connected $X \in Sm_k$ with generic point $\eta_X = \mathrm{Spec}(k(X))$, the restriction $\mathrm{F}(X) \hookrightarrow \mathrm{F}(\eta_X)$ is injective. Since $\mathrm{F}(\eta_X)$ is a free group by hypothesis, the Nielsen-Schreier theorem implies that $\mathrm{F}(X)$ is itself a free group for every connected $X \in Sm_k$.

To show that $\mathrm{F}$ is birational local, by [\cite{0bat}, Lemma 3.3.8] it suffices to show that $\mathrm{F}_{-1}=1$, i.e., that for every smooth connected $U/k$, any pointed map of sheaves $f \colon (\mathbb{G}_m, 1) \to (\mathrm{F}^U, 1)$ (where $\mathrm{F}^U := \mathrm{F}(- \times U)$) is trivial (any unpointed map can be translated to a pointed map using the group structure on the codomain). Since $\mathrm{B}_{nis}(\mathrm{F}^U) \simeq (\mathrm{B}_{nis}\mathrm{F})^U$, the $\mathbb{A}^1$-invariance of $\mathrm{B}_{nis}\mathrm{F}$ implies that $\mathrm{F}^U$ is strongly $\mathbb{A}^1$-invariant. Thus $f$ factors uniquely through a homomorphism of sheaves of groups $\bar{f} \colon F_{\mathbb{A}^1}(1) \to \mathrm{F}^U$. Evaluating on the connected scheme $Y := \mathbb{G}_m \times \mathbb{G}_m$, since $F_{\mathbb{A}^1}(1)$ is nilpotent of class $2$ [\cite{morel2012a1}, Theorem 7.29(2)] and $\mathrm{F}^U(Y) = \mathrm{F}(Y \times U)$ is a free group, the image subgroup $\mathrm{im}(\bar{f}_Y) \leq \mathrm{F}^U(Y)$ is a nilpotent free group and hence of rank at most $1$ (i.e., isomorphic to $\mathbb{Z}$ or $1$). Let $w := f(\mathrm{id}_{\mathbb{G}_m}) \in \mathrm{F}^U(\mathbb{G}_m)$. Because $\mathrm{rank}(\mathrm{im}(\bar{f}_Y)) \leq 1$, the pullbacks $\mathrm{pr}_1^* w, \mathrm{pr}_2^* w \in \mathrm{im}(\bar{f}_Y)$ (along the projections $\mathrm{pr}_i \colon \mathbb{G}_m \times \mathbb{G}_m \to \mathbb{G}_m$) must satisfy $(\mathrm{pr}_1^* w)^a = (\mathrm{pr}_2^* w)^b$ for some integers $a, b$ not both zero (say $b \neq 0$ by symmetry). Pulling this relation back along $(1, \mathrm{id}_{\mathbb{G}_m}) \colon \mathbb{G}_m \to \mathbb{G}_m \times \mathbb{G}_m$ yields $1 = w^b$ in the free (hence torsion-free) group $\mathrm{F}^U(\mathbb{G}_m)$. Thus $w = 1$, so $f$ is trivial.
\end{proof}

\begin{lem}\label[lem]{free is bir if and only if from bir}
Let $A$ be a Nisnevich sheaf of sets on $Sm_k$ and let $F \colon \mathrm{Sets} \to \mathrm{Sets}$ be an endofunctor equipped with an injective natural transformation $\eta \colon \mathrm{id} \hookrightarrow F$ that is cartesian on injections. Then $L_{nis}FA$ is birational if and only if $A$ is birational, in which case $L_\Sigma FA \simeq L_{nis}FA$.
\end{lem}
\begin{proof}
If $A$ is birational, then $L_\Sigma FA$ is clearly birational (\Cref{free group of bir is bir}), and hence $L_\Sigma FA$ is automatically Nisnevich-local [\cite{0bat}, Theorem 2.1.3]. Thus $L_\Sigma FA \simeq L_{nis}FA$ is birational.

Conversely, assume that $L_{nis}FA$ is birational. We must show that for every connected $X \in Sm_k$ with function field $k(X)$, the restriction $A(X) \to A(k(X))$ is a bijection. Since $\eta \colon A \hookrightarrow FA$ is a monomorphism of presheaves and $L_{nis}$ is left-exact, $\eta \colon A \hookrightarrow L_{nis}FA$ is a monomorphism of Nisnevich sheaves. In the commutative diagram
\[
    \xymatrix{
    A(X) \ar[r] \ar@{^(->}[d]_-{\eta_X} & A(k(X)) \ar@{^(->}[d]^-{\eta_{k(X)}} \\
    (L_{nis}FA)(X) \ar[r]^-{\simeq} & (L_{nis}FA)(k(X))
    }
\]
the left vertical arrow is injective and the bottom arrow is a bijection, so the top horizontal map $A(X) \to A(k(X))$ is injective for every connected $X \in Sm_k$.

To see that $A(X) \to A(k(X))$ is surjective, choose a section $s \in A(k(X))$ and let $\tilde{s} \in (L_{nis}FA)(X)$ be the unique section lifting $\eta_{k(X)}(s)$. Choose a Nisnevich cover $\{U_i \to X\}$ by connected schemes such that each $\tilde{s}|_{U_i}$ comes from an element $x_i \in F(A(U_i))$ satisfying the descent condition. Consider the commutative diagram
\[
    \xymatrix{
    A(U_i) \ar@{^(->}[r] \ar@{^(->}[d] & A(k(U_i)) \ar@{^(->}[d] & A(k(X)) \ni s \ar[l] \ar@{^(->}[d] \\
    x_i \in F(A(U_i)) \ar[r] & F(A(k(U_i))) & F(A(k(X))) \ar[l]
    }
\]
Since $A(U_i) \hookrightarrow A(k(U_i))$ is injective and $\eta \colon \mathrm{id} \hookrightarrow F$ is cartesian on injections, the left-hand square is a pullback. Because $x_i$ and $s$ have the same image $\tilde{s}|_{k(U_i)}$ in $F(A(k(U_i)))$, the pullback property of the left-hand square implies $x_i \in A(U_i)$. Because the elements $x_i = \tilde{s}|_{U_i}$ satisfy the descent condition in the subsheaf $A \subset L_{nis}FA$ and $A$ is a Nisnevich sheaf, they glue to a global section in $A(X)$ lifting $s$.
\end{proof}

We have thus obtained the following characterization:
\begin{prop}\label[prop]{free is strong vs birational}
Let $k$ be a perfect field. Let $A$ be a Nisnevich sheaf of sets on $Sm_k$ and let $F \colon \mathrm{Sets} \to \mathrm{Grp}$ be a functor whose underlying endofunctor on $\mathrm{Sets}$ (also denoted $F$) is equipped with an injective natural transformation $\eta \colon \mathrm{id} \hookrightarrow F$ that is cartesian on injections. Assume moreover that $L_{nis}FA$ is a generically free group presheaf. Then the following conditions are equivalent:
\begin{enumerate}
    \item $L_{nis}FA$ is strongly $\mathbb{A}^1$-invariant.
    \item $L_{nis}FA$ is strongly $\mathbb{A}^1$-invariant and $(L_{nis}FA)_{-1}=1$.
    \item $L_{nis}FA$ is birational.
    \item $A$ is a birational sheaf.
\end{enumerate}
When $A$ is a sheaf of groups, condition (4) is further equivalent to:
\begin{enumerate}
    \setcounter{enumi}{4}
    \item $A$ is strongly $\mathbb{A}^1$-invariant and $A_{-1}=1$.
\end{enumerate}
Finally, if $A$ is \'etale-weakly unramified and $F$ preserves intersection squares, then $L_{nis}FA$ is $\mathbb{A}^1$-invariant if and only if $A$ is $\mathbb{A}^1$-invariant.
\end{prop}
\begin{proof}
The equivalence (1) $\iff$ (3) follows from \Cref{free groups are strong if and only if they are birational}. The equivalences (2) $\iff$ (3) and (4) $\iff$ (5) follow from [\cite{0bat}, Lemma 3.3.9], while (3) $\iff$ (4) is \Cref{free is bir if and only if from bir}.

For the final statement, when $A$ is \'etale-weakly unramified and $F$ preserves intersection squares, \Cref{free nis_excision_G} gives $L_{nis}FA \simeq L_\Sigma FA$, which evaluates on any connected $X \in Sm_k$ as $F(A(X))$. Since $\mathbb{A}^1$-invariance is tested on connected schemes and $\eta \colon \mathrm{id} \hookrightarrow F$ is cartesian on injections (hence forces $F$ to reflect bijections between the injective maps $A(X) \hookrightarrow A(\mathbb{A}^1_X)$), $A(X) \to A(\mathbb{A}^1_X)$ is a bijection if and only if $F(A(X)) \to F(A(\mathbb{A}^1_X))$ is a bijection.
\end{proof}

\begin{rem}\label[rem]{pointed free is strong vs birational}
The exact same results (\Cref{free is bir if and only if from bir} and \Cref{free is strong vs birational}) hold in the pointed setting for a Nisnevich sheaf of pointed sets $A$ and an endofunctor $F \colon \mathrm{Sets}_* \to \mathrm{Sets}_*$ equipped with an injective natural transformation $\eta \colon \mathrm{id} \hookrightarrow F$. In fact, in the pointed setting one does not need to assume that $\eta$ is cartesian on injections: since every injection of pointed sets splits, the naturality square of $\eta$ along any injection in $\mathrm{Sets}_*$ is automatically a pullback square.
\end{rem}

\begin{cor}\label[cor]{free grp sheaf strong vs birational}
Let $k$ be a perfect field and let $A$ be a Nisnevich sheaf of sets (resp.\ pointed sets) on $Sm_k$. Then the following are equivalent:
\begin{enumerate}
    \item $L_{nis}F_{grp}A$ (resp.\ $L_{nis}\tilde{F}_{grp}A$) is strongly $\mathbb{A}^1$-invariant.
    \item $L_{nis}F_{grp}A$ (resp.\ $L_{nis}\tilde{F}_{grp}A$) is birational.
    \item $A$ is birational.
\end{enumerate}
Moreover, if $A$ is \'etale-weakly unramified, then $L_{nis}F_{grp}A$ (resp.\ $L_{nis}\tilde{F}_{grp}A$) is $\mathbb{A}^1$-invariant if and only if $A$ is $\mathbb{A}^1$-invariant.
\end{cor}
\begin{proof}
Apply \Cref{free is strong vs birational} (resp.\ \Cref{pointed free is strong vs birational}) to $F = F_{grp}$ (resp.\ $F = \tilde{F}_{grp}$), using \Cref{adding free structure preserve intersection square}.
\end{proof}

\begin{rem}
This marks a stark difference between the motivic and classical worlds: in classical topology, every group is the fundamental group of some space. In particular, every free group generated by a discrete object of $\mathrm{Spc}$ (i.e., a set) is always a fundamental group; specifically, it is the fundamental group of a wedge sum of circles. In contrast, within the motivic realm, free group sheaves generated by discrete objects of $\mathcal{H}^{mot}(k)$ cannot be realized as motivic fundamental groups unless they are birational.
\end{rem}

\begin{coex}\label[coex]{milnor coex}
Let $k$ be a perfect field and let $n\geq 1$ be a natural number. Since $\mathbb{Q}$ is flat over $\mathbb{Z}$, the functor $-\otimes _\mathbb{Z}\mathbb{Q}$ commutes with $\mathbb{G}_m$-contraction. Hence,
\[
    (\underline{\mathrm{K}}^M_n\otimes\mathbb{Q})_{-1}\simeq (\underline{\mathrm{K}}^M_n)_{-1}\otimes\mathbb{Q}\simeq \underline{\mathrm{K}}^M_{n-1}\otimes\mathbb{Q}\neq 0 \quad \text{as } n\geq 1.
\]
Because $\underline{\mathrm{K}}^M_n\otimes \mathbb{Q}$ is fully weakly unramified by \Cref{milnor is fully weakly un}, by \Cref{free is strong vs birational} the sheaf $L_{nis}F_{grp}(\underline{\mathrm{K}}^M_n\otimes \mathbb{Q})$ (as well as $L_{nis}\tilde{F}_{grp}(\underline{\mathrm{K}}^M_n\otimes \mathbb{Q})$) is $\mathbb{A}^1$-invariant but not strongly $\mathbb{A}^1$-invariant.
\end{coex}

\begin{rem}\label[rem]{free mon gm not strong}
For any Nisnevich sheaf of monoids $M$, its discrete group completion is given by $$M^{gp} \simeq \pi_0^{nis}\Omega\mathrm{B}_{nis}M \simeq \pi_1^{nis}\mathrm{B}_{nis}M.$$ Over a perfect field, if $M$ is strongly $\mathbb{A}^1$-invariant (in the sense that $\mathrm{B}_{nis}M$ is $\mathbb{A}^1$-local), then $M^{gp} \simeq \pi_1^{nis}\mathrm{B}_{nis}M$ is a strongly $\mathbb{A}^1$-invariant sheaf of groups by [\cite{morel2012a1}, Corollary 6.3]. Conversely, when $M = L_{nis}F_{mon}(G)$ (or $L_{nis}\tilde{F}_{mon}(G)$) for a sheaf of sets $G$, we have $M^{gp} \simeq L_{nis}F_{grp}G$ (as $\mathrm{B}F_{mon}(E) \simeq \mathrm{B}F_{grp}(E)$ for every set $E$, and hence $\mathrm{B}_{nis}(M) \simeq \mathrm{B}_{nis}(L_{nis}F_{grp}G)$). Consequently, the counterexamples above extend directly to free monoid sheaves: neither $L_{nis}\tilde{F}_{mon}(\mathbb{G}_m)$ nor $L_{nis}F_{mon}(G)$ (for an \'etale-weakly unramified, non-birational motivic sheaf of sets $G$) is strongly $\mathbb{A}^1$-invariant, despite both being $\mathbb{A}^1$-invariant by \Cref{free mot_excision_G}.
\end{rem}

\begin{lem}\label[lem]{free ab strong vs bir}
Let $k$ be a perfect field and let $A$ be a Nisnevich sheaf of sets (resp.\ pointed sets) on $Sm_k$. Then $M := L_{nis}\mathbb{Z}[A]$ (resp.\ $L_{nis}\tilde{\mathbb{Z}}[A]$) is strongly $\mathbb{A}^1$-invariant if and only if it is birational local (equivalently, $A$ is birational local).
\end{lem}
\begin{proof}
For simplicity, we prove only the pointed case $M = L_{nis}\tilde{\mathbb{Z}}[A]$. By \Cref{free group of bir is bir}, if $A$ is birational, then $M$ is birational. Once again, by [\cite{0bat}, Corollary 3.2.6], if $M$ is birational then it is strongly $\mathbb{A}^1$-invariant. It remains to prove that if $M$ is strongly $\mathbb{A}^1$-invariant, then $A$ is birational.

So suppose $M$ is strongly $\mathbb{A}^1$-invariant. As in the proof of \Cref{free groups are strong if and only if they are birational}, it suffices to show that for every connected $U \in Sm_k$, any pointed map $f \colon (\mathbb{G}_m, 1) \to A^U$ is trivial. Since $M$ is strictly $\mathbb{A}^1$-invariant, the induced map $f \colon (\mathbb{G}_m, 1) \to M^U$ factors through a homomorphism $\bar{f} \colon \underline{\mathrm{K}}_1^{MW} \to M^U$ [\cite{morel2012a1}, Theorem 3.37]. In $\underline{\mathrm{K}}_1^{MW}(k(t))$, the Steinberg relation $[t][1-t]=0$ yields $[t] + [1-t] = [t(1-t)]$. Evaluating $\bar{f}$ on $k(t)$ and restricting to the function field $K := k(U)(t)$ of $\mathbb{G}_m \times U$, where $M(K) = \tilde{\mathbb{Z}}[A(K)]$, gives
\[
    [f(t)_K] + [f(1-t)_K] = [f(t(1-t))_K] \quad \text{in } \tilde{\mathbb{Z}}[A(K)].
\]
By the linear independence of the non-basepoint elements of $A(K)$ in $\tilde{\mathbb{Z}}[A(K)]$, either $f(t)_K$ or $f(1-t)_K$ is the basepoint $*$. Since $t \mapsto 1-t$ induces a $k(U)$-automorphism of $K$, both cases imply $f(t)_K = *$. Because $M$ is weakly unramified [\cite{bachmann2024strongly}, Corollary 2.4], the composite $A^U(\mathbb{G}_m) \hookrightarrow M(\mathbb{G}_m \times U) \hookrightarrow M(K)$ is injective, and hence $f = *$.
\end{proof}

\begin{coex}[Abelian non-strong/strict motivic group sheaves]\label[coex]{main coex abelian}\mbox{}
\begin{enumerate}
    \item Since $\mathbb{G}_m$ is \'etale-weakly unramified, the reduced free abelian group sheaf $\tilde{\mathbb{Z}}_{nis}(\mathbb{G}_m) := L_{nis}\tilde{\mathbb{Z}}[\mathbb{G}_m]$ is $\mathbb{A}^1$-invariant (see \Cref{adding free structure to G_m}), but by the above lemma it is not strongly $\mathbb{A}^1$-invariant since $\mathbb{G}_m$ is not birational. (This is the standard example widely used in the literature, due to [\cite{choudhury2014connectivity}, Lemma 5.6].)
    \item By our lemma above, $L_{nis}\mathbb{Z}[\mathbb{G}_m]$ and $L_{nis}\mathbb{Z}[\underline{\mathrm{K}}^M_n \otimes \mathbb{Q}]$ (for $n\geq 1$) are also examples of $\mathbb{A}^1$-invariant abelian sheaves that are not strongly $\mathbb{A}^1$-invariant.
\end{enumerate}
\end{coex}

This list of examples brings us to one of the main observations of this paper:
\begin{cor}[{Cf.\ [\cite{spitzweck2012motivic}, Remark 3.5]}]\label{mot not topos for perf}
Over any perfect field $k$, the motivic homotopy category $\mathcal{H}^{mot}(k)$ is not an $\infty$-topos.
\end{cor}
\begin{proof}
Using the counterexamples above (e.g., \Cref{main coex 2}, or \Cref{main coex abelian} for any $n\geq 1$), showing that the inclusion $\mathrm{Grp}^{\mathbb{A}^1}_k\subset L_{\mathbb{A}^1}\mathrm{Grp}_k$ is strict, the result follows directly from \Cref{not a topos localization} (or, say, combining \Cref{motivic em} and \Cref{em in infinity topos}).
\end{proof}
Using essentially smooth base change, we may now extend this to all schemes:
\begin{thm}\label[thm]{mot not topos}
Let $S$ be a qcqs scheme. Then the motivic homotopy category $\mathcal{H}^{mot}(S)$ is an $\infty$-topos if and only if $S = \emptyset$.
\end{thm}
\begin{proof}
If $S = \emptyset$, then $Sm_\emptyset = \{\emptyset\}$ and $\mathcal{H}^{mot}(\emptyset) = \mathcal{P}_{nis}(\emptyset) \simeq *$, which is an $\infty$-topos. Conversely, suppose $S \neq \emptyset$ and assume for contradiction that $\mathcal{H}^{mot}(S)$ is an $\infty$-topos. By nil-invariance of $\mathcal{H}^{mot}$, $\mathcal{H}^{mot}(S)\simeq \mathcal{H}^{mot}(S^{red})$ [\cite{sfbat}, Theorem 4.2.2], so we may assume that $S$ is reduced.

\emph{Step 1: Reduction from $S \neq \emptyset$ to a residue field.} Choose a generic point $\eta = \operatorname{Spec}(k) \in S$ of an affine open subscheme, so that $\eta = \lim_\alpha U_\alpha$ is the cofiltered limit of its affine open neighborhoods $j_\alpha \colon U_\alpha \hookrightarrow S$. For each open immersion $j_\alpha \colon U_\alpha \hookrightarrow S$ (which are quasi-compact because $S$ is quasi-separated) (and each transition open immersion $j_{\alpha\beta} \colon U_\alpha \hookrightarrow U_\beta$), the pullback $j_\alpha^* \colon \mathcal{H}^{mot}(S) \to \mathcal{H}^{mot}(U_\alpha)$ admits both a left adjoint $(j_\alpha)_\shrp$ and a fully faithful right adjoint $(j_\alpha)_*$ [\cite{hoyois2017six}, \S 4], exhibiting $\mathcal{H}^{mot}(U_\alpha)$ as a left-exact localization of $\mathcal{H}^{mot}(S)$ and hence an $\infty$-topos ([\cite{lurie2009higher}, Definition 6.1.0.4]). Similarly, each $j_{\alpha\beta}^* \colon \mathcal{H}^{mot}(U_\beta) \to \mathcal{H}^{mot}(U_\alpha)$ is left-exact and colimit-preserving, so $j_{\alpha\beta}^* \dashv (j_{\alpha\beta})_*$ is a geometric morphism of $\infty$-toposes. Because $j_{\alpha\beta}$ are affine (being a morphism of affine schemes), by [\cite{hoyois2015quadratic}, Proposition C.7 (3)],
\[
\mathcal{H}^{mot}(k) \simeq \lim_{\alpha, (j_{\alpha\beta})_*} \mathcal{H}^{mot}(U_\alpha) \quad \text{in} \quad \widehat{Cat}_\infty.
\]
Since cofiltered limits of $\infty$-toposes along right adjoints are preserved by the forgetful functor $\mathcal{R}\mathcal{T}\mathrm{op} \to \widehat{Cat}_\infty$ [\cite{lurie2009higher}, Theorem 6.3.3.1], $\mathcal{H}^{mot}(k)$ is an $\infty$-topos.

\emph{Step 2: Reduction from an arbitrary field $k$ to a perfect field.} Let $k$ be a field such that $\mathcal{H}^{mot}(k)$ is an $\infty$-topos. If $\operatorname{char}(k) = 0$, then $k = k^{\mathrm{perf}}$ is already perfect; assume therefore that $\operatorname{char}(k) = p > 0$. The perfection $\operatorname{Spec}(k^{\mathrm{perf}}) = \lim_{m, \Phi_p} \operatorname{Spec}(k)$ is the cofiltered limit along the absolute Frobenius morphism $\Phi_p \colon \operatorname{Spec}(k) \to \operatorname{Spec}(k)$, which induces an equivalence
\[
\mathcal{H}^{mot}(k^{\mathrm{perf}}) \simeq \lim_{m, (\Phi_p)_*} \mathcal{H}^{mot}(k) \quad \text{in} \quad \widehat{Cat}_\infty
\]
by [\cite{hoyois2015quadratic}, Proposition C.7(3)] again. Because $\Phi_p$ is a universal homeomorphism ([\cite[\href{https://stacks.math.columbia.edu/tag/04DF}{Lemma 04DF}]{stacks-project}]), \Cref{uni homeo push preserves limits and colimits} below shows that $(\Phi_p)_* \colon \mathcal{H}^{mot}(k) \to \mathcal{H}^{mot}(k)$ is a left-exact left adjoint, so the diagram above lifts to a diagram in $\mathcal{LT}\mathrm{op}$. Since the forgetful functor $\mathcal{LT}\mathrm{op} \to \widehat{Cat}_\infty$ preserves small limits ([\cite{lurie2009higher}, Proposition 6.3.2.3]), the limit $\mathcal{H}^{mot}(k^{\mathrm{perf}})$ computed in $\widehat{Cat}_\infty$ would also be an $\infty$-topos, contradicting \Cref{mot not topos for perf} since $k^{\mathrm{perf}}$ is perfect.
\end{proof}
\begin{lem}\label[lem]{uni homeo push preserves limits and colimits}
    Let $f: X \to Y$ be a universal homeomorphism of qcqs schemes. Then the functor $f_*: \mathcal{P}(X) \to \mathcal{P}(Y)$ commutes with $L_{{mot}}$. In particular, its restrictions $f_*: \mathcal{H}^{{mot}}(X) \to \mathcal{H}^{{mot}}(Y)$ preserve all limits and colimits (and hence admit both left and right adjoints; in particular, $f_*$ is a left exact left adjoint).
\end{lem}
\begin{proof}
    Since universal homeomorphisms are integral, by [\cite[Proposition 2.11, Corollary 2.12]{bachmann2017norms}] it suffices to show that $f_*:\mathcal{P}(X)\to \mathcal{P}(Y)$ sends $L_\Sigma$-equivalences to Nisnevich equivalences. As $f_*$ preserves colimits and $L_\Sigma$-equivalences are generated by finite coproduct decompositions, this reduces to the case of the canonical maps
$\coprod_i h_X{U_i}\longrightarrow h_X(\coprod_i U_i).$ Now, Nisnevich equivalences are detected on henselian local schemes. For any such scheme $T$ over $Y$, the base change $T\times_Y X\to T$ is a universal homeomorphism. In particular, $T\times_Y X$ is nonempty and connected, so every morphism $T\times_Y X\to U$ factors uniquely through one of the $U_i$. Thus, the displayed map becomes an equivalence after applying $f_*$ and Nisnevich localization, as desired.
\end{proof}
\begin{rem}
The failure of the motivic homotopy category to be a topos was first brought to attention in [\cite{spitzweck2012motivic}, Remark 3.5]. However, their argument did not identify all the internal Eilenberg-MacLane objects of the motivic homotopy category, leaving it somewhat incomplete. This was addressed in model-categorical language in [\cite{MR3890956}, \S4.3], where the authors gave a complete proof of the model-categorical claim that the ``motivic homotopy category is not a model topos'' over a perfect field (see Proposition 4.12 in \textit{loc.\ cit.}). We were unaware of the cited article while writing the first draft of this paper. While over a perfect field the statements are ``equivalent'' in a suitable sense and, in the end, require an example of an $\mathbb{A}^1$-invariant sheaf of groups that is not strongly/strictly $\mathbb{A}^1$-invariant, the two differ significantly in their presentation. In addition, instead of relying on the standard (abelian) example of an $\mathbb{A}^1$-invariant sheaf of abelian groups that is not strongly $\mathbb{A}^1$-invariant, we have provided several non-abelian examples, and we have further extended the result from perfect fields to arbitrary non-empty qcqs base schemes (\Cref{mot not topos}).
\end{rem}

Having established this, let us explain how close the motivic homotopy category is to being an $\infty$-topos, and thus identify which of the defining Giraud axioms fails. To this end, we first recall Giraud's theorem on the characterization of $\infty$-toposes from [\cite{lurie2009higher}, Proposition 6.1.5.3]:

\begin{thm}[$\infty$-Categorical Giraud Theorem]\label{giraud's theorem}
An $\infty$-category $\mathcal{C}$ is an $\infty$-topos if and only if it satisfies the following conditions:
\begin{enumerate}
    \item $\mathcal{C}$ is presentable.
    \item Colimits in $\mathcal{C}$ are universal.
    \item Coproducts in $\mathcal{C}$ are disjoint.
    \item Every groupoid object in $\mathcal{C}$ is effective.
\end{enumerate}
\end{thm}

We have:
\begin{thm}\label{motivic giraud's theorem}
Over an arbitrary base scheme $S$, the motivic homotopy category $\mathcal{H}^{mot}(S)$ satisfies conditions (1) and (2) of \Cref{giraud's theorem}. If $S=\mathrm{Spec}(k)$ is the spectrum of a field, then the motivic homotopy category $\mathcal{H}^{mot}(k)$ also has disjoint coproducts, i.e., it satisfies (3).
\end{thm}
\begin{proof}
Being an accessible localization at a small set of morphisms, the presentability of $\mathcal{H}^{mot}(S)$ is standard. Universality of colimits is well known and follows essentially from the fact that $\mathbb{A}^1_S\to S$ admits a section; see [\cite{hoyois2017six}, Proposition 3.15].

For the last statement, the point is that when $k$ is a field, the Nisnevich-local coproduct of two Nisnevich sheaves $\esc{X}$ and $\esc{Y}$ is given by $L_\Sigma(\esc{X}\sqcup _{pre}\esc{Y})$. Again, since the points of $Sm_k$ for the $\Sigma$-topology are connected schemes, by \Cref{nis exc on conn} we are reduced to checking the Nisnevich excision condition on $Sm^{conn}_k$, where $L_\Sigma(\esc{X}\sqcup_{pre}\esc{Y})$ agrees with $\esc{X}\sqcup_{pre}\esc{Y}$ and satisfies excision because in $\mathrm{Spc}$ pullback squares are stable under coproducts.

Since $L_{\mathbb{A}^1}\mathcal{P}(k)\subset \mathcal{P}(k)$ is closed under coproducts, it follows from \Cref{L_sigma preserves a1 local} that $\esc{X}\sqcup_{\Sigma}\esc{Y} := L_\Sigma(\esc{X}\sqcup _{pre}\esc{Y})$ is $\mathbb{A}^1$-local whenever $\esc{X}$ and $\esc{Y}$ are. Thus, $\esc{X}\sqcup _{\Sigma}\esc{Y}$ is the motivic coproduct of two motivic objects $\esc{X},\esc{Y}$. The disjoint coproduct condition for $\mathcal{H}^{mot}(k)$ therefore follows directly from the same condition in the $\infty$-topos $\mathcal{P}_\Sigma(k)$.
\end{proof}

\begin{rem}
Since the inclusion $\mathcal{H}^{mot}(k)\subset \mathcal{P}_{nis}(k)$ is closed under filtered colimits, the above proof in fact shows that the inclusion is closed under all small coproducts. The only missing piece in this colimit puzzle is therefore geometric realizations (and pushouts). In fact, the counterexamples of \Cref{main coex 2} highlight precisely this failure: the colimit of the simplicial diagram $\mathrm{B}_\bullet (L_{nis}\tilde{F}_{grp}\mathbb{G}_m)$ in the motivic homotopy category $\mathcal{H}^{mot}(k)$ (which is $\mathrm{B}_{mot}(L_{nis}\tilde{F}_{grp}\mathbb{G}_m)$) does not coincide with its colimit in $\mathcal{P}_{nis}(k)$, so the groupoid object $\mathrm{B}_\bullet(L_{nis}\tilde{F}_{grp}\mathbb{G}_m)$ fails to be effective in $\mathcal{H}^{mot}(k)$. In [\cite{0bat}, Corollary 3.1.2], we show that, in fact, the birational motivic homotopy category repairs this gap by preserving geometric realizations, thus yielding an $\infty$-topos [\cite{0bat}, Theorem 3.1.13].
\end{rem}
\subsection{Motivically connective objects and grouplike objects}
In the previous subsection, we saw that the motivic homotopy category fails to internally deloop arbitrary discrete motivic groups. This subsection is devoted to understanding the condition required for the internal delooping of general motivic grouplike monoids, thereby establishing a motivic version of the loop-deloop equivalence between the $\infty$-categories of internally connected motivic spaces and of strongly $\mathbb{A}^1$-invariant grouplike motivic monoids.

The failure of the functor $\pi_1^{mot}:\mathcal{EM}_1(\mathcal{H}^{mot}(k))\to L_{\mathbb{A}^1}\mathrm{Grp}_k$ to be an equivalence, as explained in \Cref{mot not topos}, is a manifestation of the deeper fact that the loop functor $$\Omega: \mathcal{H}_*^{mot}(S)_{\geq 1}\to  \mathrm{Grp}(\mathcal{H}^{mot}(S))$$ is not an equivalence of $\infty$-categories. The primary issue is that, as shown by \Cref{main coex 2}, this functor is not essentially surjective. Over perfect fields, any grouplike monoid in the essential image must have a strongly $\mathbb{A}^1$-invariant $\pi_0^{mot}$, since this sheaf is isomorphic to $\pi_1^{\mathbb{A}^1}$ of the delooping [\cite{morel2012a1}, Corollary 6.3]. Remarkably, this condition on $\pi_0^{mot}$ is also sufficient to ensure a motivic delooping.

\begin{defn}\label[defn]{defn strongly A1 inv monoid}
Let $\mathrm{Mon}(\mathcal{H}^{mot}(S))$ and $\mathrm{Grp}(\mathcal{H}^{mot}(S))$ denote the $\infty$-categories of monoid objects and grouplike monoid objects in $\mathcal{H}^{mot}(S)$, and let $$\mathrm{B}_{mot}:=\mathrm{B}_{\mathcal{H}^{mot}(S)}:\mathrm{Mon}(\mathcal{H}^{mot}(S))\to \mathcal{H}^{mot}(S)_*$$ denote the motivic bar construction, left adjoint to $\Omega$ (\Cref{bar loop adjunction}). Following \S2.2, a motivic monoid $\esc{M}$ is \emph{strongly $L_{mot}$-local} if the canonical map $\mathrm{B}_{nis}\esc{M}\to \mathrm{B}_{mot}\esc{M}$ is an equivalence, i.e., $\mathrm{B}_{nis}\esc{M}\in \mathcal{H}^{mot}(S)$. We denote the full subcategories of strongly $L_{mot}$-local motivic monoids and grouplike motivic monoids by $\mathrm{Mon}_{L_{mot}}(\mathcal{H}^{mot}(S))$ and $\mathrm{Grp}_{L_{mot}}(\mathcal{H}^{mot}(S))$, respectively, and denote by $\mathcal{G}rp_{mot}(\mathcal{H}^{mot}(k))\subset \mathrm{Grp}(\mathcal{H}^{mot}(k))$ the full subcategory of grouplike motivic monoids $\esc{G}$ such that $\pi_0^{mot}\esc{G}\in \mathrm{Grp}_k^{\mathbb{A}^1}$.
\end{defn}

\begin{cor}\label[cor]{strong motivic groups and motivic connected spaces}
        Over a perfect field $k$, there is an equivalence of $\infty$-categories:
        $$\mathrm{B}_{\mathcal{H}^{mot}(k)}: \mathrm{Grp}_{L_{mot}}(\mathcal{H}^{mot}(k))\simeq \mathcal{H}_*^{mot}(k)_{\geq 1}:\Omega.$$
\end{cor}
\begin{proof}
         Due to \Cref{motivic 0 truncation is injective}, the categorical theorem \Cref{loop bar for localization} applies directly to the motivic localization $L_{mot}$ of the Nisnevich $\infty$-topos $\mathcal{P}_{nis}(k)$.
\end{proof}
    
    We now relate the categorical notion of strong $L_{mot}$-locality from \Cref{defn strongly A1 inv monoid} to its concrete characterization over perfect fields:
\begin{prop}\label[prop]{strong mot is strong a1}
    A motivic monoid space $\esc{M}$ is strongly $L_{mot}$-local if and only if it is strongly $\mathbb{A}^1$-invariant in the sense of \textup{[\cite{elmanto2021motivic}, Definition 3.1.6(1)]}. If $k$ is a perfect field, then a grouplike motivic monoid $\esc{M}$ is strongly $L_{mot}$-local if and only if $\pi_0^{nis}\esc{M}$ is strongly $\mathbb{A}^1$-invariant.
\end{prop} 
\begin{proof}
    The first statement follows directly from comparing the definitions of the two terms. When $k$ is a perfect field, the second statement is a restatement of [\cite{elmanto2021motivic}, Proposition 3.1.12(1)].
\end{proof}
\begin{lem}\label[lem]{pi0 of monoid}\label[lem]{strongly Lmot if and only if pi0 strongly Lmot}
    Let $\esc{G}$ be a grouplike motivic monoid space over a perfect field $k$. Then $\pi_0^{mot}\esc{G}\simeq \pi_0^{nis}\esc{G}$, and consequently $\esc{G}$ is strongly $L_{mot}$-local if and only if $\pi_0^{mot}\esc{G}$ is strongly $\mathbb{A}^1$-invariant.
\end{lem} 
\begin{proof}
    Since grouplike motivic monoid spaces are, in particular, motivic $H$-groups, the first claim is a special case of \Cref{pi0 of h space}. Then the second follows immediately from \Cref{strong mot is strong a1}.
\end{proof}

\begin{cor}\label[cor]{mot bar omega}
Over a perfect field $k$, there is an equivalence of $\infty$-categories $$ \mathrm{B}_{mot}\simeq \mathrm{B}_{nis}:  \mathcal{G}rp_{mot}(\mathcal{H}^{mot}(k))\simeq\mathcal{H}_*^{mot}(k)_{\geq 1}:\Omega,$$ where $\mathcal{G}rp_{mot}(\mathcal{H}^{mot}(k))$ denotes the full subcategory of $\mathrm{Grp}(\mathcal{H}^{mot}(k))$ consisting of grouplike motivic monoids with strongly $L_{mot}$-local $\pi_0^{mot}.$ 
\end{cor}  
\begin{proof}
Combine \Cref{strong motivic groups and motivic connected spaces} and \Cref{strongly Lmot if and only if pi0 strongly Lmot}.
\end{proof}

\subsection{Motivic theory of coverings}

\subsubsection{Spatial \texorpdfstring{$n$}{n}-coverings (motivic-\texorpdfstring{$n$}{n}-truncated morphisms and \texorpdfstring{$\mathbb{A}^1$}{A1}-\texorpdfstring{$n$}{n}-coverings)}

We first specialize the spatial covering theory of \S2.4.1 to the motivic localization $L_{mot} \colon \mathcal{P}_{nis}(S) \to \mathcal{H}^{mot}(S)$.

\begin{defn}\label[defn]{defn motivic spatial cov}
Let $S$ be a qcqs scheme and let $n \geq -2$. Specializing \Cref{defn for covering} and \Cref{definition for L covering} to $L_{mot}$,
\begin{enumerate}
    \item A morphism $f \colon \esc{Y}\to \esc{X}$ of motivic $S$-spaces is a \emph{motivic $n$-covering} (a \emph{motivic monomorphism} when $n=-1$, and simply a \emph{motivic covering} when $n=0$) if $f\in \big({\mathcal{H}^{mot}(S)_{/\esc{X}}}\big)_{\leq n}$; we denote this $\infty$-category by $Cov^n_{\mathcal{H}^{mot}}(\esc{X})$, with $Cov_{\mathcal{H}^{mot}}(\esc{X}) := Cov^0_{\mathcal{H}^{mot}}(\esc{X})$.
    \item For a Nisnevich-local $S$-space $\esc{X}$, we denote by $\mathcal{F}ib_{\esc{X}}^{\mathbb{A}^1}\subset \mathcal{P}_{nis}(S)_{/\esc{X}}$ the full subcategory of \emph{$\mathbb{A}^1$-fibrations} over $\esc{X}$, namely objects of $\mathcal{P}_{nis}(S)_{/\esc{X}}$ local with respect to motivic equivalences, and by $Cov^n_{\mathbb{A}^1}(\esc{X}):=(\mathcal{F}ib_{\esc{X}}^{\mathbb{A}^1})_{\leq n}$ the $\infty$-category of \emph{$\mathbb{A}^1$-$n$-coverings} of $\esc{X}$ (\emph{$\mathbb{A}^1$-monomorphisms} when $n=-1$, and $Cov_{\mathbb{A}^1}(\esc{X}) := Cov^0_{\mathbb{A}^1}(\esc{X})$).
    \item When $\esc{X}$ is pointed, we write $Cov^n_{\mathcal{H}^{mot}}(\esc{X})_* := Cov^n_{\mathcal{H}^{mot}_*}(\esc{X})$ and
    \[
        Cov^n_{\mathbb{A}^1}(\esc{X})_* := Cov^n_{\mathbb{A}^1}(\esc{X}) \times_{\mathcal{P}_{nis}(S)_{/\esc{X}}} \mathcal{P}_{nis}(S)_{*/\esc{X}}.
    \]
\end{enumerate}
\end{defn}

\begin{rem}
By \Cref{cov is presentable} and the remark following \Cref{definition for L covering}, for any $\esc{X}\in \mathcal{P}_{nis}(S)$ and $n \geq -2$, the categories $\mathcal{F}ib_{\esc{X}}^{\mathbb{A}^1}$ and $$Cov^n_{\mathbb{A}^1}(\esc{X}) = (\mathcal{P}_{nis}(S)_{/\esc{X}})_{\leq n} \mcap \mathcal{F}ib_{\esc{X}}^{\mathbb{A}^1}$$ (as well as $Cov^n_{\mathcal{H}^{mot}}(\esc{X})$ when $\esc{X}\in \mathcal{H}^{mot}(S)$, and all pointed counterparts) are accessible localizations of slice categories and are hence presentable (in particular, complete and cocomplete). For pointed $\esc{X} \in \mathcal{H}^{mot}(S)_*$, $Cov^n_{\mathcal{H}^{mot}}(\esc{X})_*$ therefore always admits an initial object $\tau_{\leq n}^{\mathcal{H}^{mot}_{*/\esc{X}}}(* \to \esc{X})$ (\Cref{initial object of cov in a topos}), which for $n \geq 0$ we will explicitly identify with the Postnikov fiber $\tau_{>n+1}^{nis}\esc{X} \to \esc{X}$ over perfect fields in \Cref{motivically simply connected covering space}.

Moreover, by \Cref{L fib over L local object}, $\mathcal{F}ib_{\esc{X}}^{\mathbb{A}^1} = \mathcal{H}^{mot}(S)_{/\esc{X}}$ whenever $\esc{X}\in \mathcal{H}^{mot}(S)$; in particular, $\mathcal{F}ib_{*}^{\mathbb{A}^1}\simeq \mathcal{H}^{mot}(S)$ and $Cov^n_{\mathbb{A}^1}(*)=\mathcal{H}^{mot}_{\leq n}(S)$. By \S3.1, both $\mathcal{H}^{mot}(S)$ and $\mathcal{H}^{mot}_{\leq n}(S)$ ($n \geq -2$) are closed under filtered colimits; over a field, for $n \geq 0$ they are also closed under all coproducts by \Cref{sectionwise and stalkwise checking} and the proof of \Cref{motivic giraud's theorem}(3) (since $\mathrm{Spc}_{\leq n}\subseteq \mathrm{Spc}$ is closed under coproducts for $n\geq 0$). Both properties extend to any Nisnevich-local space $\esc{X}$. Since motivic equivalences in $\mathcal{P}_{nis}(S)$ are generated under colimits by those between compact objects [\cite{MR3356340}, Theorem~4.2], the full subcategory $\mathcal{F}ib_{\esc{X}}^{\mathbb{A}^1}\subseteq \mathcal{P}_{nis}(S)_{/\esc{X}}$ of $\mathbb{A}^1$-fibrations over $\esc{X}$ is closed under filtered colimits, and hence so is $Cov^n_{\mathbb{A}^1}(\esc{X})$ ($n \geq -2$) by [\cite{lurie2009higher}, Corollary~5.5.7.4(1)]. Over a field, closure of $Cov^n_{\mathbb{A}^1}(\esc{X})$ under coproducts likewise holds for all $n \geq 0$, since in any $\infty$-category where coproduct injections are monomorphisms and coproducts distribute over pullbacks, $n$-truncated objects are closed under coproducts for all $n\geq 0$ (by induction via [\cite{lurie2009higher}, Lemma~5.5.6.15]).
\end{rem}

\begin{prop}
Let $S$ be a qcqs scheme. For any $n \geq -2$, a morphism $\esc{Y}\to \esc{X}$ of motivic $S$-spaces is a motivic $n$-covering if and only if it is an $\mathbb{A}^1$-$n$-covering, if and only if it is a Nisnevich $n$-covering.
\end{prop}
\begin{proof}
This follows immediately from \Cref{Covering of Localization} and \Cref{L covering of L local object}.
\end{proof}

\begin{lem}[{Cf.\ [\cite{morel2012a1}, Lemma 7.7(1,2)] for the $n=0$ case of $\mathbb{A}^1$-covering}]\label[lem]{motivic covering closure}
Let $S$ be a qcqs scheme and let $n \geq -2$.
\begin{enumerate}
    \item Both $\mathbb{A}^1$-fibrations and $\mathbb{A}^1$-$n$-coverings in $\mathcal{P}_{nis}(S)$ (as well as motivic $n$-coverings in $\mathcal{H}^{mot}(S)$) are stable under pullbacks and compositions.
    \item Any morphism over $\esc{X}$ between motivic $n$-coverings in $\mathcal{H}^{mot}(S)$ (resp.\ between $\mathbb{A}^1$-fibrations or $\mathbb{A}^1$-$n$-coverings in $\mathcal{P}_{nis}(S)$) is again a motivic $n$-covering (resp.\ an $\mathbb{A}^1$-fibration or $\mathbb{A}^1$-$n$-covering).
\end{enumerate}
\end{lem}
\begin{proof}
This follows immediately from \Cref{covering lemma}(1)-(3) and \Cref{lemma for L covering}.
\end{proof}

\begin{rem}
Because the motivic localization functor $L_{mot}$ is effective (see \Cref{Effective localization remark}), we obtain, by \Cref{covering pullback description}, a pullback square of $\infty$-categories for every $n \geq -2$ (writing $Cov^n_{nis}(\esc{X}) := Cov^n_{\mathcal{P}_{nis}(S)}(\esc{X})$):
\[
\xymatrix{
Cov^n_{\mathcal{H}^{mot}}(L_{mot}\esc{X})\ar@{^(->}[r]\ar[d]&\mathcal{H}^{mot}(S)_{/L_{mot}\esc{X}}\ar[d]\\
Cov^n_{nis}(\esc{X})\ar@{^(->}[r] &{\mathcal{P}_{nis}(S)}_{/\esc{X}}.
}
\]
\end{rem}

\begin{cor}
Let $k$ be a perfect field, let $n \geq -2$, and let $\esc{Z},\esc{X}$ be Nisnevich-local $k$-spaces, with $\esc{X}$ pointed and $\mathbb{A}^1$-connected. Then a morphism $f \colon \esc{Z}\to L_{mot}\esc{X}$ is an $\mathbb{A}^1$-$n$-covering (equivalently, a motivic $n$-covering) if and only if $\mathrm{fib}(f)\in \mathcal{H}^{mot}_{\leq n}(k)$ (which for $n=0$ is $Sh_{nis}^{\mathbb{A}^1}(k)$).
\end{cor}
\begin{proof}
This follows immediately from \Cref{L covering if and only if fiber is L discrete} (together with \Cref{L covering of L local object}) and the fact that $\mathcal{H}^{mot}(k)$ is closed under extensions (see \Cref{defn extension closed}) over $\mathbb{A}^1$-connected objects [\cite{morel2012a1}, Lemma 6.51].
\end{proof}

\begin{cor}[{Cf.\ [\cite{morel2012a1}, Lemma 7.5(1)]}]\label[cor]{strong torsor is a1 covering}
Let $k$ be a perfect field, $n \geq 0$, $1 \leq m \leq n+1$, and let $G \in m\mbox{-}\mathrm{Grp}^{L_{mot}}$ (i.e., $G \in \mathrm{Grp}^{\mathbb{A}^1}_k$ for $m=1$ and $G \in \mathrm{Ab}^{\mathbb{A}^1}_k$ for $m \geq 2$). Then any principal $\mathrm{B}^{m-1}_{nis}G$-bundle $\esc{Y} \to \esc{X}$ in $\mathcal{P}_{nis}(k)$ (in particular, for $m=1$, any $G$-torsor) is an $\mathbb{A}^1$-$n$-covering.
\end{cor}
\begin{proof}
This follows at once from \Cref{torsors are covering} and \Cref{n-strict groups}.
\end{proof}

\begin{cor}\label[cor]{Lmot covering is covering of Lmot}
Suppose $\esc{X}$ is a Nisnevich sheaf of spaces over an arbitrary base scheme $S$. Pullback along the localization unit $\eta_{mot} \colon \mathrm{id}\to L_{mot}$ induces equivalences of $\infty$-categories:
\begin{enumerate}
    \item\hspace{2.5cm}$\eta^*_{mot} \colon \mathcal{H}^{mot}(S)_{/L_{mot}\esc{X}} = \mathcal{F}ib^{\mathbb{A}^1}_{L_{mot}\esc{X}} \leftrightarrows \mathcal{F}ib^{\mathbb{A}^1}_{\esc{X}} \colon L_{mot}$.
    \item For every $n \geq -2$,
    \[
        \eta^*_{mot} \colon Cov^n_{\mathbb{A}^1}(L_{mot}\esc{X})=Cov^n_{\mathcal{H}^{mot}}(L_{mot}\esc{X})\leftrightarrows Cov^n_{\mathbb{A}^1}(\esc{X}) \colon L_{mot}
    \]
    (and $\eta^*_{mot} \colon Cov^n_{\mathcal{H}^{mot}}(L_{mot}\esc{X})_* \leftrightarrows Cov^n_{\mathbb{A}^1}(\esc{X})_* \colon L_{mot}$ when $\esc{X}$ is pointed).
\end{enumerate}
\end{cor}
\begin{proof}
Since $L_{mot}$ is locally cartesian [\cite{hoyois2017six}, Proposition 3.15] and effective (\Cref{Effective localization remark}), the result follows directly from \Cref{L covering is covering of L}.
\end{proof}

\begin{defn}\label[defn]{defn motivic fundamental groupoid}
For an $S$-space $\esc{X}$ and $m\geq 1$, we call $\tau_{\leq m}^{nis}L_{mot}\esc{X}$ and $\tau_{\leq m}^{mot}L_{mot}\esc{X}$ the \emph{$\mathbb{A}^1$-fundamental $m$-groupoid} (recovering Morel's $\mathbb{A}^1$-fundamental groupoid when $m=1$) and the \emph{motivic fundamental $m$-groupoid} of $\esc{X}$, respectively. We say that $L_{mot}$ is \emph{$m$-fundamental on $\esc{X}$} if the canonical map $\tau_{\leq m}^{nis}L_{mot}\esc{X}\to \tau_{\leq m}^{mot}L_{mot}\esc{X}$ is an equivalence (\Cref{defn m fundamental}).
\end{defn}

\begin{rem}\label[rem]{mot fundamental on connected}
Recall that when $k$ is a perfect field and $(\esc{X},x) \in \mathcal{P}_{nis}(k)_*$ is an $\mathbb{A}^1$-connected (equivalently, motivically connected) $k$-space, \Cref{layers of connected motivic space} implies that $L_{mot}$ is $n$-fundamental on $\esc{X}$ for every $n\geq 1$: the canonical map $\tau_{\leq n}^{nis}L_{mot}\esc{X} \xrightarrow{\sim} \tau_{\leq n}^{mot}L_{mot}\esc{X}$ is an equivalence, and hence there are equivalences
\[
       \Omega\tau_{\leq n}^{mot}(L_{mot}\esc{X},x)\simeq\Omega\tau_{\leq n}^{nis}(L_{mot}\esc{X},x)\simeq  \tau_{\leq n-1}^{nis}\Omega(L_{mot}\esc{X},x) \simeq\tau_{\leq n-1}^{mot}\Omega(L_{mot}\esc{X},x)
\]
(inducing $\pi_n^{\mathbb{A}^1}(\esc{X},x) \xrightarrow{\sim} \pi_n^{mot}(L_{mot}\esc{X},x)$).
\end{rem}

\begin{prop}[{Cf.\ [\cite{morel2012a1}, Remark 7.10] for $n=0$}]
Let $k$ be a perfect field and let $n \geq 0$. Suppose $\esc{Y}\in \mathcal{P}_{nis}(k)$ is $\mathbb{A}^1$-connected and $y \colon *\to \esc{Y}$ is a basepoint. Then pullback along $y$ induces fiber functors on both internal and external $n$-coverings that canonically factor through equivalences of $\infty$-categories with the category $Act_{\tau_{\leq n}^{mot}\Omega(L_{mot}\esc{Y},y)}\bigl(\mathcal{H}^{mot}_{\leq n}(k)\bigr)$ of motivically $n$-truncated spaces equipped with a coherent action of $\tau_{\leq n}^{mot}\Omega(L_{mot}\esc{Y},y) \simeq \tau_{\leq n}^{nis}\Omega(L_{mot}\esc{Y},y)$:
\[
Cov^n_{\mathcal{H}^{mot}(k)}(L_{mot}\esc{Y}) \xrightarrow[\sim]{\eta_{mot}^*} Cov^n_{\mathbb{A}^1}(\esc{Y}) \xrightarrow[\sim]{\Gamma_y^{\mathbb{A}^1,n}} Act_{\tau_{\leq n}^{mot}\Omega(L_{mot}\esc{Y},y)}\bigl(\mathcal{H}^{mot}_{\leq n}(k)\bigr)
\]
(which for $n=0$ gives $Cov_{\mathcal{H}^{mot}(k)}(L_{mot}\esc{Y}) \simeq Cov_{\mathbb{A}^1}(\esc{Y}) \xrightarrow{\sim} Act_{\pi_1^{\mathbb{A}^1}(\esc{Y},y)}\bigl(Sh_{nis}^{\mathbb{A}^1}(k)\bigr)$).
\end{prop}
\begin{proof}
Since $L_{mot}$ is locally cartesian and effective and $\mathcal{H}^{mot}(k)$ is closed under extensions over connected objects [\cite{morel2012a1}, Lemma 6.51], this follows directly from \Cref{action of the fundamental groupoid}, \Cref{L galois theory}, and \Cref{mot fundamental on connected}.
\end{proof}

\begin{defn}
For $n\geq 0$ and a pointed $S$-space $\esc{X}$, an initial object $\widetilde{\esc{X}}_n^{mot}\to \esc{X}$ of $Cov^n_{\mathcal{H}^{mot}}(\esc{X})_*$ (when $\esc{X}\in \mathcal{H}^{mot}(S)_*$), respectively $\widetilde{\esc{X}}_n^{\mathbb{A}^1}\to \esc{X}$ of $Cov^n_{\mathbb{A}^1}(\esc{X})_*$, is called the \emph{universal motivic $n$-covering}, respectively the \emph{universal $\mathbb{A}^1$-$n$-covering}, of $\esc{X}$, and is denoted $\widetilde{\esc{X}}^{mot}$, respectively $\widetilde{\esc{X}}^{\mathbb{A}^1}$, when $n=0$.
\end{defn}

\begin{thm}\label{motivically simply connected covering space}
Let $\esc{X}$ be a pointed and motivically connected $k$-motivic space over a perfect field $k$, and let $n \geq 0$.
\begin{enumerate}
    \item The $\infty$-category $Cov^n_{\mathcal{H}^{mot}(k)}(\esc{X})_*$ has an initial object $\widetilde{\esc{X}}_n^{mot}\to \esc{X}$ (written $\widetilde{\esc{X}}^{mot}$ when $n=0$).
    \item This initial object is the unique pointed motivic $n$-covering with a motivically $(n+1)$-connected source, i.e., satisfying $\tau_{\leq n+1}^{mot}\widetilde{\esc{X}}_n^{mot}\simeq *$. Moreover, the inclusion
    \[
    Cov^n_{\mathcal{H}^{mot}(k)}(\esc{X})_*\subset Cov^n_{nis}(\esc{X})_*
    \]
    preserves this initial object, i.e., $\widetilde{\esc{X}}_n^{mot}\to \esc{X}$ is the unique pointed Nisnevich $n$-covering of $\esc{X}$ with a Nisnevich-$(n+1)$-connected source.
    \item The canonical map $\Omega\tau_{\leq n+1}^{mot}\esc{X}\to \underline{\mathrm{End}}_{\mathcal{H}^{mot}(k)_{/\esc{X}}}(\widetilde{\esc{X}}_n^{mot})$ is an equivalence, inducing equivalences
    \begin{align*}
    \tau_{\leq n}^{mot}\Omega\esc{X} \simeq \Omega\tau_{\leq n+1}^{mot}\esc{X}
    &\simeq \underline{\mathrm{Aut}}_{\mathcal{H}^{mot}(k)_{/\esc{X}}}(\widetilde{\esc{X}}_n^{mot}) \simeq \underline{\mathrm{End}}_{\mathcal{H}^{mot}(k)_{/\esc{X}}}(\widetilde{\esc{X}}_n^{mot}) \\
    &\simeq \underline{\mathrm{Aut}}_{\mathcal{P}_{nis}(k)_{/\esc{X}}}(\widetilde{\esc{X}}_n^{mot}) \simeq \underline{\mathrm{End}}_{\mathcal{P}_{nis}(k)_{/\esc{X}}}(\widetilde{\esc{X}}_n^{mot})
    \end{align*}
    (which, when $\esc{X}$ is motivically $n$-connected, are further equivalent to $\mathrm{B}^n_{mot}\pi_{n+1}^{mot}(\esc{X}) \simeq \mathrm{B}^n_{nis}\pi_{n+1}^{mot}(\esc{X})$, and for $n=0$ to $\pi_1^{mot}(\esc{X})$).
\end{enumerate}
\end{thm}
\begin{proof}
By \Cref{motivic 0 truncation is injective} and \Cref{layers of connected motivic space}, $L_{mot}$ is path-injective and $m$-fundamental on connected objects for all $m \geq 1$. Statements (1) and (2) therefore follow directly from \Cref{Universal cover for fundamental localization} (with $\widetilde{\esc{X}}_n^{mot} \simeq \tau_{>n+1}^{nis}\esc{X}$), while statement (3) follows from \Cref{aut=pi_1} and \Cref{motivic em}.
\end{proof}
\begin{prop}
Let $k$ be a perfect field, let $n \geq 0$, and suppose $\esc{X}$ is a pointed, Nisnevich-locally connected $k$-space such that $\tau_{\leq n+1}^{nis}\esc{X}$ is $\mathbb{A}^1$-local (in particular, if $\esc{X}$ is Nisnevich-locally $n$-connected and $\pi_{n+1}^{nis}\esc{X} \in (n+1)\mbox{-}\mathrm{Grp}^{L_{mot}}$). Then $$L_{mot}:Cov_{nis}^n(\esc{X})\to Cov_{\mathcal{H}^{mot}(k)}^n(L_{mot}\esc{X})$$ preserves the initial object, i.e., $\widetilde{\esc{X}}_n^{mot} \simeq L_{mot}\tau_{>n+1}^{nis}\esc{X}$.
\end{prop}
\begin{proof}
By \Cref{nis trunc preserve mot equiv} $\tau_{\leq n+1}^{nis}\esc{X}\to \tau_{\leq n+1}^{nis}L_{mot}\esc{X}$ is a $L^{mot}_{\leq n+1}$-equivalence. Since $\tau_{\leq n+1}^{nis}\esc{X}$ is $\mathbb{A}^1$-local, we find that $$\tau_{\leq n+1}^{nis}\esc{X}\simeq L_{\leq n+1}^{mot} \tau_{\leq n+1}^{nis}L_{mot}\esc{X}\simeq \tau_{\leq n+1}^{mot}L_{mot}\esc{X}.$$ The claim now follows from [\cite{morel2012a1}, Theorem 6.53] (cf.\ Theorems 6.53 and 6.59 in \textit{loc.\ cit.}).
\end{proof}

\begin{thm}[{Compare with [\cite{morel2012a1}, Theorem 7.8] for $n=0$}]\label{a1 simply connected covering space}
Let $k$ be a perfect field, let $n \geq 0$, and let $\esc{X}$ be a pointed, $\mathbb{A}^1$-connected, Nisnevich-local $k$-space.
\begin{enumerate}
    \item The $\infty$-category $Cov^n_{\mathbb{A}^1}(\esc{X})_*$ has an initial object $\widetilde{\esc{X}}_n^{\mathbb{A}^1} \to \esc{X}$ (written $\widetilde{\esc{X}}^{\mathbb{A}^1}$ when $n=0$).
    \item This initial object is the unique pointed $\mathbb{A}^1$-$n$-covering with an $\mathbb{A}^1$-$(n+1)$-connected source.
    \item The canonical map $\Omega\tau_{\leq n+1}^{nis}L_{mot}\esc{X}\to \underline{\mathrm{End}}_{\mathcal{P}_{nis}(k)_{/\esc{X}}}(\widetilde{\esc{X}}_n^{\mathbb{A}^1})$ is an equivalence, inducing equivalences
    \[
    \tau_{\leq n}^{nis}\Omega L_{mot}\esc{X} \simeq \Omega\tau_{\leq n+1}^{nis}L_{mot}\esc{X} \xrightarrow{\sim} \underline{\mathrm{Aut}}_{\mathcal{P}_{nis}(k)_{/\esc{X}}}(\widetilde{\esc{X}}_n^{\mathbb{A}^1}) \xrightarrow{\sim} \underline{\mathrm{End}}_{\mathcal{P}_{nis}(k)_{/\esc{X}}}(\widetilde{\esc{X}}_n^{\mathbb{A}^1})
    \]
    (which, when $\esc{X}$ is $\mathbb{A}^1$-$n$-connected, are further equivalent to $\mathrm{B}^n_{nis}\pi_{n+1}^{\mathbb{A}^1}(\esc{X})$, and for $n=0$ to $\pi_1^{\mathbb{A}^1}(\esc{X})$).
\end{enumerate}
\end{thm}
\begin{proof}
Since $L_{mot}$ is cartesian, locally cartesian, effective, and $(n+1)$-fundamental on $\esc{X}$ (\Cref{mot fundamental on connected}), statements (1) and (2) follow immediately from \Cref{Universal cover for fundamental localization}(2) (with $\widetilde{\esc{X}}_n^{\mathbb{A}^1} \simeq \eta_{mot}^*\widetilde{L_{mot}\esc{X}}_n^{mot}$), while statement (3) follows from \Cref{aut=pi_1} and \Cref{mot fundamental on connected}.
\end{proof}

However, it is not true in general that the ordinary universal Nisnevich-$(n+1)$-connected $n$-covering of a space maps to its universal $\mathbb{A}^1$-$(n+1)$-connected $n$-covering under $L_{mot}$. This holds for Nisnevich-locally connected spaces precisely when $\tau^{nis}_{\leq n+1}\esc{X}$ is already $\mathbb{A}^1$-local (for instance, when $\esc{X}$ is Nisnevich-locally $n$-connected and $\pi_{n+1}^{nis}\esc{X}$ lies in $(n+1)\mbox{-}\mathrm{Grp}^{L_{mot}}$):

\begin{thm}[Motivic Galois correspondence]\label{motivic galois corr}
Over a perfect field $k$, the motivic homotopy category satisfies the $n$-th Galois correspondence for every $n \geq 0$: given a pointed $\mathbb{A}^1$-$n$-connected space $\esc{X}$, there is a commutative square of equivalences of posets
\[
\xymatrix{
Cov^n_{\mathbb{A}^1}(\esc{X})_*^{>_{\mathbb{A}^1}n}
\ar[r]^-{\simeq} \ar[d]_{\simeq}
&
Cov^n_{\mathcal{H}^{mot}(k)}(L_{mot}\esc{X})_*^{>_{mot}n}
\ar[d]^-{\simeq}
\\
Sub_{(n+1)\mbox{-}\mathrm{Grp}^{L_{mot}}}(\pi_{n+1}^{\mathbb{A}^1}\esc{X})
\ar[r]^-{\simeq}
&
Sub_{(n+1)\mbox{-}\mathrm{Grp}^{L_{mot}}}(\pi_{n+1}^{mot}L_{mot}\esc{X})
}
\]
Here, $(n+1)\mbox{-}\mathrm{Grp}^{L_{mot}} = \mathrm{Grp}^{\mathbb{A}^1}_k$ when $n=0$ and $\mathrm{Ab}^{\mathbb{A}^1}_k$ when $n \geq 1$,
\[
Sub_{(n+1)\mbox{-}\mathrm{Grp}^{L_{mot}}}(\pi_{n+1}^{\mathbb{A}^1}\esc{X})\quad \text{(resp.\ } Sub_{(n+1)\mbox{-}\mathrm{Grp}^{L_{mot}}}(\pi_{n+1}^{mot}L_{mot}\esc{X})\text{)}
\]
is the poset of $(n+1)$-strictly $\mathbb{A}^1$-invariant subgroup sheaves of $\pi_{n+1}^{\mathbb{A}^1}\esc{X}$ (resp.\ of $\pi_{n+1}^{mot}L_{mot}\esc{X}$), and
\[
Cov^n_{\mathbb{A}^1}(\esc{X})_*^{>_{\mathbb{A}^1} n}\quad \textup{(resp.\ } Cov^n_{\mathcal{H}^{mot}(k)}(L_{mot}\esc{X})_*^{>_{mot}n} = Cov^n_{\mathcal{H}^{mot}(k)}(L_{mot}\esc{X})_*^{>_{nis}n} \textup{)}
\]
is the full subcategory (poset) of pointed $\mathbb{A}^1$-$n$-coverings of $\esc{X}$ with $\mathbb{A}^1$-$n$-connected total spaces (resp.\ of pointed internal motivic $n$-coverings of $L_{mot}\esc{X}$ with motivically, or equivalently Nisnevich-, $n$-connected total spaces).
\end{thm}
\begin{proof}
This follows directly from \Cref{LC local galois corr} and \Cref{n-strict groups}, as $L_{mot}$ over a perfect field satisfies all the required hypotheses (summarized at the end of \S3.4).
\end{proof}

\begin{rem}
For $n=0$, although one can formulate a theorem concerning the ``normal'' part of the $0$-th Galois correspondence (this question was posed to me by Virgile Constantin in private communication), its statement is less satisfactory at present because it is unclear whether group-theoretic normality in the Nisnevich-local setting coincides with categorical normality in $\mathrm{Grp}^{\mathbb{A}^1}_k$ (i.e., whether the quotient of a strongly $\mathbb{A}^1$-invariant group sheaf by a strongly $\mathbb{A}^1$-invariant normal subgroup sheaf is again strongly $\mathbb{A}^1$-invariant). See [\cite{asok2025whitehead}, Remark 3] for a discussion of this issue, and Lemma 2.7 therein for related results. However, for $n \geq 1$ (or for $n=0$ when $\pi_1^{\mathbb{A}^1}(\esc{X})$ is abelian), Morel's celebrated theorem that \textit{strongly $\mathbb{A}^1$-invariant implies strictly $\mathbb{A}^1$-invariant} ensures that $(n+1)\mbox{-}\mathrm{Grp}^{L_{mot}} = \mathrm{Ab}^{\mathbb{A}^1}_k \subset \mathrm{Ab}_k$ is an abelian subcategory closed under quotients over perfect fields; see [\cite{asok2025whitehead}, Proposition 2.8]. Consequently, every subgroup sheaf in $Sub_{\mathrm{Ab}^{\mathbb{A}^1}_k}(\pi_{n+1}^{\mathbb{A}^1}\esc{X})$ is normal with quotient in $\mathrm{Ab}^{\mathbb{A}^1}_k$, so the normal part of the $n$-th Galois correspondence holds unconditionally for all $n \geq 1$ (as well as for $n=0$ when $\pi_1^{\mathbb{A}^1}\esc{X}$ is abelian).
\end{rem}

\begin{rem}\label[rem]{mot n cov}
A comprehensive treatment of $n$-truncated morphisms in an $\infty$-topos has recently appeared in [\cite{constantin2026higher}]. Although the motivic homotopy category is not an $\infty$-topos (\Cref{mot not topos}), the localization framework developed above shows that the internal and external theories of motivic $n$-coverings retain the full structure of \textit{loc.\ cit.} over connected bases, subject only to the normality caveat for $n=0$ noted above.
\end{rem}

\subsubsection{Geometric \texorpdfstring{$n$}{n}-coverings (motivic \texorpdfstring{$n$}{n}-connective morphisms)}

We next specialize the results of \S2.4.2 to the motivic localization $L_{mot} \colon \mathcal{P}_{nis}(S) \to \mathcal{H}^{mot}(S)$. Recall from \Cref{defn eff epi} that a morphism $f \colon \esc{X} \to \esc{Y}$ in $\mathcal{H}^{mot}(S)$ is an \emph{effective epimorphism} (or \emph{motivic $0$-connective morphism}) if its \v{C}ech nerve $\check{C}^{\mathcal{H}^{mot}}_\bullet(f)$ is a colimit diagram in $\mathcal{H}^{mot}(S)$, i.e., $|\check{C}^{\mathcal{H}^{mot}}_\bullet(f)| \xrightarrow{\sim} \esc{Y}$.

\begin{prop}\label[prop]{motivic eff epi basic properties}
Let $S$ be a qcqs scheme.
\begin{enumerate}
    \item The functor $\pi_0^{mot} \colon \mathcal{H}^{mot}(S) \to Sh_{nis}^{\mathbb{A}^1}(S)$ preserves effective epimorphisms.
    \item Effective epimorphisms in $\mathcal{H}^{mot}(S)$ are stable under pullbacks and satisfy right cancellation.
    \item Every morphism $f \colon \esc{X} \to \esc{Y}$ in $\mathcal{H}^{mot}(S)$ admits an image factorization $$\esc{X} \to \mathrm{im}_{\mathcal{H}^{mot}}(f) := |\check{C}^{\mathcal{H}^{mot}}_\bullet(f)| \to \esc{Y}$$ into an effective epimorphism followed by a $(-1)$-truncated morphism (monomorphism), and the canonical comparison map $$\tau_{\leq -1}^{\mathcal{H}^{mot}_{/\esc{Y}}} f \to \mathrm{im}_{\mathcal{H}^{mot}}(f)$$ over $\esc{Y}$ is an equivalence. In particular, $f$ is an effective epimorphism in $\mathcal{H}^{mot}(S)$ if and only if $\tau_{\leq -1}^{\mathcal{H}^{mot}_{/\esc{Y}}} f \simeq \esc{Y}$.
    \item A morphism $f \colon \esc{U} \to \esc{X}$ in $\mathcal{H}^{mot}(S)$ is an effective epimorphism if and only if $$f^* \colon Sub_{\mathcal{H}^{mot}(S)}(\esc{X}) \to Sub_{\mathcal{H}^{mot}(S)}(\esc{U})$$ is injective.
\end{enumerate}
\end{prop}
\begin{proof}
Since $\mathcal{H}^{mot}(S)$ is presentable and has universal colimits (\Cref{motivic giraud's theorem}), these properties follow directly from \Cref{pi0 preserves eff epi,pullback stability eff epi,prop right cancellation eff epi,epi mono factorization in a presentable infty cat,rem eff epi vs -1 trunc,eff epi sub injective}.
\end{proof}
Based on the third statement, we thus say that:
\begin{defn}
    A morphism $h:\esc{X}\to \esc{Y}$ of Nisnevich sheaves is motivically-$n$-connective if $L_{mot}h$ has terminal $(n-1)$-truncation in $\mathcal{H}^{mot}(S)_{/L_{mot}(\esc{Y})}$.
\end{defn}

\begin{cor}\label[cor]{motivic n connective factorization}
Let $S$ be a qcqs scheme and let $n \geq 0$. Every morphism $f \colon \esc{X} \to \esc{Y}$ in $\mathcal{P}_{nis}(S)$ factors as $f = g_n \circ h_n$, where $h_n$ is a motivic $n$-connective morphism and $g_n \in Cov^{n-1}_{\mathbb{A}^1}(\esc{Y})$ is an $\mathbb{A}^1$-$(n-1)$-covering. In particular, for $n=0$, every morphism $f$ of Nisnevich-local spaces factors as $f = g_0 \circ h_0$, where $h_0$ is a motivic effective epimorphism and $g_0 \in Cov^{-1}_{\mathbb{A}^1}(\esc{Y})$ is an $\mathbb{A}^1$-monomorphism.
\end{cor}
\begin{proof}
Since $L_{mot}$ is locally cartesian and effective, this follows immediately from \Cref{L n connective factorization} (together with \Cref{motivic eff epi basic properties}(3) when $n=0$).
\end{proof}
\begin{rem}
    Over a general (disconnected) base $\esc{Y}$, the relative analog of \Cref{mot conn if and only if trunc trivial}, namely whether an $\mathbb{A}^1$-fibration $f \colon \esc{X} \to \esc{Y}$ is motivically-$n$-connective if and only if it is Nisnevich-$n$-connective, fails for $n \geq 1$. For example, when $n=1$ and $\esc{X} \in \mathcal{H}^{mot}(k)$ given by Ayoub's counterexample, $\esc{X} \to \pi_0^{mot}\esc{X}$ is motivically-$1$-connective, whereas its Nisnevich $0$-truncation $\pi_0^{nis}\esc{X} \to \pi_0^{mot}\esc{X}$ is not an equivalence. Nevertheless, by \Cref{Lmot covering is covering of Lmot}, Nisnevich-$n$-connective $\mathbb{A}^1$-fibrations are always motivically-$n$-connective. Moreover, the converse holds for $n=0$ over any base (as the following proposition shows), and for all $n \geq 1$ over a perfect field whenever $\esc{Y}$ is $\mathbb{A}^1$-connected (by passing to homotopy fibers and applying \Cref{mot conn if and only if trunc trivial}).
\end{rem}
\begin{prop}\label[prop]{Lmot preserves and detects eff epi}
Let $S$ be a qcqs scheme. The localization $L_{mot}$ preserves effective epimorphisms (both as a functor to $\mathcal{H}^{mot}(S)$ and as an endofunctor of $\mathcal{P}_{nis}(S)$). Moreover, $L_{mot}$ detects Nisnevich effective epimorphisms over motivic spaces $\esc{Y} \in \mathcal{H}^{mot}(S)$. In particular, an $S$-space $\esc{X} \in \mathcal{P}_{nis}(S)$ is $\mathbb{A}^1$-$0$-connective if and only if it is Nisnevich-$0$-connective.
\end{prop}
\begin{proof}
Because $L_{mot}$ is an effective localization (\Cref{Effective localization remark}), the claims follow from \Cref{L preserves effective epi,eff loc preserves detects eff epi} (compare also \Cref{-1-2 mot trunc}).
\end{proof}
\begin{defn}\label[defn]{defn motivic geom cov}
Specializing \Cref{defn geom cov} to $L_{mot}$, a morphism $f \colon \esc{Y} \to \esc{X}$ in $\mathcal{P}_{nis}(S)$ is:
\begin{enumerate}
    \item an \emph{$\mathbb{A}^1$-effective epimorphism} (or \emph{external $L_{mot}$-effective epimorphism}) if $$L_{mot}\check{C}^{nis}_\bullet(f) \to L_{mot}\esc{X}$$ is a colimit diagram in $\mathcal{H}^{mot}(S)$ (equivalently, the Nisnevich image map $\mathrm{im}_{nis}(f) \to \esc{X}$ is a motivic equivalence);
    \item a \emph{motivic effective epimorphism} (or \emph{motivic $0$-connective morphism}, or \emph{internal $L_{mot}$-effective epimorphism}) if $L_{mot}f$ is an effective epimorphism in $\mathcal{H}^{mot}(S)$.
\end{enumerate}
A collection of morphisms $\{p_i \colon \esc{U}_i \to \esc{X}\}_{i \in I}$ in $\mathcal{P}_{nis}(S)$ is an \emph{external} (resp.\ \emph{internal}) \emph{geometric $\mathbb{A}^1$-covering} if $p := \bigsqcup_{i \in I} p_i \colon \bigsqcup_{i \in I} \esc{U}_i \to \esc{X}$ is an external (resp.\ internal) $L_{mot}$-effective epimorphism.
\end{defn}

\begin{lem}\label[lem]{motivic comparison of geom cov}
Let $S$ be a qcqs scheme.
\begin{enumerate}
    \item Every Nisnevich covering (geometric $\mathcal{P}_{nis}(S)$-covering) is an external geometric $\mathbb{A}^1$-covering, and every external geometric $\mathbb{A}^1$-covering is an internal geometric $\mathbb{A}^1$-covering.
    \item A collection of morphisms in $\mathcal{H}^{mot}(S)$ is a geometric $\mathcal{H}^{mot}(S)$-covering if and only if, viewed in $\mathcal{P}_{nis}(S)$, it is an external geometric $\mathbb{A}^1$-covering, if and only if it is an internal geometric $\mathbb{A}^1$-covering.
\end{enumerate}
\end{lem}
\begin{proof}
This is \Cref{comparison of various geom cov} applied to $L_{mot} \colon \mathcal{P}_{nis}(S) \to \mathcal{H}^{mot}(S)$, using \Cref{motivic giraud's theorem}(1,2).
\end{proof}

\begin{lem}[Local-to-global $\mathbb{A}^1$-connectivity]\label[lem]{motivic local to global conn}
Let $k$ be a perfect field, and let $\{p_i \colon \esc{U}_i \to \esc{X}\}_{i \in I}$ be an internal geometric $\mathbb{A}^1$-covering in $\mathcal{P}_{nis}(k)$ (for example, an external geometric $\mathbb{A}^1$-covering, in particular a Nisnevich covering) such that each $\esc{U}_i$ is $\mathbb{A}^1$-connected and there exists $i_0 \in I$ such that $\esc{U}_{i_0} \times_{\esc{X}} \esc{U}_i$ is $\mathbb{A}^1$-$0$-connective (equivalently, Nisnevich-$(-1)$-connected by \textup{\Cref{connected vs connective mot}}; for instance, pointed) for all $i \in I$. Then $\esc{X}$ is $\mathbb{A}^1$-connected.
\end{lem}
\begin{proof}
Since $L_{mot}$ is path-injective over any perfect field (\Cref{motivic 0 truncation is injective}) and $\mathcal{H}^{mot}(k)$ has universal colimits (\Cref{motivic giraud's theorem}(2)), this follows immediately from \Cref{L geom L connectivity} (the examples in parentheses follow from the corresponding parenthetical statements in \Cref{L geom L connectivity}, using \Cref{Lmot preserves and detects eff epi}).
\end{proof}

\begin{cor}[{Higher $\mathbb{A}^1$-Van Kampen theorem, cf.\ [\cite{morel2012a1}, Theorem 7.12]}]\label[cor]{gen Van Kampen}
Let $k$ be a perfect field, let $n \geq 0$, and let $\{\esc{U}_i \to \esc{X}\}_{i \in I}$ be a pointed external geometric $\mathbb{A}^1$-covering of a $k$-space $\esc{X}$ (in particular, a Zariski covering $\{U_i \to X\}_{i \in I}$ of a smooth $k$-scheme $X$) such that each $m$-fold intersection $$\esc{U}_{i_1\dots i_m} := \esc{U}_{i_1} \times_{\esc{X}} \cdots \times_{\esc{X}} \esc{U}_{i_m}\quad \text{for}\quad 1 \leq m \leq n+2$$ is pointed by the basepoints of $\esc{U}_i$'s (through the universal property of pullbacks) and $\mathbb{A}^1$-$(n-1)$-connected for every $i_1,\dots,i_m \in I$. Then:
\begin{enumerate}
    \item $\esc{X}$ is $\mathbb{A}^1$-$(n-1)$-connected.
    \item The $n$-th $\mathbb{A}^1$-homotopy group sheaf of $\esc{X}$ is given by the following coequalizer in $n\mbox{-}\mathrm{Grp}^{L_{mot}}_k$ (namely $Sh_{k,*}^{\mathbb{A}^1}$ when $n=0$, $\mathrm{Grp}^{\mathbb{A}^1}_k$ when $n=1$, and $\mathrm{Ab}^{\mathbb{A}^1}_k$ when $n \geq 2$):
\[
    \coprod_{i,j} \pi_n^{\mathbb{A}^1}(\esc{U}_i \times_{\esc{X}} \esc{U}_j) \rightrightarrows \coprod_{i} \pi_n^{\mathbb{A}^1}(\esc{U}_i) \to \pi_n^{\mathbb{A}^1}(\esc{X}) \to *
\]
(with $\pi_0^{\mathbb{A}^1}$ replaced by $\pi_0^{mot}L_{mot} \simeq L_{\leq 0}^{mot}$ when $n=0$), where $\coprod$ denotes the wedge sum $\bigvee_i$ in $Sh_{k,*}^{\mathbb{A}^1}$ when $n=0$, the free product $\underset{i}{*}$ in $\mathrm{Grp}^{\mathbb{A}^1}_k$ when $n=1$, and the direct sum $\bigoplus_i$ in $\mathrm{Ab}^{\mathbb{A}^1}_k$ when $n \geq 2$.
\end{enumerate}
\end{cor}
\begin{proof}
By \Cref{motivic 0 truncation is injective} and \Cref{layers of connected motivic space} $L_{mot}$ over a perfect field is path-injective and $m$-fundamental on $(m-1)$-connected objects for all $m \geq 1$. The result therefore follows directly from \Cref{L van campen}, using the identifications $0\mbox{-}\mathrm{Grp}^{L_{mot}} = Sh_{k,*}^{\mathbb{A}^1}$, $1\mbox{-}\mathrm{Grp}^{L_{mot}} = \mathrm{Grp}^{\mathbb{A}^1}_k$, $n\mbox{-}\mathrm{Grp}^{L_{mot}} = \mathrm{Ab}^{\mathbb{A}^1}_k$ for $n \geq 2$ (\Cref{n-strict groups}), and $\pi_n^{mot}L_{mot} \simeq \pi_n^{\mathbb{A}^1}$ for $n \geq 1$ (\Cref{pi_i mot=nis}).
\end{proof}

\begin{rem}
A Van Kampen theorem internal to the motivic homotopy category follows immediately as in the remark following \Cref{L van campen}. In detail, for $n \geq 0$ and a pointed geometric $\mathcal{H}^{mot}$-covering $\{\esc{U}_i \to \esc{X}\}_{i \in I}$ in $\mathcal{H}^{mot}(k)$ over a perfect field $k$ whose $m$-fold motivic fiber products $\esc{U}_{i_1} \times_{\esc{X}} \cdots \times_{\esc{X}} \esc{U}_{i_m}$ for $1 \leq m \leq n+2$ are motivically $(n-1)$-connected, there is a coequalizer diagram in $n\mbox{-}\mathrm{Grp}^{L_{mot}}$ of internal motivic homotopy groups:
\[
    \coprod_{i,j} \pi_n^{mot}(\esc{U}_i \times_{\esc{X}}\esc{U}_j) \rightrightarrows \coprod_{i} \pi_n^{mot}(\esc{U}_i) \to \pi_n^{mot}(\esc{X}) \to *.
\]
More generally, for a pointed \textit{internal} geometric $\mathbb{A}^1$-covering $\{\esc{U}_i \to \esc{X}\}_{i \in I}$ in $\mathcal{P}_{nis}(k)$ whose motivic fiber products $L_{mot}\esc{U}_{i_1} \times_{L_{mot}\esc{X}} \cdots \times_{L_{mot}\esc{X}} L_{mot}\esc{U}_{i_m}$ are motivically $(n-1)$-connected for $1 \leq m \leq n+2$, this yields a coequalizer diagram in $n\mbox{-}\mathrm{Grp}^{L_{mot}}$:
\[
    \coprod_{i,j} \pi_n^{mot}(L_{mot}\esc{U}_i \times_{L_{mot}\esc{X}} L_{mot}\esc{U}_j) \rightrightarrows \coprod_{i} \pi_n^{mot}(L_{mot}\esc{U}_i) \to \pi_n^{mot}(L_{mot}\esc{X}) \to *
\]
(where $\pi_n^{mot}L_{mot} \simeq \pi_n^{\mathbb{A}^1}$ when $n \geq 1$). Because $L_{mot}$ does not preserve general fiber products, however, \Cref{gen Van Kampen} cannot be deduced directly from this internal version unless the $\esc{U}_i$ and $\esc{X}$ are already $\mathbb{A}^1$-local (\Cref{motivic comparison of geom cov}(2)).
\end{rem}

\begin{rem}
\begin{enumerate}[nosep,itemsep=0.5em,wide=0pt,label=(\arabic*),labelsep=1em]
    \item While Morel establishes an $\mathbb{A}^1$-local Van Kampen theorem specifically for Zariski covers over perfect fields [\cite{morel2012a1}, Theorem 7.12], the preceding result applies to general Nisnevich (and geometric $\mathbb{A}^1$-) covers. Furthermore, our statement includes a connectivity assumption on the triple intersections to resolve a technical gap in \textit{loc.\ cit.} For a Zariski cover $\mathcal{U} = \{U_i \to X\}_{i \in I}$, the argument in \textit{loc.\ cit.} relies on the exactness of the following diagram of sets:
\[
    \mathrm{Hom}_{\mathrm{Grp}^{\mathbb{A}^1}_k}(\pi_1^{\mathbb{A}^1}(|\check{C}(\mathcal{U})|), G) \to \prod_i \mathrm{Hom}_{\mathrm{Grp}^{\mathbb{A}^1}_k}(\pi_1^{\mathbb{A}^1}(U_i), G) \rightrightarrows \prod_{i,j} \mathrm{Hom}_{\mathrm{Grp}^{\mathbb{A}^1}_k}(\pi_1^{\mathbb{A}^1}(U_i \mcap U_j), G).
\]
Technically, this diagram is obtained by applying $\pi_0$ to the totalization of the cosimplicial diagram of mapping spaces:
\[
    \prod_i \mathrm{Map}(U_i,\mathrm{B}_{nis} G) \rightrightarrows \prod_{i,j} \mathrm{Map}(U_i \mcap U_j,\mathrm{B}_{nis} G) \underset{\rightarrow}{\rightrightarrows} \prod_{i,j,\ell} \mathrm{Map}(U_i \mcap U_j \mcap U_\ell,\mathrm{B}_{nis} G)\cdots
\]
Since $\pi_0$ does not commute with totalizations (homotopy limits) in general, exactness at the middle term is not automatic: a collection of maps $U_i \to \mathrm{B}_{nis}G$ whose restrictions to $U_i \mcap U_j$ are homotopic descends to $|\check{C}(\mathcal{U})|$ only if the chosen homotopies satisfy a $2$-cocycle coherence condition on the triple intersections $U_i \mcap U_j \mcap U_\ell$. Assuming the triple intersections are $\mathbb{A}^1$-connected ensures that the mapping spaces $\mathrm{Map}_*(U_i \mcap U_j \mcap U_\ell, \mathrm{B}_{nis}G)$ are discrete (since $\pi_0(\mathrm{Map}_*(Y, \mathrm{B}_{nis}G)) \simeq \mathrm{Hom}(\pi_1^{\mathbb{A}^1}(Y), G)$ requires $Y$ to be connected to avoid basepoint issues across components), trivializing the higher coherence obstructions.

    \item The other thing worth mentioning is that, as we have done above, and even in \textit{loc.\ cit.}, for smooth schemes one states the case of Zariski open covers rather than arbitrary Nisnevich covers because the self-fiber products $U_i \times_X U_i$ in a Nisnevich cover are always disconnected: when $X \in Sm_k$, $\mathbb{A}^1$-connectivity of $U_i \times_X U_i$ forces the clopen diagonal $U_i \hookrightarrow U_i \times_X U_i$ to be an isomorphism.
    \item Finally, Morel mentions only the $n=1$ case of the Van Kampen theorem. While our method of generalizing it to include all integers is new, it is not a surprise. Indeed, for all $n\geq 2$ one can always use the motivic Hurewicz theorem, thanks to the stable connectivity theorem of Morel, to achieve this exact sequence through the Mayer-Vietoris sequence of $\mathbb{A}^1$-homology.
\end{enumerate}
\end{rem}

\begin{cor}\label[cor]{a1 geom conn}
Over a perfect field $k$, a scheme $X/k$ is $\mathbb{A}^1$-connected if (and only if) there is a Nisnevich covering $\{U_i \to X\}_{i \in I}$ such that each $U_i$ is $\mathbb{A}^1$-connected and for every pair $i,j \in I$ the fiber product $U_i \times_X U_j$ has a $k$-point. 

More generally, for any $n \geq 1$, a scheme $X/k$ is $\mathbb{A}^1$-$n$-connected if (and only if) there is a pointed Nisnevich covering $\{U_i \to X\}_{i \in I}$ (necessarily a Zariski open covering when $X \in Sm_k$) such that each $U_i$ is $\mathbb{A}^1$-$n$-connected and for every $2 \leq m \leq n+2$ and $i_1,\dots,i_m \in I$, the fiber product $U_{i_1\dots i_m}$ is $\mathbb{A}^1$-$(n-1)$-connected.
\end{cor}
\begin{proof}
Every Nisnevich covering is an external (hence internal) geometric $\mathbb{A}^1$-covering (\Cref{motivic comparison of geom cov}(1)). When $n=0$, having a $k$-point implies $\mathbb{A}^1$-$(-1)$-connectivity (\Cref{Lmot preserves and detects eff epi}), so $\mathbb{A}^1$-connectivity of $X$ follows from \Cref{motivic local to global conn}. When $n \geq 1$, $\mathbb{A}^1$-$n$-connectivity of $X$ follows directly from \Cref{higher L geom L connectivity} (using \Cref{motivic giraud's theorem}(2), \Cref{motivic 0 truncation is injective}, and \Cref{mot fundamental on connected}).

(For the ``only if'' direction, note that any $\mathbb{A}^1$-connected $k$-scheme automatically admits a $k$-point (as $\mathrm{Spec}(k)$ is a Nisnevich point and $X \to L_{mot}X$ induces an epimorphism $\pi_0^{nis}X \twoheadrightarrow \pi_0^{\mathbb{A}^1}X \simeq *$ by \Cref{mot truncation cart}(2)); hence the trivial covering $\{1_X \colon X \to X\}$ establishes the ``only if'' direction in both cases (with $X \times_X \cdots \times_X X = X$).)
\end{proof}

\medskip

\begin{con-rem}[Summary of properties of the motivic localization functor $L_{mot}$]\label[con-rem]{properties of Lmot}
This concludes our analysis of the internal homotopy theory of the motivic homotopy category. For the reader's convenience, we record below the structural properties of the motivic localization functor $L_{mot} \colon \mathcal{P}_{nis}(S) \to \mathcal{H}^{mot}(S)$ that made these results possible (where properties (1)-(4) hold over an arbitrary qcqs scheme $S$, and (6)-(8) hold over the spectrum of a perfect field, denoted $k$):
\begin{enumerate}
    \item $L_{mot}$ is an accessible localization (\Cref{motivic giraud's theorem}(1)).
    \item It is cartesian ([\cite{hoyois2017six}, Proposition 3.15]) and path-cartesian (\Cref{path cart}; in fact, $L_{\leq n}^{mot}$ preserves finite products for all $n \geq -2$ by \Cref{mot truncation cart}(1)).
    \item It is locally cartesian ([\cite{hoyois2017six}, Proposition 3.15]).
    \item It is an effective localization (\Cref{Effective localization remark}, [\cite{morel19991}, Corollary 2.3.22]; as is $L_{\leq n}^{mot}$ for all $n \geq -1$ by \Cref{mot truncation cart}(2)).
    \item $L_{mot}: \mathcal{P}_{nis}(S)\to  \mathcal{P}_{nis}(S)$ preserves filtered colimits and small coproducts when $S = \mathrm{Spec}(k)$ for a field $k$ (\Cref{motivic giraud's theorem}(3)).
    \item It is $m$-fundamental (\Cref{defn m fundamental}, \Cref{defn motivic fundamental groupoid}) on connected objects for all $m \geq 1$ (\Cref{layers of connected motivic space}, [\cite{morel2012a1}, Corollary 6.2]).
    \item It is path-injective (\Cref{path injectivity}, \Cref{motivic 0 truncation is injective}); in fact, due to fundamentality as mentioned above, it identifies internal-$n$-connectivity with external-$n$-connectivity (\Cref{mot conn if and only if trunc trivial}) for all $n\geq 0$.
    \item It is closed under extensions over connected objects (\Cref{defn extension closed}) ([\cite{morel2012a1}, Lemma 6.51]).
\end{enumerate}
\end{con-rem}

\phantomsection
\bibliographystyle{amsalpha}	
\renewcommand\refname{Bibliography}
\bibliography{references}

@book{lurie2009higher,
  author = {Lurie, Jacob},
  title = {Higher topos theory},
  series = {Annals of Mathematics Studies},
  volume = {170},
  publisher = {Princeton University Press, Princeton, NJ},
  year = {2009},
  pages = {xviii+925},
  isbn = {978-0-691-14049-0; 0-691-14049-9},
  mrclass = {18-02 (18B25 18E35 18G30 18G55 55U40)},
  mrnumber = {2522659},
  mrreviewer = {Mark\ Hovey},
  doi = {10.1515/9781400830558},
  url = {https://doi.org/10.1515/9781400830558},
}

@misc{0bat,
  author = {Maity, Dipankar},
  title = {Birational Algebraic Topology},
  year = {2026},
  note = {Preprint, \href{https://arxiv.org/abs/2606.22887}{arXiv:2606.22887 [math.AG]}},
}

@article{morel19991,
  author = {Morel, Fabien and Voevodsky, Vladimir},
  title = {{${\mathbb A}^1$}-homotopy theory of schemes},
  journal = {Inst. Hautes \'Etudes Sci. Publ. Math.},
  fjournal = {Institut des Hautes \'Etudes Scientifiques. Publications Math\'ematiques},
  number = {90},
  year = {1999},
  pages = {45--143},
  issn = {0073-8301,1618-1913},
  mrclass = {14F35 (19E08)},
  mrnumber = {1813224},
  mrreviewer = {Marc\ Levine},
  url = {http://www.numdam.org/item?id=PMIHES_1999__90__45_0},
}

@article{robalo2012noncommutative,
  author = {Robalo, Marco},
  title = {{$K$}-theory and the bridge from motives to noncommutative motives},
  journal = {Adv. Math.},
  year = {2015},
  fjournal = {Advances in Mathematics},
  volume = {269},
  pages = {399--550},
  issn = {0001-8708,1090-2082},
  mrclass = {14F42 (14A22 18D10 19E15)},
  mrnumber = {3281141},
  mrreviewer = {Jens\ Hornbostel},
  doi = {10.1016/j.aim.2014.10.011},
  url = {https://doi.org/10.1016/j.aim.2014.10.011},
}

@misc{dugger1999sheaves,
  title = {Sheaves and homotopy theory},
  author = {Dugger, Daniel},
  note = {preprint},
  year = {1999},
}

@article{spitzweck2012motivic,
  author = {Spitzweck, Markus and {\O} stv\ae r, Paul Arne},
  title = {Motivic twisted {$K$}-theory},
  journal = {Algebr. Geom. Topol.},
  fjournal = {Algebraic \& Geometric Topology},
  volume = {12},
  year = {2012},
  number = {1},
  pages = {565--599},
  issn = {1472-2747,1472-2739},
  mrclass = {19L50 (14F42 55P43)},
  mrnumber = {2916287},
  mrreviewer = {Niko\ Naumann},
  doi = {10.2140/agt.2012.12.565},
  url = {https://doi.org/10.2140/agt.2012.12.565},
}

@book{morel2012a1,
  author = {Morel, Fabien},
  title = {{$\mathbb A^1$}-algebraic topology over a field},
  series = {Lecture Notes in Mathematics},
  volume = {2052},
  publisher = {Springer, Heidelberg},
  year = {2012},
  pages = {x+259},
  isbn = {978-3-642-29513-3},
  mrclass = {14F35 (14F05)},
  mrnumber = {2934577},
  mrreviewer = {Matthias\ Wendt},
  doi = {10.1007/978-3-642-29514-0},
  url = {https://doi.org/10.1007/978-3-642-29514-0},
}

@article{ayoubcounterexamples,
  author = {Joseph Ayoub},
  title = {Counterexamples to {F.} {Morel{\textquoteright}s} conjecture on {$\protect \pi _0^{\protect \mathbb{A}^1}$}},
  journal = {Comptes Rendus. Math\'ematique},
  pages = {1087--1090},
  year = {2023},
  publisher = {Acad\'emie des sciences, Paris},
  volume = {361},
  doi = {10.5802/crmath.472},
  language = {en},
}

@article{MR4296353,
  author = {Clausen, Dustin and Mathew, Akhil},
  title = {Hyperdescent and \'etale {$K$}-theory},
  journal = {Invent. Math.},
  fjournal = {Inventiones Mathematicae},
  volume = {225},
  year = {2021},
  number = {3},
  pages = {981--1076},
  issn = {0020-9910,1432-1297},
  mrclass = {18F25 (14F20 55N15)},
  mrnumber = {4296353},
  mrreviewer = {Barry\ H.\ Dayton},
  doi = {10.1007/s00222-021-01043-3},
  url = {https://doi.org/10.1007/s00222-021-01043-3},
}

@article{elmanto2021motivic,
  author = {Elmanto, Elden and Hoyois, Marc and Khan, Adeel A. and Sosnilo, Vladimir and Yakerson, Maria},
  title = {Motivic infinite loop spaces},
  journal = {Camb. J. Math.},
  fjournal = {Cambridge Journal of Mathematics},
  volume = {9},
  year = {2021},
  number = {2},
  pages = {431--549},
  issn = {2168-0930,2168-0949},
  mrclass = {14F42 (14C05 19E15 55P47)},
  mrnumber = {4325285},
  mrreviewer = {Ferdinando\ Zanchetta},
  doi = {10.4310/CJM.2021.v9.n2.a3},
  url = {https://doi.org/10.4310/CJM.2021.v9.n2.a3},
}

@misc{nlab:effective_epimorphism_in_ainfinity1category,
  author = {{nLab authors}},
  title = {Effective epimorphism in an {$(\infty,1)$}-category},
  note = {\href{https://ncatlab.org/nlab/revision/effective+epimorphism+in+an+%28infinity%2C1%29-category/29}{effective epimorphism: [Revision 29]}},
  month = jun,
  year = {2026},
}

@misc{p1algtop,
  author = {Maity, Dipankar},
  title = {$\mathbb{P}^1$ Homotopy theory},
  year = {2026},
  note = {Preprint, \href{https://arxiv.org/abs/2609.26260}{arXiv:2609.26260 [math.AG]}},
}

@article{MR3423073,
  author = {Nikolaus, Thomas and Schreiber, Urs and Stevenson, Danny},
  title = {Principal {$\infty$}-bundles: general theory},
  journal = {J. Homotopy Relat. Struct.},
  fjournal = {Journal of Homotopy and Related Structures},
  volume = {10},
  year = {2015},
  number = {4},
  pages = {749--801},
  issn = {2193-8407,1512-2891},
  mrclass = {55R99 (18G60 55U35)},
  mrnumber = {3423073},
  mrreviewer = {Timothy\ Porter},
  doi = {10.1007/s40062-014-0083-6},
  url = {https://doi.org/10.1007/s40062-014-0083-6},
}

@article{hoyois2017six,
  author = {Hoyois, Marc},
  title = {The six operations in equivariant motivic homotopy theory},
  journal = {Adv. Math.},
  fjournal = {Advances in Mathematics},
  volume = {305},
  year = {2017},
  pages = {197--279},
  issn = {0001-8708,1090-2082},
  mrclass = {14F42 (55P91)},
  mrnumber = {3570135},
  mrreviewer = {Jeremiah\ Ben\ Heller},
  doi = {10.1016/j.aim.2016.09.031},
  url = {https://doi.org/10.1016/j.aim.2016.09.031},
}

@article{amit,
  author = {Balwe, Chetan and Hogadi, Amit and Sawant, Anand},
  title = {{$\mathbb{A}^1$}-connected components of schemes},
  journal = {Adv. Math.},
  fjournal = {Advances in Mathematics},
  volume = {282},
  year = {2015},
  pages = {335--361},
  issn = {0001-8708,1090-2082},
  mrclass = {14F05 (14F42)},
  mrnumber = {3374529},
  mrreviewer = {Matthias\ Wendt},
  doi = {10.1016/j.aim.2015.07.003},
  url = {https://doi.org/10.1016/j.aim.2015.07.003},
}

@misc{bachmann2024strongly,
  title = {Strongly {${\mathbb{A}^1}$}-invariant sheaves (after F. Morel)},
  author = {Bachmann, Tom},
  note = {Preprint, \href{https://arxiv.org/abs/2406.11526}{arXiv:2406.11526}},
  year = {2024},
}

@article{choudhury2014connectivity,
  author = {Choudhury, Utsav},
  title = {Connectivity of motivic {$H$}-spaces},
  journal = {Algebr. Geom. Topol.},
  fjournal = {Algebraic \& Geometric Topology},
  volume = {14},
  year = {2014},
  number = {1},
  pages = {37--55},
  issn = {1472-2747,1472-2739},
  mrclass = {14F42 (18E35)},
  mrnumber = {3158752},
  mrreviewer = {Andrei\ D.\ Halanay},
  doi = {10.2140/agt.2014.14.37},
  url = {https://doi.org/10.2140/agt.2014.14.37},
}

@article{10.2140/akt.2022.7.385,
  author = {Chetan Balwe and Bandna Rani and Anand Sawant},
  title = {{Remarks on iterations of the {$\mathbb A^1$}-chain connected components construction}},
  volume = {7},
  journal = {Annals of K-Theory},
  number = {2},
  publisher = {MSP},
  pages = {385 -- 394},
  year = {2022},
  doi = {10.2140/akt.2022.7.385},
  url = {https://doi.org/10.2140/akt.2022.7.385},
}

@article{MR3890956,
  author = {Raptis, Georgios and Strunk, Florian},
  title = {Model topoi and motivic homotopy theory},
  journal = {Doc. Math.},
  fjournal = {Documenta Mathematica},
  volume = {23},
  year = {2018},
  pages = {1757--1797},
  issn = {1431-0635,1431-0643},
  mrclass = {55U35 (14F42 18B25)},
  mrnumber = {3890956},
  mrreviewer = {Matthias\ Wendt},
}

@article{choudhury2022characterisation,
  author = {Choudhury, Utsav and Roy, Biman},
  title = {{$\mathbb{A}^1$}-connected components and characterisation of {$\mathbb A^2$}},
  journal = {J. Reine Angew. Math.},
  fjournal = {Journal f\"ur die Reine und Angewandte Mathematik. [Crelle's Journal]},
  volume = {807},
  year = {2024},
  pages = {55--80},
  issn = {0075-4102,1435-5345},
  mrclass = {14F42},
  mrnumber = {4698492},
  mrreviewer = {Daiki\ Kawabe},
  doi = {10.1515/crelle-2023-0084},
  url = {https://doi.org/10.1515/crelle-2023-0084},
}

@book{MR3076731,
  author = {Weibel, Charles A.},
  title = {The {$K$}-book},
  series = {Graduate Studies in Mathematics},
  volume = {145},
  note = {An introduction to algebraic $K$-theory},
  publisher = {American Mathematical Society, Providence, RI},
  year = {2013},
  pages = {xii+618},
  isbn = {978-0-8218-9132-2},
  mrclass = {19-01},
  mrnumber = {3076731},
  mrreviewer = {L.\ N.\ Vaserstein},
  doi = {10.1090/gsm/145},
  url = {https://doi.org/10.1090/gsm/145},
}

@article{bachmann2017norms,
  author = {Bachmann, Tom and Hoyois, Marc},
  title = {Norms in motivic homotopy theory},
  journal = {Ast\'erisque},
  fjournal = {Ast\'erisque},
  number = {425},
  year = {2021},
  pages = {ix+207},
  issn = {0303-1179,2492-5926},
  isbn = {978-2-85629-939-5},
  mrclass = {14F42 (19E15)},
  mrnumber = {4288071},
  mrreviewer = {Jon\ Eivind\ Vatne},
  doi = {10.24033/ast},
  url = {https://doi.org/10.24033/ast},
}

@misc{sfbat,
  author = {Maity, Dipankar},
  title = {Schematic Functorialities of Birational Motivic Homotopy Categories},
  year = {2026},
  note = {Preprint, \href{https://arxiv.org/abs/2608.04793}{arXiv:2608.04793 [math.AG]}},
}

@misc{stacks-project,
  author = {The {Stacks project}},
  title = {The Stacks project},
  howpublished = {\url{https://stacks.math.columbia.edu}},
  year = {2026},
}

@article{asok2025whitehead,
  author = {Asok, Aravind and Bachmann, Tom and Hopkins, Michael J.},
  title = {On the {W}hitehead theorem for nilpotent motivic spaces},
  journal = {Ann. K-Theory},
  fjournal = {Annals of K-Theory},
  volume = {10},
  year = {2025},
  number = {4},
  pages = {673--706},
  issn = {2379-1683,2379-1691},
  mrclass = {14F42 (20F19 55S45)},
  mrnumber = {4998433},
  doi = {10.2140/akt.2025.10.673},
  url = {https://doi.org/10.2140/akt.2025.10.673},
}

@misc{constantin2026higher,
  title = {Higher covering spaces in an {$\infty $}-topos},
  author = {Constantin, Virgile},
  note = {arXiv preprint arXiv:2607.14773},
  year = {2026},
}

@article{MR3987298,
  author = {Joyce, Dominic},
  title = {Algebraic geometry over {$C^{\infty}$}-rings},
  journal = {Mem. Amer. Math. Soc.},
  fjournal = {Memoirs of the American Mathematical Society},
  volume = {260},
  year = {2019},
  number = {1256},
  pages = {v+139},
  issn = {0065-9266,1947-6221},
  isbn = {978-1-4704-3645-2; 978-1-4704-5336-7},
  mrclass = {14A20 (46E25 58A40)},
  mrnumber = {3987298},
  mrreviewer = {Mattia\ Talpo},
  doi = {10.1090/memo/1256},
  url = {https://doi.org/10.1090/memo/1256},
}

@book{MR1083355,
  author = {Moerdijk, Ieke and Reyes, Gonzalo E.},
  title = {Models for smooth infinitesimal analysis},
  publisher = {Springer-Verlag, New York},
  year = {1991},
  pages = {x+399},
  isbn = {0-387-97489-X},
  mrclass = {18F99 (03H05 18B25 51K10)},
  mrnumber = {1083355},
  mrreviewer = {Kimmo\ I.\ Rosenthal},
  doi = {10.1007/978-1-4757-4143-8},
  url = {https://doi.org/10.1007/978-1-4757-4143-8},
}

@article{MR4404880,
  author = {Balwe, Chetan and Sawant, Anand},
  title = {{$\mathbb A^1$}-connected components of ruled surfaces},
  journal = {Geom. Topol.},
  fjournal = {Geometry \& Topology},
  volume = {26},
  year = {2022},
  number = {1},
  pages = {321--376},
  issn = {1465-3060,1364-0380},
  mrclass = {14F42 (14F06 14J26 55Q05)},
  mrnumber = {4404880},
  mrreviewer = {Matthias\ Wendt},
  doi = {10.2140/gt.2022.26.321},
  url = {https://doi.org/10.2140/gt.2022.26.321},
}

@book{MR1930091,
  author = {Lee, John M.},
  title = {Introduction to smooth manifolds},
  series = {Graduate Texts in Mathematics},
  volume = {218},
  publisher = {Springer-Verlag, New York},
  year = {2003},
  pages = {xviii+628},
  isbn = {0-387-95495-3},
  mrclass = {58-01 (53-01 57-01)},
  mrnumber = {1930091},
  mrreviewer = {J\"urgen\ Eichhorn},
  doi = {10.1007/978-0-387-21752-9},
  url = {https://doi.org/10.1007/978-0-387-21752-9},
}

@book{MR1867354,
  author = {Hatcher, Allen},
  title = {Algebraic topology},
  publisher = {Cambridge University Press, Cambridge},
  year = {2002},
  pages = {xii+544},
  isbn = {0-521-79160-X; 0-521-79540-0},
  mrclass = {55-01 (55-00)},
  mrnumber = {1867354},
  mrreviewer = {Donald\ W.\ Kahn},
}

@article{MR2545,
  author = {Whitehead, J. H. C.},
  title = {On {$C^1$}-complexes},
  journal = {Ann. of Math. (2)},
  fjournal = {Annals of Mathematics. Second Series},
  volume = {41},
  year = {1940},
  pages = {809--824},
  issn = {0003-486X},
  mrclass = {56.0X},
  mrnumber = {2545},
  mrreviewer = {R.\ H.\ Fox},
  doi = {10.2307/1968861},
  url = {https://doi.org/10.2307/1968861},
}

@article {MR3356340,
    AUTHOR = {Raptis, G. and Rosick\'y, J.},
     TITLE = {The accessibility rank of weak equivalences},
   JOURNAL = {Theory Appl. Categ.},
  FJOURNAL = {Theory and Applications of Categories},
    VOLUME = {30},
      YEAR = {2015},
     PAGES = {Paper No. 19, 687--703},
      ISSN = {1201-561X},
   MRCLASS = {18C35 (18G55 55U35)},
  MRNUMBER = {3356340},
MRREVIEWER = {Philippe\ Gaucher},
}

@article{hoyois2015quadratic,
  title = {A quadratic refinement of the Grothendieck--Lefschetz--Verdier trace formula},
  author = {Hoyois, Marc},
  journal = {Algebraic \& Geometric Topology},
  volume = {14},
  number = {6},
  pages = {3603--3658},
  year = {2015},
  publisher = {Mathematical Sciences Publishers},
}
\noindent\rule{\textwidth}{0.4pt} 
\end{document}